\documentclass[12pt, reqno]{amsart}

\usepackage{amsfonts}
\usepackage{amssymb}
\usepackage{amsthm}
\usepackage{amsmath}
\usepackage{a4wide}
\usepackage{mathrsfs}
\usepackage{epsfig}
\usepackage{palatino}
\usepackage{tikz,fp,ifthen}
\usepackage{float}
\usepackage{graphicx,cite}
\usepackage{comment}
\usepackage{accents}
\usepackage{xfrac}
\usepackage{bm}
\usepackage{appendix}
\usepackage[framemethod=tikz]{mdframed}
\mdfsetup{frametitlealignment=\center}
\makeatletter
\newenvironment{ProblemSpecBox}[2]{
	\protected@edef\@currentlabelname{#2}
	\protected@edef\@currentlabel{#2}
	\begin{mdframed}[
		innerlinewidth=0.5pt,
		innerleftmargin=10pt,
		innerrightmargin=10pt,
		innertopmargin = 10pt,
		innerbottommargin=10pt,
		skipabove=\dimexpr\topsep+\ht\strutbox\relax,
		roundcorner=5pt,
		frametitle={#1},
		frametitlerule=true,
		frametitlerulewidth=1pt]
	}{
	\end{mdframed}
}
\makeatother

\usepackage{marginnote}

\usepackage{graphicx}

\usepackage{color}
\definecolor{citegreen}{rgb}{0,0.6,0}
\definecolor{refred}{rgb}{0.8,0,0}
\usepackage[colorlinks, citecolor=citegreen,linkcolor=refred]{hyperref}

\usepackage{mathtools}
\mathtoolsset{showonlyrefs}

\newtheorem{thm}{Theorem}[section]
\newtheorem{lem}[thm]{Lemma}
\newtheorem{prop}[thm]{Proposition}
\newtheorem{cor}[thm]{Corollary}

\theoremstyle{definition}
\newtheorem{defn}[thm]{Definition}

\theoremstyle{remark}
\newtheorem{rem}[thm]{Remark}

\numberwithin{equation}{section}

\def\F{\mathcal F}

\def\cL{\mathcal L}

\def\E{{\sf E}}
\def\R{\mathbb R}

\def\R{{{\mathbb R}}}

\def\NN{\mathbb N}
\def\N{\mathbb N}

\def\cP{\mathcal P}
\def\cM{\mathcal M}
\def\cB{\mathcal B}
\def\cT{\mathcal T}
\def\cK{\mathcal K}
\def\cA{\mathcal A}
\def\fB{\mathfrak B}
\def\bP{{\sf P}}

\newcommand{\W}{\mathcal{W}}
\newcommand{\Norm}[2]{\left\Vert #1 \right\Vert_{#2}}

\usetikzlibrary{backgrounds}
\usetikzlibrary{decorations.pathmorphing,backgrounds,fit,calc,through}
\usetikzlibrary{arrows}
\usetikzlibrary{shapes,decorations,shadows}
\usetikzlibrary{fadings}
\usetikzlibrary{patterns}
\usetikzlibrary{mindmap}
\usetikzlibrary{decorations.text}
\usetikzlibrary{decorations.shapes}

\begin{document}

\title[Optimal control problems driven by nonlinear degenerate FP equations]{Optimal control problems driven by nonlinear degenerate Fokker-Planck equations}

\author[F. Anceschi]{Francesca Anceschi}
\address{Francesca Anceschi\\
Dipartimento di Ingegneria Industriale e Scienze Matematiche, Università Politecnica delle Marche, Via Brecce Bianche 12, 
60131 Ancona, Italy}
\email{f.anceschi@staff.unipvm.it}

\author[G. Ascione]{Giacomo Ascione}
\address{Giacomo Ascione\\
Scuola Superiore Meridionale,
Universit\`a di Napoli, Largo San Marcellino 10, 80138 Napoli,
Italy}
\email{g.ascione@ssmeridionale.it}

\author[D. Castorina]{Daniele Castorina}
\address{Daniele Castorina\\
Dipartimento di Matematica e Applicazioni,
Universit\`a di Napoli, Via Cintia, Monte S. Angelo 80126 Napoli,
Italy}
\email{daniele.castorina@unina.it}

\author[F. Solombrino]{Francesco Solombrino}
\address{Francesco Solombrino\\
Dipartimento di Matematica e Applicazioni,
Universit\`a di Napoli, Via Cintia, Monte S. Angelo 80126 Napoli,
Italy}
\email{francesco.solombrino@unina.it}

\begin{abstract}  
The well-posedness of a class of optimal control problems is analysed, where the state equation couples a nonlinear degenerate Fokker-Planck equation  with a system of Ordinary Differential Equations (ODEs). Such problems naturally arise as mean-field limits of Stochastic Differential
models for multipopulation dynamics, where a large number of agents (followers) is steered through parsimonious intervention on a selected class of leaders.
The proposed approach combines stability estimates for  measure solutions of nonlinear degenerate Fokker-Planck equations with a general framework of assumptions on the cost functional, ensuring compactness and lower semicontinuity properties. The Lie structure of the state equations allows one for considering non-Lipschitz nonlinearities, provided some suitable dissipativity assumptions are considered in addition to non-Euclidean H\"older and sublinearity conditions.
\end{abstract}
\thanks{FA is partially supported by INdAM-GNAMPA Project “Problemi ellittici e sub-ellittici: non linearità, singolarità e crescita critica” CUP E53C23001670001. GA is supported by the GNAMPA-INdAM Project ”Deterministic Control of Stochastic Dynamics”, CUP E53C23001670001 and the PRIN 2022XZSAFN Project ”Anomalous Phenomena on Regular and Irregular Domains: Approximating Complexity for the Applied Sciences” CUP E53D23005970006. DC and FS are supported by the PRIN 2022HKBF5C Project “Variational Analysis of Complex Systems in Material Science, Physics and Biology” CUP E53D23005720006. PRIN projects are part of PNRR Italia Domani, financed by European Union through NextGenerationEU. FS additionally acknowledeges support from the GNAMPA-INdAM Project ”Problemi di controllo ottimo nello spazio di Wasserstein delle misure definite su spazi di Banach” CUP E53C23001670001.}

\setcounter{tocdepth}{1}

\maketitle
\tableofcontents

\section{Introduction} \label{intro}
\subsection{Presentation of the problem}
The aim of this work is to study a class of optimal control problems driven by a dynamic of the type
\begin{equation}\label{probmunl2-intro}
	\begin{cases} \displaystyle \partial_t \mu_t = - v \cdot \nabla_x \mu_t + \sigma \Delta_v \mu_t - \mathrm{div}_v (\mathfrak{v}[t,\bm{\mu}](z) \mu_t) \quad &(t,x,v) \in  (0, T] \times\mathbb{R}^{2d}, \\
		\mu_0 = \bar{\mu}  \quad & (x,v) \in \R^{2d},
	\end{cases}
\end{equation}
where $\sigma \in \R$ and the drift field $\mathfrak{v}$ is allowed to depend explicitly on a solution $\bm{\mu}$ of the problem. Hence,
the above is a nonlinear Fokker-Planck equation driven by a nonlocal drift field $\mathfrak{v}[t,\bm{\mu}]$ depending on the state of the system, 
a typical example being the choice 
\begin{equation}\label{introeq:nonlocaldrift}
	\mathfrak{v}[t,\bm{\mu}](z)=H(t,z)\ast \mu_t\,.
\end{equation}
Furthermore, \eqref{probmunl2-intro} is the natural mean field counterpart of second order multi-agents systems with additive noise, see \eqref{eq:MKVSDEauxdef} below. 
In fact, in this case a drift field of the form $\mathfrak{v}[t,\bm{\mu}]$ naturally arises if one assumes that the natural dynamic is driven by mutual interactions 
(see Appendix \ref{appendix:example}).
These type of mean field systems are in connection with a great number of applications. Indeed, multiagent systems have been used in several contexts, such as in biology to model cell aggregation \cite{camazine2020self,cucker2007emergent,hofbauer1998evolutionary,keller1970initiation}, in chemistry to model chemical networks \cite{lim2020quantitative,mozgunov2018review,oelschlager1989derivation}, or even to describe human interaction in social sciences \cite{during2009boltzmann,toscani2006kinetic,lombardi2020nonverbal,calabrese2021spontaneous}. Furthermore, controlled leader-follower multiagent systems have been used, for instance, in the context of gene regulation in microbial consortia \cite{martinelli2022multicellular,salzano2023vivo}, traffic control \cite{piacentini2018traffic}, swarms control \cite{armbruster2017elastic}, or even oil cleaning through controlled robots \cite{zahugi2013oil}. The controllability of this kind of leader-follower systems, with a fixed finite (sufficiently small) number of followers, has been widely studied, see, for instance, \cite{ko2020asymptotic,lama2024shepherding,auletta2022herding}. On the other hand, the mean-field approach for a large number of followers has been considered for instance in \cite{albi2020mathematical} for crowd control and in \cite{orlando2023mean} for prevention of maritime crime. A general study of these mean-field control problems arising from first-order leader-follower systems with additive noise has been presented in \cite{Ascione20236965}, while the notion of \textit{continuification of controlled systems} has been further explored for instance in \cite{maffettone2022continuification,maffettone2023continuification}. It is clear that this list of applications and references is far from being exhaustive. Furthermore, let us stress that one could be interested in the control of a second-order system, as in \cite{albi2016invisible,albi2017mean}, in which the followers are subject to an additive noise only in the velocity component. Hence, we consider \eqref{probmunl2-intro} as a second-order generalization of the mean-field system obtained in \cite{Ascione20236965}.


Before describing our point of view in detail, it is useful to give a closer look to the notion of solution we need to consider when dealing with 
\eqref{probmunl2-intro}. Indeed, the well-posedness of the class of optimal control problems we consider heavily relies on stability estimates for 
solutions in suitable spaces. Thus, this structure immediately suggests that it is natural to look for
solutions to \eqref{probmunl2-intro} in Wasserstein spaces (see Section \ref{sub:notation}), and hence the definition of solution to \eqref{probmunl2-intro} we are interested into is the following. 
\begin{defn} \label{solutionnl}
	Let $p \ge 1$, $T>0$ (possibly $T=\infty$) and $\overline{\mu} \in \W_p(\R^{2d})$. A{ continuous} curve of probability measures $\bm{\mu}\in C([0,T]; 
	\W_p(\R^{2d}))$ is a solution of \eqref{probmunl2-intro} if and only if for all $\psi \in C^\infty_c(\R^{2d})$ and for all $t \in [0,T]$ it holds
	\begin{equation}\label{eq:weaksolnl}
		\int_{\R^{2d}}\psi \, d\mu_t-\int_{\R^{2d}}\psi \, d\overline{\mu}=\int_0^t \int_{\R^{2d}}\left(v \cdot \nabla_x \psi+ \mathfrak{v}[s,\bm{\mu}](z)\cdot \nabla_v \psi+ \sigma\Delta_v \psi\right) \, d\mu_s \,  ds. 
	\end{equation}
	If $T=\infty$, we say that $\bm{\mu}$ is a global solution of \eqref{probmunl2-intro}, otherwise we call it a local solution.
\end{defn}

As a first goal of our work, we will focus on uniqueness and stability issues for \eqref{probmunl2-intro}. In doing so, 
we consider a set of assumptions on the drift term $\mathfrak{v}[t,\bm{\mu}]$ which is as general as possible 
and allows for possibly unbounded non-globally Lipschitz drift fields provided some H\"older continuity combined with
a dissipativity condition are satisfied (see Assumptions \ref{ass:v} below). Indeed, we refer to $\mathfrak{v}_3$ as a dissipativity condition since for $p=2$ it is implied by the metrical dissipativity condition considered in \cite{cavagnari2023dissipative}, see also \cite{cavagnari2023lagrangian}, as shown in Appendix \ref{appendix:equivalence}.

\medskip

Despite the presence of a nonlinearity in the structure of \eqref{probmunl2-intro}, the results we are going to show strongly rely on well-posedness results for a linear counterpart, that has been analyzed in \cite{Anceschi2025}.
Furthermore a key point for uniqueness and stability issues consists in showing that \eqref{probmunl2-intro} is indeed equivalent to its McKean-Vlasov counterpart,
in the sense that a solution $\mu$ is always the law of a stochastic process 
of the type
\begin{equation}\label{eq:MKVSDEauxdef}
		\begin{cases}
			dX(t)=V(t)dt\\
			dV(t)=\mathfrak{v}[t,\bm{\mu}](X(t),V(t))dt+\sqrt{2\sigma} dB(t)\\
			X(0)=X_0 \qquad V(0)=V_0 \qquad \bm{\mu}={\rm Law}(X,V),
		\end{cases}
	\end{equation}
where $B$ is a $d$-dimensional Brownian motion on $(\Omega, \Sigma, \bP)$, $\mathfrak{v}:[0,T]\times C([0,T];\W_p(\R^{2d})) \times \R^{2d} \to \R^{d}$ is a measurable function, $\sigma>0$ and $Z_0=(X_0,V_0) \in L^p(\Omega;\R^{2d})$. 
For this reason, throughout this work we fix a probability space $(\Omega, \Sigma, \bP)$ and we consider the following adapted formalism for McKean-Vlasov SDEs to the case of our interest. For simplicity, by $Z \in L^p(\Omega;C(\R_0^+;\R^{2d}))$ we actually mean that $\{Z(t)\}_{t \in [0,T]} \in L^p(\Omega;C([0,T];\R^{2d}))$ for all $T>0$.

\begin{defn}\label{def:uniq}
	Let $p \ge 1$, $T>0$ (possibly $T=\infty$) and consider the McKean-Vlasov SDE \eqref{eq:MKVSDEauxdef}, for $t \in [0,T]$.
	We say that a stochastic process $Z=(X,V)$ is a strong soution of \eqref{eq:MKVSDEauxdef} if $Z \in L^p(\Omega;C([0,T];\R^{2d}))$, $\int_0^t\left|\mathfrak{v}[t,\bm{\mu}](Z(s))\right|ds<\infty$ for all $t \in [0,T]$ and
	\begin{equation}\label{eq:strongsol}
		X(t)=X_0+\int_0^tV(s)ds \ \mbox{ and } \ 
		V(t)=V_0+\int_0^t\mathfrak{v}[s,\bm{\mu}](Z(s))ds+\sqrt{2\sigma}B(t),
	\end{equation}
	for a.a. $t \in [0,T]$ a.s., where $\bm{\mu}={\rm Law}(Z)$. We say that uniqueness in law holds for \eqref{eq:MKVSDEauxdef} if for any two strong solutions $Z_1,Z_2$ we have ${\rm Law}(Z_1(t))={\rm Law}(Z_2(t))$ for all $t \in [0,T]$. We say that pathwise uniqueness holds for \eqref{eq:MKVSDEauxdef} if for any two strong solutions $Z_1,Z_2$ we have $\bP(Z_1=Z_2)=1$. Clearly pathwise uniqueness implies uniqueness in law, while the converse is also true by \cite[Theorem 3.2]{cherny2002uniqueness}. For such a reason, we just state that the strong solution is unique.
\end{defn}

\subsection{Our assumptions}
Specifically, in order to state our main results we need to introduce a suitable set of assumptions to consider. 
For any $p \ge 1$, $t \ge 0$ and $\bm{\mu} \in C([0,T];\W_p(\R^{2d}))$ we set 
$$
	\overline{M}_p(t,\bm{\mu}):=\sup_{0 \le s \le t}M_p(\mu_s):=\sup_{0 \le s \le t}\int_{\R^{2d}}|z|^{p}d\mu_s,
$$ 
then the assumptions we will consider throughout this work are the following ones.
\begin{ProblemSpecBox}{Assumptions on $\mathfrak{v}$: $(\mathfrak{v})$}{$(\mathfrak{v})$}\label{ass:v}
	\begin{itemize}
		\item[$(\mathfrak{v}_0)$] $\mathfrak{v}:[0,+\infty)\times C(\R_0^+;\W_p(\R^{2d})) \times \R^{2d} \to \overline{\R}^d$ is a Carath\'eodory map, i.e. it is continuous in the variable $(\bm{\mu},z) \in C([0,T];\W_p(\R^{2d}))\times \R^{2d}$ and measurable in the variable $t \in [0,T]$.
		\item[$(\mathfrak{v}_1)$] For all $T>0$ there exist two constants $K>0$ and $\beta \in [0,1)$ such that 
		\begin{equation*}	
			|\mathfrak{v} [t,\bm{\mu}] (z)| \le K \left(1+|x|^{\frac{\beta}{3}}+|v|^\beta+\left(\overline{M}_p(T,\bm{\mu})\right)^{\frac{1}{p}}\right)
		\end{equation*} 
		for every $z \in \R^{2d}$, $t\in [0,T]$ and for every $\bm{\mu} \in  C([0,T];\W_p(\R^{2d}))$.
		\item[$(\mathfrak{v}_2)$] For all $T>0$ there exists an exponent $\alpha \in (\beta,1]$ with the property that for all $R>0$ 
		there exists a constant $L>0$ and  such that 
		for all $t \in [0,T]$, for all $\bm{\mu} \in C([0,T];\W_p(\R^{2d}))$ with $\overline{M}_p(T,\bm{\mu}) \le R$ and for all $z_1,z_2 \in \R^{2d}$ with $|z_1|,|z_2| \le R$ it holds
		\begin{align*}
			\left|\mathfrak{v}[t,\mu](z_1)-\mathfrak{v}[t,\mu](z_2)\right| \le L\left(|x_1-x_2|^{\frac{\alpha}{3}}+|v_1-v_2|^{\alpha}\right).
		\end{align*}
		\item[$(\mathfrak{v}_3)$] For all $T>0$ there exists a constant $D \ge 0$ such that for any couple of stochastic processes $Z_j=(X_j,V_j) \in L^p(\Omega;C([0,T];\R^{2d})$ with $X_j^\prime=V_j$ in $[0,T]$, $Z_1(0)=Z_2(0)$ and $\bm{\mu}^j={\rm Law}(X_j,V_j)$, $j=1,2$, it holds
		\begin{align*}
			\quad \int_0^t\E & \left[|V_1(s)-V_2(s)|^{p-2} (\mathfrak{v}[s,\bm{\mu}^1](Z_1(s))-\mathfrak{v}[s,\bm{\mu}^2](Z_2(s))) \cdot (V_1(s)-V_2(s)) \right]\, ds \\
			&\qquad \le D\int_0^t \sup_{0 \le z \le s}\E[|Z_1(z)-Z_2(z)|^p]\, ds \qquad \forall t \in [0,T].
		\end{align*}
		\item[$(\mathfrak{v}_4)$] For all $T>0$, $t \in [0,T]$ and ${\bm \mu}^1,{\bm \mu}^2 \in C([0,T];\W_p(\R^{2d}))$ with $\mu^1_s=\mu^2_s$ for all $s \in [0,t]$ it holds $\mathfrak{v}[t,\bm{\mu}^1]=\mathfrak{v}[t,\bm{\mu}^2]$.
	\end{itemize}
\end{ProblemSpecBox}

Despite its generality, it is clear that the dissipativity condition $(\mathfrak{v}_3)$ is not easily verified. However, we can still provide some
 sufficient conditions for its validity, see Appendix
\ref{appendix:equivalence}. In the notation for $\mathfrak{v}$, we separate the variables $t$ and $\bm{\mu}$ and $(X(t),V(t))$ for the following reason: throughout the proof of the main theorem, we will freeze the variable $\bm{\mu}$ and study \eqref{eq:MKVSDEauxdef} as a simple SDE, and we only need measurability in $t$.

\subsection{Main results}
Our first main result concerns well-posedness of \eqref{probmunl2-intro} in the sense discussed above
of which we give a precise statement here, while for the proof we refer to Section \ref{exuninl}. 

\begin{thm}\label{thm:main1}
	Let $p \ge 1$, $\overline{\mu} \in \W_p(\R^{2d})$ and suppose Assumptions \ref{ass:v} hold. 
	Then the non linear Cauchy problem \eqref{probmunl2-intro} 
	admits a unique global solution 
	$\bm{\mu} \in C^{\gamma_p}_{\rm loc}(\R_0^+;\W_p(\R^{2d}))$, where $\gamma_p=\frac{1}{\max\{2,p\}}$,
	that can be charachterized as follows.
	Let $B$ be a $d$-dimensional Brownian motion and $(X_0,V_0) \in L^p(\Omega;\R^{2d})$ be independent of $B$ with $\overline{\mu}={\rm Law}(X_0,V_0)$. Then there exists a unique global strong solution $(X,V) \in L^p(\Omega;C(\R_0^+;\R^{2d}))$ of \eqref{eq:MKVSDEauxdef}
	and $\bm{\mu}$ is the unique solution of \eqref{probmunl2-intro}. 
	Furthermore, for all $T>0$ there exist two functions $\mathcal{C}:\R_0^+ \to \R_0^+$ and $\widetilde{\Phi}:\R_0^+ \times \R_0^+ \to \R_0^+$ such that $\mathcal{C}$ is non-decreasing, $\widetilde{\Phi}(t;\cdot)$ is non-increasing for all $t \ge 0$, $\widetilde{\Phi}(\cdot;h)$ is a Young function for all $h \ge 0$ and it holds
	\begin{equation}\label{eq:momentbounds}
			\overline{M}_p(T,\bm{\mu})+\overline{M}_{\widetilde{\Phi}_p(\cdot;K)}(T,\bm{\mu})+\sup_{\substack{0 \le t,s \le T \\ t \not = s}}\frac{\W_p(\mu_{t},\mu_s)}{|t-s|^{\gamma_p}} \le \mathcal{C}(K),
	\end{equation}
		where $K$ is the constant defined in $(\mathfrak{v}_1)$.
\end{thm}
Furthermore, we derive the following stability result for solutions to the nonlinear system.
\begin{cor}\label{cor:stab}
	Fix $\overline{\mu} \in \W_p(\R^{2d})$ and consider a function $\mathfrak{v}:[0,T]\times C([0,T];\W_p(\R^{2d})) \times \R^{2d} \to \R^d$ sequence $\mathfrak{v}_j:[0,T]\times C([0,T];\W_p(\R^{2d})) \times \R^{2d} \to \R^d$ satisfying Assumptions \ref{ass:v}, where the constants $K$ and $D$ in $(\mathfrak{v}_1)$ and $(\mathfrak{v}_3)$ are independent of $j$. Denote by $\cK$ the set of curves of probability measures $\bm{\mu} \in C([0,T];\W_p(\R^{2d}))$ satisfying \eqref{eq:momentbounds}. Assume that  
	\begin{equation*}
		\lim_{j \to \infty}\mathfrak{v}_j[t,\bm{\mu}](z)=\mathfrak{v}[t,\bm{\mu}](z), \qquad \mbox{ for any } \quad t \in [0,T],\bm{\mu} \in \cK, z \in \R^{2d}.
	\end{equation*}
	Let $\bm{\mu}^j$, $\bm{\mu}$ be the unique solutions of \eqref{probmunl2-intro} with $\mathfrak{v}_j$ and $\mathfrak{v}$. Then
	\begin{equation*}
		\lim_{j \to \infty}\sup_{0 \le t \le T}\W_p(\mu_t^j,\mu_t)=0.
	\end{equation*}
\end{cor}
Eventually, we will turn our attention to optimal control problems connected with nonlinear Fokker-Planck equations. 
In the situation we consider, a two population system is given. A discrete leader population $Y_i(t)$, with $i=1, \ldots, m$, 
is controlled by a policy maker whose aim is steering a population of followers towards a decided goal
through mutual interactions. As naturally done in the kinetic approximation of multi-agent systems, the followers population is 
described by a distribution $\mu(t)$, whereas the dynamic is a second order one. Indeed, control fields act on the leaders $Y_i$, with
$i=1, \ldots, m$. Furthermore, the diffusion term $\Delta_v \mu$ reflects the presence of additive noise on the velocity of the generic 
follower as in \eqref{eq:MKVSDEauxdef}.
Specifically, the control dynamic we consider is of the form 
\begin{equation}\label{PDEODE-intro}
	\begin{cases}
		\partial_t \mu_t=-v\cdot \nabla_x \mu_t+\sigma\Delta_v\mu_t-{\rm div}_v((\mathfrak{v}[t,\bm{\mu}](z)\\
		\qquad \qquad \quad \, \, \, \qquad +\mathfrak{w}[t,\bm{Y}(t), {\bm W}(t)](z))\mu_t)) & (t,x,v) \in (0,T] \times \R^{2d} \\
		\dot{\mathbf{Y}}(t)=\mathbf{W}(t)=F[t,\bm{\mu}](\bm{Y}(t))+\mathbf{u}(t,\bm{\mu}) & t \in (0,T]\\
		\mu_0=\overline{\mu}, \qquad \mathbf{Y}(0)=\overline{\mathbf{Y}}, \qquad \mathbf{W}(0)=\overline{\mathbf{W}},
	\end{cases}
\end{equation}
under some specific assumptions thoroughly described in Section \ref{sec:optimal} and that allow for admissible controls ${\bf u}(t, \mu)$
to allow the policy-maker to tune its action on the state of the system. In this sense, our results can be seen as a genaralization to second 
order systems of control problems considered in \cite{Anceschi2025}. However, a relevant difference is that we allow for a way more general
structure of the drift $\mathfrak{v}[t,\bm{\mu}]$ considering unbounded Lipschitz frameworks. Then equation \eqref{PDEODE} is coupled with 
the cost functional 
\begin{equation*}
\mathcal{F} [\bm{u}] = \int_{0}^{T} \mathcal{L} (t, \bm{Y}, \bm{\mu} ) \, dt + \int_0^T \Psi (\bm{u}(t,\bm{\mu})) \, dt,
\end{equation*}
where $(\bm{Y},\bm{\mu})$ is the solution of \eqref{PDEODE} with control $\bm{u}$, $\mathcal{L}$
is a Lagrangian functional accounting for closedness to the decided target and $\Psi$ is a convex control cost. 
The specific assumptions on admissible controls and on the functional $\mathcal{F}$ are discussed in Section \ref{sec:optimal}. Thus, our second main result consists in showing the existence of a minimizer as it is stated below. 
\begin{thm} \label{main2}
	Suppose that Assumptions \ref{ass:v}, \ref{ass:w}, \ref{ass:F} and \ref{ass:calF} hold and let $\mathcal{A}$ be the class of admissible controls described via Assumptions \ref{ass:A}. Then the optimal control problem
	\begin{ProblemSpecBox}{Problem $1$}{$1$}\label{prob:1}
		Find $\bm{u}^\star \in \cA$ such that
		\begin{equation*}
			\F[\bm{u}^\star]=\min_{\bm{u} \in \cA}\F[\bm{u}].
		\end{equation*}
	\end{ProblemSpecBox}
	admits at least a solution.
\end{thm}

These control type results can be applied to various problems arising from real life phenomena. 
A model application is presented in Appendix \ref{appendix:example}. It can be considered as a second order generalization of multiagent control problem analysed in \cite{Ascione20236965,auletta2022herding}, as proposed, for instance, in \cite{albi2016invisible,albi2017mean}. 

\subsection{Structure of our work}
The paper is structured as follows. In Section \ref{sub:notation} we introduce the notation and we give some preliminaries on the involved spaces of functions and measures. In particular, in Section \ref{subs:linear} we recall the statement of the well-posedness result in the linear case, as shown in \cite{Anceschi2025}, and, at the same time, we prove, in Proposition \ref{prop:Young}, a moment estimate that will be crucial in the remainder of the work. 

Section \ref{exuninl} is entirely devoted to the proofs of Theorem \ref{thm:main1} and Corollary \ref{cor:stab}. The proof of Theorem \ref{thm:main1} is divided in $6$ steps: first we prove the existence of the solution in case $\mathfrak{v}$ is \textit{uniformly} sublinear (i.e., when $\mathfrak{v}_1$ holds without the additional term involving the measure) for a finite time horizon; this is right after extended to the general case, still for a finite time horizon; next, uniqueness of the local solution is shown and then this is used to construct the unique global solution; once existence and uniqueness has been shown, the H\"older continuity (in the Wasserstein distance) and finally the higher moment estimate is shown. Corollary \ref{cor:stab} represent a stability result in terms of suitable perturbations of the velocity field $\mathfrak{v}$. Notice that, since $\mathfrak{v}_3$ could be not trivial to verify, sufficient conditions, including the metric dissipativity condition, are given in Appendix \ref{appendix:equivalence}.

In Section \ref{sec:optimal} we discuss the optimal control Problem \ref{prob:1} on the PDE-ODE system \ref{PDEODE-intro}. Precisely, we first prove the existence and uniqueness of the solution of \eqref{PDEODE-intro} for any given control $\mathbf{u}$, together with a stability result with respect to suitable perturbations of the control function. Technical results related to the PDE-ODE systems are left in Appendix \ref{wp-pre}. Right after, we set the control problem and then we prove Theorem \ref{main2}. Possible different classes of admissible controls are described in Appendix \ref{sec:opcc} while a specific example is given in Appendix \ref{appendix:example}.

\section{Useful notation and preliminary results} 
\label{sub:notation}
Throughout the paper, $C$ will denote any constant depending at most on 
$d,\alpha,\beta,\sigma,p,T$ and $\overline{\mu}$ whose value is not relevant.

\subsection{Wasserstein spaces}
Let $(\fB,d_{\fB})$ be a complete separable metric space and $\cB(\fB)$ be the Borel $\sigma$-algebra on $\fB$. We denote by $\cP(\fB)$ the space of (Borel) probability measures on $\fB$. For any $\mu,\nu \in \cP(\fB)$, let $\Pi(\mu,\nu)$ be the set of (measure) couplings of $\mu$ and $\nu$ (see \cite[Definition 1.1]{villani}), i.e. Borel probability measures $\gamma \in \cP(\fB \times \fB)$ such that for any $A,B \in \cB(\fB)$ we have
\begin{equation*}
	\gamma(A \times \fB)=\mu(A), \qquad \gamma(\fB \times B)=\nu(B).
\end{equation*}
For any $p \ge 1$ and $x_0 \in \fB$ we denote by
\begin{equation*}
	M_p(\mu;x_0)=\int_{\fB}(d_{\fB}(x,x_0))^pd\mu(x)
\end{equation*}
the $p$-th moment (with respect to $x_0 \in \fB$) of $\mu \in \cP(\fB)$. It is clear that if $M_p(\mu;x_0)<\infty$ for some $x_0 \in \fB$, then it is finite for all $x_0 \in \fB$. The $p$-th Wasserstein space over $\fB$ is defined as (see \cite[Definition 6.4]{villani})
\begin{equation*}
	\W_p(\fB):=\{\mu \in \cP(\fB): \ M_p(\mu;x_0)<\infty, \forall x_0 \in \fB\},
\end{equation*}
equipped with the distance (see \cite[Definition 6.1]{villani})
\begin{equation*}
	\W_p(\mu,\nu)=\inf_{\gamma \in \Pi(\mu,\nu)}\left(\int_{\fB \times \fB}(d_{\fB}(x,y))^p\, d\gamma(x,y)\right)^{\frac{1}{p}}.
\end{equation*}
It is well-known (see \cite[Theorem 6.18]{villani}) that $\W_p(\fB)$ is a complete separable metric space. If $\fB$ is a Banach space, then we will denote $M_p(\mu):=M_p(\mu;0)$.
\subsection{Spaces of continuous functions and related Wasserstein spaces}
For any $T>0$, we denote by $C([0,T];\fB)$ the space of continuous functions $f:[0,T] \to \fB$ equipped with the uniform distance, i.e.
\begin{equation*}
	d_{\infty}(f,g)=\sup_{t \in [0,T]}d_{\fB}(f(t),g(t)).
\end{equation*}
Clearly, $C([0,T];\fB)$ is a complete separable metric space.  Furthermore, we denote by $C(\R_0^+;\fB)$, where $\R_0^+:=[0,+\infty)$, the space of continuous functions $f:\R_0^+ \to \fB$. In general, for reader's convenience, we will directly refer to $C(\R_0^+;\fB)$ as $C([0,T];\fB)$ with $T=\infty$.

Furthermore, for any $t \in [0,T]$, we denote by ${\rm ev}_t$ the \textit{evaluation map} defined on $f \in C([0,T];\fB)$ as ${\rm ev}_t(f)=f(t)$. For any given measure $\bm{\mu} \in \cP(C([0,T];\fB))$ and any $t \in [0,T]$, we denote 
$$
\mu_t={\rm ev}_t \sharp \bm{\mu},
$$ 
where $\sharp$ denotes the pushforward.

In what follows, we will use the notation $\bm{\mu} \in C([0,T];\W_p(\mathfrak{B}))$ to denote continuous curves of probability measures, with $\bm{\mu}=\{\mu_t\}_{t \in [0,T]}$. In particular, if $\mathfrak{B}$ is a Banach space and $T<\infty$, also $C([0,T];\W_p(\mathfrak{B}))$ is a Banach space under the norm
\begin{equation*}
	\left(\overline{M}_p(T,\bm{\mu})\right)^{\frac{1}{p}}:=\sup_{0 \le s \le T}\left(M_p(\mu_s)\right)^{\frac{1}{p}}.
\end{equation*}
We write $\bm{\mu} \in \W_p(C([0,T];\mathfrak{B}))$ to state that there exists a Borel probability measure (that we still denote $\bm{\mu}$) on $C([0,T];\mathfrak{B})$ such that $\mu_t={\rm ev}_t \sharp \bm{\mu}$. With this identification, we can say that $\W_p(C([0,T];\mathfrak{B}))\subset C([0,T];\W_p(\mathfrak{B}))$, since, clearly,
\begin{equation*}
	\overline{M}_p(T,\bm{\mu}) \le M_p(\bm{\mu}),
\end{equation*}
where the second $p$-th order moment is taken in $\W_p(C([0,T];\mathfrak{B}))$.

Notice that for any $T_1>T_2$ and any $\bm{\mu}=\{\mu_t\}_{t \in [0,T_1]} \in C([0,T_1];\W_p(\mathfrak{B}))$ we still denote $\bm{\mu}=\{\mu_t\}_{t \in [0,T_2]}$ (i.e. truncating the flow at $T_2$), so that with such an identification we have $C([0,T_1];\W_p(\mathfrak{B})) \subset C([0,T_2];\W_p(\mathfrak{B}))$.

For any $\gamma \in (0,1]$ and $T>0$ (with $T \not = \infty$), we can consider the set $C^\gamma([0,T];\W_p(\mathfrak{B}))$ of $\gamma$-H\"older continuous curves, i.e. $\bm{\mu} \in C^\gamma([0,T];\W_p(\mathfrak{B}))$ if and only if
	\begin{equation*}
		\sup_{\substack{0 \le t,s \le T \\
				t\not = s}}\frac{\W_p(\mu_t,\mu_s)}{|t-s|^\gamma}<\infty.
	\end{equation*}
Thanks to the previous identification, we can also define the set $C^\gamma_{\rm loc}(\R_0^+;\W_p(\mathfrak{B}))$ of locally $\gamma$-H\"older continuous curves, that is to say that $\bm{\mu} \in C^\gamma_{\rm loc}(\R_0^+;\W_p(\mathfrak{B}))$ if and only if $\bm{\mu} \in C^\gamma([0,T];\W_p(\mathfrak{B}))$ for all $T>0$.

\subsection{Random variables and couplings}
Throughout the paper, we will fix a probability space $(\Omega, \Sigma, \bP)$ on which all the random quantities will be defined. In general, we call any measurable function $X:\Omega \to \fB$ a $\fB$-valued random variable and we denote the set of all $\fB$-valued random variables as $\cM(\Omega;\fB)$. 

The law of a random variable $X \in \cM(\Omega;\fB)$ is defined as ${\rm Law}(X)=X \sharp \bP$, i.e. for any $A \in \cB(\fB)$
\begin{equation*}
	{\rm Law}(X)(A)=\bP(X^{-1}(A)).
\end{equation*}
We denote by $L^p(\Omega; \fB)$ the subspace of $\cM(\Omega;\fB)$ of random variables $X$ 
with the property that there exists $x_0 \in \fB$ such that
\begin{equation}\label{Lpgen}
	\E[(d_{\fB}(X,x_0))^p]<\infty,
\end{equation}
where $\E$ is the expected value. It is clear that if \eqref{Lpgen} holds for a specific choice of $x_0 \in \fB$, then it holds for any $x_0 \in \fB$.
We define a coupling of two probability measures $\mu,\nu \in \cP(\fB)$ as any random variable $(X,Y) \in \cM(\Omega;\fB \times \fB)$ such that ${\rm Law}(X)=\mu$ and ${\rm Law}(Y)=\nu$ (see \cite[Definition 1.1]{villani}). It is clear that for any $\mu,\nu \in \cP(\fB)$ and any measure coupling $\gamma \in \Pi(\mu,\nu)$, there exists a coupling $(X,Y)$ such that ${\rm Law}(X,Y)=\gamma$. Also note that $\mu \in \W_p(\fB)$ if and only if for any $X \in \cM(\Omega;\fB)$ with ${\rm Law}(X)=\mu$ it holds $X \in L^p(\Omega;\fB)$. In particular, this means that we can rewrite the Wasserstein distance between $\mu$ and $\nu$ as
\begin{equation*}
	\W^p_p(\mu,\nu)=\inf_{\substack{(X,Y) \in L^p(\Omega;\fB \times \fB)\\ {\rm Law}(X)=\mu, \ {\rm Law}(Y)=\nu}}\E[(d_{\fB}(X,Y))^p].
\end{equation*}
For $T>0$ (possibly $T=\infty)$ and $X \in \cM(\Omega;C([0,T];\fB))$, we can define for any $t \in [0,T]$ 
\begin{align*}
	X(t): &\, \Omega \mapsto \fB \\
	&\, \omega \mapsto X(t)(\omega)={\rm ev}_t(X(\omega)).
\end{align*} 
Hence, $X(t) \in \cM(\Omega;\fB)$ for any $t \in [0,T]$, and if ${\rm Law}(X)=\mu$, then ${\rm Law}(X(t))=\mu_t$. In particular, we can identify $X$ as a $\fB$-valued stochastic process with continuous trajectories.

\subsection{Young functions}
Let us recall the definition of Young function: we say that $\Phi$ is a Young function if for any $x \ge 0$
\begin{equation*}
	\Phi(x)=\int_0^x \varphi(y)dy,
\end{equation*}
where
\begin{itemize}
	\item[$(i)$] $\varphi(0)=0$ and $\varphi(y)>0$ for any $y>0$;
	\item[$(ii)$] $\varphi$ is non-decreasing;
	\item[$(iii)$] $\varphi$ is right-continuous;
	\item[$(iv)$] $\lim\limits_{y \to +\infty}\varphi(y)=+\infty$.
\end{itemize}
We say that a Young function $\Phi$ satisfies the $\Delta_2$ condition if there exists a constant $C_\Phi>1$ such that for any $x \ge 0$
\begin{equation*}
	\Phi(2x) \le C_\Phi\Phi(x).
\end{equation*}
It is known (see \cite[Remark 4.4.6]{pick2012function}) that any Young function $\Phi$ satisfying the $\Delta_2$ condition is such that there exist two constants $C_0>0$ and $p>1$ such that
\begin{equation*}
	\Phi(x) \le C_0x^p, \ \forall x \ge 1.
\end{equation*}
The converse is not always true. By definition, Young functions are non-decreasing, convex, and satisfy 
\begin{equation}\label{eq:suplin}
	\lim_{x \to +\infty}\frac{\Phi(x)}{x}=+\infty
\end{equation}
For further details on Young functions, we refer to \cite[Chapter 4]{pick2012function}. 

For our purposes, it will be useful to consider moments with respect to Young functions. Indeed, let $\Phi:[0,+\infty) \to [0,+\infty)$ be a Young function
and $(\fB, d_{\fB})$ be a complete separable metric space. Then for any $\mu \in \cP(\fB)$, we define
\begin{equation*}
	M_\Phi(\mu;x_0)=\int_{\fB}\Phi(d_{\fB}(x,x_0))\, d\mu(x)
\end{equation*}
and we say that $\mu \in \W_{\Phi}(\fB)$ if and only if there exists $x_0 \in \fB$ such that $M_\Phi(\mu;x_0)<\infty$ (or, equivalently, $M_\Phi(\mu;x_0)<\infty$ for any $x_0 \in \fB$). If $\fB$ is a Banach space, then we set $M_\Phi(\mu):=M_\Phi(\mu;0)$. Observe that, for instance, if we set $\Phi_p(x)=x^p$ for $p>1$, then $\W_{\Phi}(\fB)=\W_p(\fB)$. Similarly to what we did for the moments of order $p$, we set
\begin{equation*}
	\overline{M}_\Phi(t,\bm{\mu})=\sup_{0 \le s \le t}M_\Phi(\mu_s)
\end{equation*}
for any $\bm{\mu} \in C([0,T];\W_\Phi(\fB))$, where $\fB$ is a Banach space.

An important result that will be frequently employed is the de la Vall\'ee-Poussin Theorem (see \cite[Theorem T22]{meyer1966probability}).
More precisely, it states that a sequence $X_n$ of random variables is uniformly integrable if and only if 
there exists a Young function $\Phi:[0,+\infty) \to [0,+\infty)$ such that $\Phi(x) \le x^2$ and
\begin{equation}\label{eq:dVP}
	\sup \limits_{n} \E\left[\Phi(|X_n|)\right]<\infty.
\end{equation}
We use this result in combination with the Dunford-Pettis Theorem (see \cite[Theorem 4.30]{brezis2011functional}). 
Indeed, if for some $p \ge 1$ we have $X \in L^p(\Omega)$, then the set $\{ |X|^p \}$ is clearly a weakly compact set in $L^1$. Hence, it must be uniformly integrable and thus 
by the de la Vall\'ee-Poussin Theorem we know that there exists a Young function such that $\E [ \Phi(|X|^P) ] < \infty$. This proves the following result. 
\begin{prop}
	\label{prop:improvmom}
	Let $\mu \in \W_p(\fB)$. Then there exists a Young function $\Phi$ with the property that $\Phi(x) \le x^2$ and defining $\Phi_p(x) = \Phi(x^p)$ we have
	$\mu \in \W_{\Phi_p}(\fB)$.
\end{prop}

\subsection{The linear equation}\label{subs:linear}

Given a function $F:[0,+\infty) \times \R^{2d} \to \R^d$ we will consider the following set of assumptions:

\begin{itemize} 
	\item[$(A_0)$] $F:[0,+\infty) \times \R^{2d} \to \R^d$ is a Carath\'{e}odory map, i.e. it is measurable in the first variable and continuous in the second one; \label{itemA0}
	\item[$(A_1)$] For any $T>0$, there exist two constants $C>0$ and $\beta \in (0,1)$ such that 
	\begin{equation*}
			F(t,x,v) \le C(1+|x|^{\frac{\beta}{3}}+|v|^\beta),
	\end{equation*}
	for all $t \in [0,T]$ and $(x,v) \in \R^{2d}$;
	\item[$(A_2)$] For any $T>0$ and any compact set $K \subset \R^{2d}$ there exist two constants $L>0$ and $\alpha\in (\beta,1]$ such that
	\begin{equation*}
		|F(t,x_1,v_1)-F(t,x_2,v_2)| \le L(|x_1-x_2|^{\frac{\alpha}{3}}+|v_1-v_2|^\alpha)
	\end{equation*}
	for all $t \in [0,T]$ and $(x_1,v_1),(x_2,v_2) \in K$.
\end{itemize}
We also define 

\begin{equation*}
[F]_{\beta,T}:=\sup_{\substack{t \in [0,T] \\ (x,v) \in \R^{2d}}}\frac{F(t,x,v)}{1+|x|^{\frac{\beta}{3}}+|v|^\beta}.
\end{equation*}

We recall the following existence and uniqueness result for an associated linear system
proved in \cite{Anceschi2025}.

\begin{thm}\label{thm:existence}
	Under Assumptions $(A_0)$,$(A_1)$ and $(A_2)$, for any $\mu_0 \in \W_1(\R^{2d})$ and $T>0$, the equation
	\begin{equation}\label{probmulin1}
		\begin{cases} \partial_t \mu_t = - v \cdot \nabla_x \mu_t + \sigma \Delta_v \mu_t - \mathrm{div}_v (F(t,z) \mu_t) \quad &(t,x,v) \in  \R^+_0 \times\mathbb{R}^{2d}, \\
			\mu_0 = \overline{\mu}  \quad &(x,v) \in \R^{2d},
		\end{cases}
	\end{equation}
	admits exactly one global solution ${\bm \mu} :=  \{\mu_t\}_{t \in \R_0^+}\in C(\R_0^+; \W_1(\R^{2d}))$, which can be characterized as follows. Let $B$ be a $d$-dimensional Brownian motion and $(X_0,V_0) \in \cM(\Omega;\R^{2d})$ be independent of $B$ and such that ${\rm Law}(X_0,V_0)=\overline{\mu}$, then there exists a stochastic process $(X,V) \in L^1(\Omega; C([0,T]; \R^{2d}))$ for all $T>0$ such that ${\bm \mu} = Law(X,V)$ and $(X,V)$ is the pathwise unique strong solution to 
	\begin{equation}\label{eq:SDEstrongauxlin}
		\begin{cases}
			dX(t)=V(t)\, dt & t \in \R_0^+\\
			dV(t)=F(t,X(t),V(t))\, dt+\sqrt{2\sigma} \, dB(t) & t \in \R_0^+ \\
			X(0)=X_0, \quad V(0)=V_0.
		\end{cases}
	\end{equation} 
\end{thm}

Actually, we can show that the solution ${\bm \mu}$ is slightly more regular, depending on the regularity of the initial data $\overline{\mu}$. Before doing this, we state an inequality that we will use often throughout the paper, whose interest is not only related to the proof of the next proposition. Precisely, for any $T>0$ and $p>1$ it holds
\begin{align}\label{eq:Doob}
	\E\left[\sup_{t \in [0,T]}|B(t)|^p\right] \le \frac{1}{\sqrt{\pi}}\left(\frac{2p\sqrt{T}}{p-1}\right)^p\Gamma\left(\frac{p+1}{2}\right),
\end{align}
that follows by Doob's maximal inequality (see \cite[Theorem II.1.7]{revuzyor}) and the formula for the absolute $p$-moments of a Gaussian random variable [.....].

We have the following result.
\begin{prop}\label{prop:Young}
	Let Assumptions $(A_0)$, $(A_1)$ and $(A_2)$ hold true. Let $\Phi:[0,+\infty) \to [0,+\infty)$ be either the identity or a Young function satisfying $\Phi(x) \le x^2$ and define $\Phi_p(x)=\Phi(x^p)$ for any $x \ge 0$ and $p \ge 1$. Let $\bm{\mu} \in C([0,T];\W_1(\R^d))$ be the unique solution of \eqref{probmulin1}. If $\overline{\mu} \in \W_{\Phi_p}(\R^{2d})$ for some $p \ge 1$, then $\bm{\mu} \in \W_{\Phi_p}(C([0,T];\R^{2d}))$ and there exist two constants $\overline{C}_1,\overline{C}_2$ depending only on $p,\beta,T$ and $\sigma$ such that
	\begin{equation}\label{eq:Phi1}
		\overline{M}_{\widetilde{\Phi}_p}(T,\bm{\mu}) \le M_{\widetilde{\Phi}_p}(\bm{\mu}) \le  \overline{C}_2\left(1+M_{\Phi_p}(\overline{\mu})+[F]_{\beta,T}^{2p}\right),
	\end{equation}
	where
	\begin{equation}\label{eq:tildephi}
		\widetilde{\Phi}(r)=\Phi\left(\frac{e^{-\overline{C}_1[F]_{\beta,T}^pT}}{3\overline{C}_1}r\right).
	\end{equation}
	If, furthermore, $\Phi$ satisfies the $\Delta_2$ condition, then there exists a constant $\overline{C}_3$ depending only on $\Phi,p,\beta,T$ and $\sigma$ such that
	\begin{equation}\label{eq:Phi2}
		\overline{M}_{\Phi_p}(T,\bm{\mu}) \le \overline{C}_3\left(1+M_{\Phi_p}(\overline{\mu})+[F]_{\beta,T}^{2p}\right)e^{\overline{C}_3[F]_{\beta,T}^p}.
	\end{equation}
	Finally, if we consider the solution $Z=(X,V)$ of \eqref{eq:SDEstrongauxlin}, it also holds, almost surely,
	\begin{equation}\label{Zsupbound}
			\sup_{0 \le s \le T}|Z(s)| \le C_4\left(|Z_0|+[F]_{\beta,T}+\sigma\sup_{0 \le s \le T}|B(s)|\right)e^{C_4([F]_{\beta,T}+1)T}.
	\end{equation}
\end{prop}

\begin{proof}
		Let us first show that if $\Phi$ satisfies the $\Delta_2$ condition, then \eqref{eq:Phi1} implies \eqref{eq:Phi2}. Indeed, in such a case, since $\Phi(2x) \le C_\Phi \Phi(x)$ for a constant $C_\Phi>1$, we have
	\begin{equation*}
		\Phi(x^p) \le C_\Phi^{|\log_2(2\overline{C}_1)|}e^{\overline{C}_1\log(C_\Phi)[F]_{\beta,T}^pT}\widetilde{\Phi}(x^p).
	\end{equation*}
	This is enough to prove the desired implication.
	
	Now we prove \eqref{eq:Phi1}. Let $Z=(X,V)$ be any strong solution of \eqref{eq:SDEstrongauxlin} and recall that $\bm{\mu}={\rm Law}(Z)$. By Assumption $(A_1)$ and the second equation in \eqref{eq:SDEstrongauxlin} we know that there exists a constant $C_1$ depending on $p$ and $T$ such that
	\begin{equation*}
		|V(s)|^p\le C_1\left(|V_0|^p+ [F]_{\beta,T}^p  + [F]_{\beta,T}^p \int_0^t \left(|X(\tau)|^{\frac{p\beta}{3}}+|V(\tau)|^{p\beta}\right) d\tau+\sup_{0 \le s \le t}\sigma^p |B(s)|^p\right)
	\end{equation*}
	By using Young's inequality on $|X(\tau)|^{\frac{p\beta}{3}}$ and $|V(\tau)|^{p\beta}$ with exponents respectively $\frac{3}{\beta}$ and $\frac{1}{\beta}$ and taking the supremum inside the integral, we get
	\begin{equation}\label{eq:Vp}
		|V(s)|^p\le C_2\left(|V_0|^p+ [F]_{\beta,T}^p  + [F]_{\beta,T}^p\int_0^t \sup_{0 \le z \le \tau}\left(|X(z)|^{p}+|V(z)|^{p}\right) d\tau+\sigma^p\sup_{0 \le s \le t}|B(s)|^p\right)
	\end{equation}
	for some constant $C_2$ depending only on $p,T$ and $\beta$. Moreover, we also have, by the first eqution in \eqref{eq:SDEstrongauxlin},
	\begin{align}\label{eq:Xp}
		|X(s)|^p \le C_3\left(|X_0|^p+\int_0^t\sup_{0 \le z \le \tau}|V(z)|^pd\tau\right),
	\end{align}
	where $C_3$ is a constant depending only on $p$ and $T$. Summing \eqref{eq:Xp} and \eqref{eq:Vp} and taking the supremum, it is clear that
	\begin{align}
		\sup_{0 \le s \le t}|Z(s)|^p \le C_4\left(|Z_0|^p+[F]_{\beta,T}^p+([F]_{\beta,T}^p+1)\int_0^t\sup_{0 \le z \le \tau}|Z(z)|^pd\tau+\sigma^p\sup_{0 \le s \le t} |B(s)|^p\right),
	\end{align}
	where $C_4=\max\{C_2,C_3\}$ and we can use Gr\"onwall's inequality to get
	\begin{align}\label{eq:Zp}
		\sup_{0 \le s \le t}|Z(s)|^p \le C_4\left(|Z_0|^p+[F]_{\beta,T}^p+\sigma^p\sup_{0 \le s \le t}|B(s)|^p\right)e^{C_4([F]_{\beta,T}^p+1)T}.
	\end{align}
	For $p=1$, \eqref{eq:Zp} implies \eqref{Zsupbound}.
	
	Now we apply the function $\Phi$ on both sides of \eqref{eq:Zp}, obtaining, by convexity and the bound $\Phi(x) \le x^2$, setting $\overline{C}_1=\frac{1}{3C_4}$, \begin{align}\label{eq:Zp3}
		\widetilde{\Phi}\left(\sup_{0 \le s \le t}|Z(s)|^p\right) \le \frac{1}{3}\Phi\left(|Z_0|^p\right)+[F]_{\beta,T}^{2p}+\sigma^{2p}\sup_{0 \le s \le t}|B(s)|^{2p}.
	\end{align}
	Finally, we take the expectation on both sides of \eqref{eq:Zp3} and we use \eqref{eq:Doob} to achieve \eqref{eq:Phi1}.
\end{proof}

\section{Proof of Theorem \ref{thm:main1} and Corollary \ref{cor:stab}}\label{exuninl}
\subsection{Proof of Theorem \ref{thm:main1}}
	We proceed in four steps: first we prove the existence of a solution under a stronger assumption; then we weaken such assumption to get the actual existence result; next we prove uniqueness of the local solution; finally, we prove existence and uniqueness of the global solution. 
	\subsubsection{Existence with $\mathfrak{v}$ uniformly sublinear}\label{step1}
	We first assume that $\mathfrak{v}$ satisfies $(\mathfrak{v}_0)$, $(\mathfrak{v}_2)$ and $(\mathfrak{v}_3)$, while we substitute $(\mathfrak{v}_1)$ with the following further assumption:
	\begin{itemize}
		\item[$(\overline{\mathfrak{v}}_1)$] For all $T>0$ there exist two constants $K>0$ and $\beta \in [0,1)$ such that 
			\begin{equation*}	
				|\mathfrak{v} [t,\bm{\mu}] (z)| \le K \left(1+|x|^{\frac{\beta}{3}}+|v|^\beta\right)
			\end{equation*} 
			for every $z \in \R^{2d}$, $t\in [0,T]$ and for every $\bm{\mu} \in  C([0,T];\W_p(\R^{2d}))$.
	\end{itemize}
	Let $T>0$ (with $T \not = +\infty$) and $(X_0,V_0)$ be any random variable independent of $B$ with ${\rm Law}(X_0,V_0)=\overline{\mu}$. Fix any curve of probability measures $\bm{\nu} \in C([0,T];\W_p(\R^{2d}))$ and consider the following SDE.
	\begin{equation}\label{eq:SDEaux}
		\begin{cases}
			dX(t)=V(t)dt & t \in (0,T]\\
			dV(t)=F_{\bm{\nu}}(t,X(t),V(t))dt+\sqrt{2\sigma} dB(t) & t \in (0,T]\\
			X(0)=X_0, \qquad V(0)=V_0,
		\end{cases}
	\end{equation}
	where $F_{\bm{\nu}}(t,x,v)=\mathfrak{v}[t,\bm{\nu}](x,v)$.
	Notice that, by Assumptions $(\mathfrak{v}_0)$, $(\overline{\mathfrak{v}}_1)$ and $(\mathfrak{v}_2)$, we know that the function $F_{\bm{\nu}}$ satisfies Assumptions $(A_0)$, $(A_1)$ and $(A_2)$, hence, by Theorem~\ref{thm:existence}, there exists a unique strong solution $(X_{\bm{\nu}},V_{\bm{\nu}}) \in L^1(\Omega;C([0,T];\R^{2d}))$ of \eqref{eq:SDEaux}. Let $\bm{\widetilde{\nu}}={\rm Law}(X_{\bm{\nu}},V_{\bm{\nu}})$. By Proposition~\ref{prop:improvmom}, we know that there exists a Young functions $\Phi$ such that $\Phi(x) \le x^2$ and $\overline{\mu} \in \W_{\Phi_p}(\R^{2d})$, where $\Phi_p(r)=\Phi(r^p)$. Thus, by Proposition~\ref{prop:Young}, in particular \eqref{eq:Phi2}, we know that
	\begin{equation}\label{eq:Mp1}
		M_p(\bm{\widetilde{\nu}}) \le C(1+[F_{\bm{\nu}}]_{\beta,T}^{2p})e^{C[F_{\bm{\nu}}]^p_{\beta,T}} \le C(1+K^{2p})e^{CK^p},
	\end{equation}
	where we also used Assumption $(\overline{\mathfrak{v}}_1)$, and
	\begin{equation}\label{eq:MPhi1}
		M_{\widetilde{\Phi}^{\bm{\nu}}_p}(\bm{\widetilde{\nu}}) \le C(1+[F_{\bm{\nu}}]_{\beta,T}^{2p}) \le C(1+K^{2p}),
	\end{equation}
	where
	\begin{equation*}
		\widetilde{\Phi}^{\bm{\nu}}(r)=\Phi\left(Ce^{-C[F_{\bm{\nu}}]_{\beta,T}^p}r\right).
	\end{equation*}
	It is clear, since $\Phi$ is increasing, that for all $r \ge 0$
	\begin{equation*}
		\overline{\Phi}(r):=\Phi\left(Ce^{-CK^p}r\right) \le \Phi\left(Ce^{-C[F_{\bm{\nu}}]_{\beta,T}^p}r\right)=\widetilde{\Phi}^{\bm{\nu}}(r),
	\end{equation*}
	hence, by \eqref{eq:MPhi1} we have
	\begin{equation}\label{eq:MPhi2}
		M_{\overline{\Phi}_p}(\bm{\widetilde{\nu}}) \le C(1+K^{2p}).
	\end{equation}
 	Notice that both bounds \eqref{eq:Mp1} and \eqref{eq:MPhi1} are uniform with respect to $\bm{\nu}$. It is also clear, by \eqref{eq:Mp1}, that $(X_{\bm{\nu}},Y_{\bm{\nu}}) \in L^p(\Omega;C([0,T];\R^{2d}))$. Then we can consider the map
 	\begin{equation*}
 	\cT:\bm{\nu} \in \W_p(C([0,T];\R^{2d})) \mapsto \cT\bm{\nu}={\rm Law}(X_{\bm{\nu}},Y_{\bm{\nu}}) \in \W_p(C([0,T];\R^{2d})).
 	\end{equation*}
	We want to show that $\cT$ admits a fixed point. To do this, we first prove that
	\begin{equation*}
		{\sf Im}(\cT)=\{\cT{\bm{\nu}}: \ {\bm{\nu}} \in \W_p(C([0,T];\R^{2d}))\}
	\end{equation*}
	is relatively (sequentially) compact in $\W_p(C([0,T];\R^{2d}))$. First, we claim that the family $\{(X_{\bm{\nu}},V_{\bm{\nu}})\}_{\nu \in \W_p(C([0,T];\R^{2d}))}$ is uniformly tight, i.e. for any $\varepsilon>0$ there exists a compact set $K_\varepsilon \subset C([0,T];\R^{2d})$ such that
	\begin{equation*}
		\sup_{\bm{\nu} \in \W_p(C([0,T];\R^{2d}))}\bP\left((X_{\bm{\nu}},V_{\bm{\nu}}) \not \in K_\varepsilon\right)<\varepsilon.
	\end{equation*}
	 To prove this, we first fix $\gamma<\sfrac{1}{2}$ and consider the random variable
	\begin{equation*}
		L_{B,\gamma}=\sup_{\substack{t,s \in [0,T]\\ t \not = s }}\frac{|B(t)-B(s)|}{|t-s|^\gamma},
	\end{equation*}
	that is a.s. finite by the $\gamma$-H\"older continuity of the Brownian motion. By Fernique's theorem (see \cite[Theorem 1.3.2]{fernique1974regularite}), we know that there exists $\alpha>0$ such that $\E\left[e^{\alpha L_{B,\gamma}}\right]<\infty$, which in turn implies that $\E\left[|L_{B,\gamma}|^n\right]<\infty$ for all $n \in \N$. Furthermore, we observe that
	\begin{align}
		|V_{\bm{\nu}}(t)-V_{\bm{\nu}}(s)| &\le \int_s^t|\mathfrak{v}[\tau,\bm{\nu}](X_{\bm{\nu}}(\tau),V_{\bm{\nu}}(\tau))|d\tau +\sqrt{2\sigma} L_{B,\gamma}|t-s|^{\gamma}\\
		&\le K|t-s|+K\int_s^t(|X_{\bm{\nu}}(\tau)|^{\frac{\beta}{3}}+|V_{\bm{\nu}}(\tau)|^\beta)d\tau +\sqrt{2\sigma} L_{B,\gamma}|t-s|^{\gamma}\\
		&\le CK(1+\sup_{\tau \in [0,T]}|Z_{\bm{\nu}}(\tau)|)|t-s|+\sqrt{2\sigma} L_{B,\gamma}|t-s|^{\gamma}\\
		&\le CK(1+|Z_0|+K+\sup_{\tau \in [0,T]}|B(\tau)|+L_{B,\gamma})e^{CK}|t-s|^{\gamma},\label{eq:tightV}
	\end{align}
	where we used Assumption $(\overline{\mathfrak{v}}_1)$ in the second inequality, Young's inequality with exponents $\frac{3}{\beta}$ and $\frac{1}{\beta}$ in the third inequality and \eqref{Zsupbound} in the fourth inequality.
	%
	Furthermore, we easily have
	\begin{equation}\label{eq:tightX}
		|X_{\bm{\nu}}(t)-X_{\bm{\nu}}(s)| \le \left(\sup_{\tau \in [0,T]}|V_{\bm{\nu}}(\tau)|\right)|t-s| \le C\left(1+|Z_0|+K+\sup_{0 \le \tau \le T}|B(\tau)|\right)|t-s|,
	\end{equation}
	where we used again \eqref{Zsupbound}. Consider, for any $\delta \in (0,T]$ and $n \in \NN$, the (random) modulus of continuity
	\begin{equation*}
		m_{\bm{\nu}}(\delta)=\sup_{\substack{t,s \in (0,T] \\ |t-s|<\delta}}|(X_{\bm{\nu}}(t),V_{\bm{\nu}}(t))-(X_{\bm{\nu}}(s),V_{\bm{\nu}}(s))|.
	\end{equation*}
	Then we have, by \eqref{eq:tightV} and \eqref{eq:tightX},
	\begin{equation*}
		m_{\bm{\nu}}(\delta) \le C\left(1+|Z_0|+K+\sup_{\tau \in [0,T]}|B(\tau)|+L_{B,\gamma}\right)\delta^\gamma
	\end{equation*}
	and taking the expectation
	\begin{align*}
		\E\left[m_{\bm{\nu}}(\delta)\right] &\le C\left(1+M_1(\overline{\mu})+K+\E[\sup_{\tau \in [0,T]}|B(\tau)|]+\E\left[L_{B,\gamma}\right]\right)\delta^\gamma \le C(1+K)\delta^\gamma
	\end{align*}
	where we used the inequality $\E[\sup_{\tau \in [0,T]}|B(\tau)|] \le \sqrt{\E[\sup_{\tau \in [0,T]}|B(\tau)|^2]}$ and \eqref{eq:Doob}.
	Now fix any $\varepsilon>0$ and any $\eta \in (0,1]$. By Markov's inequality we have
	\begin{equation*}
		\bP(m_{\bm{\nu}}(\delta) \ge \varepsilon) \le \frac{\E[m_{\bm{\nu}}(\delta)]}{\varepsilon} \le C(1+K)\frac{\delta^{\gamma}}{\varepsilon}.
	\end{equation*}
	Hence, if we set $\delta_0=\frac{(\varepsilon \eta)^{\frac{1}{\gamma}}}{C(1+K)}$, we get, for any $\delta<\delta_0$ and any $n \in \NN$,
	\begin{equation*}
		\bP(m_n(\delta) \ge \varepsilon) \le \eta.
	\end{equation*}
	Combining this with the fact that ${\rm Law}(X_{\bm{\nu}}(0),V_{\bm{\nu}}(0))=\overline{\mu}$ for any $n \in \NN$, we conclude that the sequence $\{(X_{\bm{\nu}},V_{\bm{\nu}})\}_{n \in \N}$ is uniformly tight by \cite[Theorem 7.3]{billingsley2013convergence}. 
	
	Next, we claim that $\{(X_{\bm{\nu}},V_{\bm{\nu}})\}_{{\bm{\nu}} \in \W_p(C([0,T];\R^{2d}))}$ are uniformly $L^p$-integrable, i.e., setting $\widetilde{\bm{\nu}}={\rm Law}(X_{\bm{\nu}},Y_{\bm{\nu}})$ and $B_R=\{Z \in C([0,T];\R^{2d}), \ \sup_{t \in [0,T]}|Z(t)|<R\}$, we have
	\begin{equation*}
		\sup_{\bm{\nu} \in \W_p(C[0,T];\R^{2d})}\lim_{R \to +\infty}\int_{C([0,T];\R^{2d})\setminus B_R}\sup_{t \in [0,T]}|Z(t)|^pd\widetilde{\bm{\nu}}(Z)=0.
	\end{equation*}
	Indeed, \eqref{eq:MPhi2} guarantees that there exists a Young function $\overline{\Phi}$ such that
	\begin{equation*}
		\sup_{{\bm{\nu}} \in \W_p(C([0,T];\R^{2d}))}\E\left[\overline{\Phi}\left(\sup_{\tau \in [0,T]}|(X_{\bm{\nu}}(\tau),V_{\bm{\nu}}(\tau))|^p\right)\right] \le C(1+K^{2p}),
	\end{equation*}
	which in turn implies the uniform $L^p$-integrability by the de la Vall\'ee-Poussin Theorem. Thus, by the characterization of relatively compact sets in Lebesgue-Bochner spaces, shown in \cite[Theorem 2.1]{wang2019compactness}, we know that
	\begin{equation}
		{\sf Im}^\star(\cT):=\{(X_{\bm \nu},Y_{\bm \nu}) \in L^p(\Omega;C([0,T];\R^{2d})), \ \bm \nu \in \W_p(C([0,T];\R^{2d}))\}
	\end{equation}
	is relatively compact in $L^p(\Omega;C([0,T];\R^{2d}))$, while by \cite[Proposition 7.1.5]{ambrosio2005gradient} we get that ${\sf Im}(\cT)$ is relatively compact.
%
%
%
%
%
	
	Now we prove that $\cT$ is continuous. Indeed, consider any sequence $\{\bm{\nu}^n\}_{n \in \N} \subset \W_p(C([0,T];\R^{2d}))$ converging towards $\bm{\nu} \in \W_p(C([0,T];\R^{2d}))$ and denote by $(X_n,V_n)$ the solutions of \eqref{eq:SDEaux} with respect to $\bm{\nu}^n$ for each $n \in \N$. Assume, without loss of generality, that
	\begin{equation*}
		\limsup_{n \to +\infty}\W_p(\cT\bm{\nu}^n,\cT\bm{\nu})=\lim_{n \to +\infty}\W_p(\cT\bm{\nu}^n,\cT\bm{\nu})
	\end{equation*}
	and recall that there exists a subsequence $\{(X_{n_k},V_{n_k})\}_{k \in \N}$ converging towards $(X_\infty,V_\infty) \in L^p(\Omega;C([0,T];\R^{2d}))$ both in the $L^p$-norm and almost surely. Now we show that the limit $(X_\infty,V_\infty)$ is the solution of \eqref{eq:SDEaux} associated with $\bm{\nu}$. Indeed, by Assumption $(\mathfrak{v}_0)$, we know that for any $t \in [0,T]$ and almost surely
	\begin{equation}\label{eq:limit}
		\lim_{k \to \infty}\mathfrak{v}[t,\bm{\nu}^{n_k}](X_{n_k}(t),V_{n_k}(t))=\mathfrak{v}[t,\bm{\nu}](X_\infty(t),V_\infty(t)).
	\end{equation}
	Now fix any $\omega \in \Omega$ such that \eqref{eq:limit} holds. In the following we will omit the dependence on $\omega$ for the ease of the reader. By Assumption $(\overline{\mathfrak{v}}_1)$, it holds
	\begin{align*}
		\mathfrak{v}[t,\bm{\nu}^{n_k}](X_{n_k}(t),V_{n_k}(t))&\le K(1+|X_{n_k}(t)|^{\frac{\beta}{3}}+|V_{n_k}(t)|^\beta) \le CK(1+\sup_{0 \le s \le T}|Z_{n_k}(t)|) \\
		&\le CK(1+|Z_0|+K+\sup_{0 \le s \le T}|B(s)|)e^{CK},
	\end{align*}
	where the upper bound is independent of $t \in [0,T]$ and $n \in \N$. Hence, by the dominated convergence theorem, we get
	\begin{align*}
		V_\infty(t)&=\lim_{k \to \infty}V_{n_k}(t)=V_0+\lim_{k \to \infty}\int_0^t\mathfrak{v}[s,\bm{\nu}^{n_k}](X_{n_k}(s),V_{n_k}(s))ds+\sqrt{2\sigma}B(s)\\
		&=V_0+\int_0^t\mathfrak{v}[s,\bm{\nu}](X_\infty(s),V_\infty(s))ds+\sqrt{2\sigma}B(s),
	\end{align*}
	while, since $V_{n_k} \to V_\infty$ uniformly almost surely,
	\begin{equation*}
		X_\infty(t)=\lim_{k \to \infty}X_{n_k}(t)=X_0+\lim_{k \to \infty}\int_0^tV_{n_k}(s)ds=X_0+\int_0^tV_\infty(s)ds,
	\end{equation*}
	i.e. $(X_\infty,V_\infty)$ solves \eqref{eq:SDEaux} with respect to $\bm{\nu}$. This shows that for all $k \in \N$ $(X_{n_k},V_{n_k})$ and $(X_\infty,V_\infty)$ constitute a coupling of $\cT\bm{\nu}^{n_k}$ and $\cT\bm{\nu}$ and then
	\begin{align*}
		\lim_{n \to +\infty}\W_p(\cT\bm{\nu}^n,\cT\bm{\nu})&=\lim_{k \to +\infty}\W_p(\cT\bm{\nu}^{n_k},\cT\bm{\nu}) \\
		&\le \lim_{k \to +\infty}\E\left[\sup_{t \in [0,T]}|(X_{n_k}(t),V_{n_k}(t))-(X_\infty(t),V_\infty(t))|^p\right]=0.
	\end{align*}
	Since $\bm{\nu} \in W_p(C([0,T];\R^{2d}))$ is arbitrary, this implies that $\cT$ is continuous.

	We can then use Schauder's fixed point theorem to ensure that there exists \linebreak ${\bm{\mu}} \in \W_p(C([0,T];\R^{2d}))$ such that ${\bm{\mu}}=\cT{\bm{\mu}}$. We only need to show that ${\bm{\mu}}$ is a solution of \eqref{probmunl2-intro}. To do this, let $(X,V)$ be the solution of \eqref{eq:SDEaux} associated with ${\bm{\mu}}$. Then, since ${\bm{\mu}}=\cT{\bm{\mu}}={\rm Law}(X,V)$, the process $(X,V)$ solves the McKean-Vlasov SDE \eqref{eq:MKVSDEauxdef}. To obtain \eqref{eq:weaksolnl}, let $\psi \in C^\infty_c(\R^{2d})$. By It\^o's formula \cite[Theorem 4.1.2]{oksendal2013stochastic} we have
	\begin{align*}
		\psi(X(t),V(t))&=\psi(X_0,V_0)+\int_0^t\left(V(s)\cdot \nabla_x\psi(X(s),V(s))\right.\\
		&\quad \left. +\mathfrak{v}[s,\bm{\mu}](X(s),V(s))\cdot \nabla_v \psi(X(s),V(s))+\sigma \Delta_v\psi(X(s),V(s))\right)ds\\
		&\quad +\sqrt{2\sigma}\sum_{j=1}^{d}\int_{0}^{s}\frac{\partial \, \psi}{\partial \, v_j}(X(s),V(s))dB_j(s),
	\end{align*}
	where $B_j$ is the $j$-th coordinate of the $d$-dimensional Brownian motion $B=(B_1,\dots,B_d)$. Since $\psi \in C^\infty_c(\R^{2d})$, it is clear that for all $j=1,\dots,d$ it holds
	\begin{equation*}
		\left|\frac{\partial \, \psi}{\partial \, v_j}(X(s),V(s))\right| \le \Norm{\nabla \psi}{L^\infty(\R^{2d})}, \ \forall s \ge 0,
	\end{equation*}
	hence $\int_{0}^{s}\frac{\partial \, \psi}{\partial \, v_j}(X(s),V(s))dB_j(s)$ is a martingale \cite[Corollary 3.2.6]{oksendal2013stochastic} and then
	\begin{equation*}
		\E\left[\int_{0}^{s}\frac{\partial \, \psi}{\partial \, v_j}(X(s),V(s))dB_j(s)\right]=0, \ j=1,\dots,d.
	\end{equation*}
	As a consequence we get
	\begin{align}
		\E\left[\psi(X(t),V(t))\right]&=\E\left[\psi(X_0,V_0)\right]+\E\left[\int_0^t\left(V(s)\cdot \nabla_x\psi(X(s),V(s))\right.\right.\\
		&\left.\vphantom{\int_0^t}\left.+\mathfrak{v}[s,\bm{\mu}](X(s),V(s))\cdot \nabla_v \psi(X(s),V(s))+\sigma \Delta_v\psi(X(s),V(s))\right)ds\right].\label{eq:expect}
	\end{align}
	Next, notice that
	\begin{align*}
		|V(s)\cdot \nabla_x\psi(X(s),V(s))| &\le |Z(s)|\Norm{\nabla \psi}{L^\infty(\R^{2d})}\\
		|\mathfrak{v}[s,\bm{\mu}](X(s),V(s))\cdot \nabla_v \psi(X(s),V(s))|&\le CK(1+|Z(s)|)\Norm{\nabla \psi}{L^\infty(\R^{2d})}\\
		\sigma \Delta_v\psi(X(s),V(s)) & \le \sigma \Norm{\Delta \psi}{L^\infty(\R^{2d})},
	\end{align*}
	hence
	\begin{align*}
		\E&\left[\int_0^t\left|V(s)\cdot \nabla_x\psi(X(s),V(s))\right.\right.\\
		&\left.\vphantom{\int_0^t}\left.\qquad \qquad +\mathfrak{v}[s,\bm{\mu}](X(s),V(s))\cdot \nabla_v \psi(X(s),V(s))+\sigma \Delta_v\psi(X(s),V(s))\right|ds\right]\\
		&\qquad \le C(1+K)(\Norm{\nabla \psi}{L^\infty(\R^{2d})}+\Norm{\Delta \psi}{L^\infty(\R^{2d})})\int_0^t \E[1+|Z(s)|]ds\\
		&\qquad \le C(1+K)(\Norm{\nabla \psi}{L^\infty(\R^{2d})}+\Norm{\Delta \psi}{L^\infty(\R^{2d})})(1+\overline{M}_1(T;\bm{\mu}))T,
	\end{align*}
	where we recall, by H\"older's inequality, that $\overline{M}_1(T;\bm{\mu}) \le \left(\overline{M}_p(T;\bm{\mu})\right)^{\frac{1}{p}}<\infty$. Thus we can use Fubini's theorem in \eqref{eq:expect} to achieve
	\begin{align}
		\E\left[\psi(X(t),V(t))\right]&=\E\left[\psi(X_0,V_0)\right]+\int_0^t\E\left[\left(V(s)\cdot \nabla_x\psi(X(s),V(s))\right.\right.\\
		&\left.\left.+\mathfrak{v}[s,\bm{\mu}](X(s),V(s))\cdot \nabla_v \psi(X(s),V(s))+\sigma \Delta_v\psi(X(s),V(s))\right)\right]ds.\label{eq:expect2}
	\end{align}
	Now we use the fact that $\mu_t={\rm Law}(X(t),V(t))$ and $\overline{\mu}={\rm Law}(X_0,V_0)$ to finally get
	\begin{align}
		\int_{\R^{2d}}\psi d\mu_t&=\int_{\R^{2d}}\psi d\overline{\mu}+\int_0^t\int_{\R^{2d}}\left(v \cdot \nabla_x\psi+\mathfrak{v}[s,\bm{\mu}](z)\cdot \nabla_v \psi+\sigma \Delta_v\psi\right) d\mu_s ds,
	\end{align}
	that is \eqref{eq:weaksolnl}.

	\subsubsection{Existence for a general $\mathfrak{v}$}\label{step2}
	Now, let us consider $\mathfrak{v}$ as in the statement of Theorem~\ref{thm:main1}. For any $N>0$ let $\eta_N \in C^\infty_{c}(\R)$ such that ${\sf supp}(\eta_N)\subset [-1,N+1]$, $\eta_N(r)=1$ for any $r \in [0,N]$ and $\eta_N(x) \in [0,1]$ for any $r \in \R$. Define, for any $t \in [0,T]$, ${\bm{\mu}} \in C([0,T];\W_p(\R^{2d}))$ and $z \in \R^{2d}$,
	\begin{equation*}
		\mathfrak{v}_{N}[t,{\bm{\mu}}](z)=\mathfrak{v}[t,{\bm{\mu}}](z)\, \eta_N(\overline{M}_p(t,{\bm{\mu}}))
	\end{equation*} 
	By definition, $\mathfrak{v}_N$ satisfies $(\mathfrak{v}_0)$, $(\mathfrak{v}_2)$ and $(\mathfrak{v}_4)$. Furthermore, notice that
	\begin{equation*}
		|\mathfrak{v}_N[t,\bm{\mu}](z)| \le K\left(1+|x|^{\frac{\beta}{3}}+|y|^\beta+(\overline{M}_p(T,\bm{\mu}))^{\frac{1}{p}}\right)\eta_N(\overline{M}_p(t,{\bm{\mu}})).
	\end{equation*}
	If $\overline{M}_p(t,\bm{\mu}) \le N+1$, then
	\begin{equation*}
		|\mathfrak{v}_N[t,\bm{\mu}](z)| \le K(N+1)^{\frac{1}{p}}\left(1+|x|^{\frac{\beta}{3}}+|y|^\beta\right),
	\end{equation*}
	otherwise
	\begin{equation*}
		|\mathfrak{v}_N[t,\bm{\mu}](z)|=0 \le K(N+1)^{\frac{1}{p}}\left(1+|x|^{\frac{\beta}{3}}+|y|^\beta\right),
	\end{equation*}
	hence $\mathfrak{v}_N$ satisfies $(\overline{\mathfrak{v}}_1)$. To use the first clam of this proof, it is not necessary to check whether $(\mathfrak{v}_3)$ still holds. For any $N \in \N$, we can find a solution $\bm{\mu}^N \in \W_p(C([0,T];\R^{2d}))$ of \eqref{probmunl2-intro} with drift $\mathfrak{v}_N$ and a stochastic process $(X_N,Y_N) \in L^p(\Omega;C([0,T];\R^{2d}))$ solving \eqref{eq:MKVSDEauxdef} such that $\bm{\mu}^N={\rm Law}(X_N,V_N)$. Now we need to prove a further uniform bound on the moments $\overline{M}_p(t,\bm{\mu}^N)$. To do this we argue as in the proof of Proposition \ref{prop:Young}. Fix $t \in [0,T]$ and observe that for any $s \le t$
	\begin{align*}
		&|V_N(s)|^p \le C\left(|V_0|^p+\int_0^s|\mathfrak{v}_N[\tau,\bm{\mu}^N](X_N(\tau),V_N(\tau))|^pd\tau+\sup_{0 \le \tau \le T}|B(\tau)|^p\right)\\
		&\le C\left(|V_0|^p+K +K\int_0^s\left(|X_N(\tau)|^{\frac{p\beta}{3}}+|V_N(\tau)|^{p\beta}+\overline{M}_p(\tau;\bm{\mu}^N)\right)d\tau+\sup_{0 \le \tau \le T}|B(\tau)|^p\right)\\
		&\le C\left(|V_0|^p+K +K\int_0^s\left(|X_N(\tau)|^p+|V_N(\tau)|^p\right)d\tau+K\int_0^t\overline{M}_p(\tau;\bm{\mu}^N)d\tau+\sup_{0 \le \tau \le T}|B(\tau)|^p\right),
	\end{align*}
	where we used the fact that $|\mathfrak{v}_N[\tau,\bm{\mu}^N](z)| \le |\mathfrak{v}[\tau,\bm{\mu}^N](z)|$, Assumption $(\mathfrak{v}_1)$ and Young's inequality with exponents $\frac{3}{\beta}$ and $\frac{1}{\beta}$. On the other hand, for any $s \le t$
	\begin{align*}
		|X_N(t)|^p \le C\left(|X_0|^p+\int_0^s|V_N(\tau)|^pd\tau\right)
	\end{align*}
	and then, setting $Z_N=(X_N,V_N)$,
	\begin{align*}
		&|Z_N(s)|^p \le C\left(|Z_0|^p+K +K\int_0^s|Z_N(\tau)|^pd\tau+K\int_0^t\overline{M}_p(\tau;\bm{\mu}^N)d\tau+\sup_{0 \le \tau \le T}|B(\tau)|^p\right).
	\end{align*}
	By Gr\"onwall's inequality, this implies
	\begin{align*}
		&|Z_N(s)|^p \le C\left(|Z_0|^p+K +K\int_0^t\overline{M}_p(\tau;\bm{\mu}^N)d\tau+\sup_{0 \le \tau \le T}|B(\tau)|^p\right)e^{KT}.
	\end{align*}
	Now take the expectation to achieve
	\begin{equation}\label{eq:Mps}
		M_p(\mu_s^N) \le C\left(1+M_p(\overline{\mu})+K +K\int_0^t\overline{M}_p(\tau;\bm{\mu}^N)d\tau\right)e^{KT},
	\end{equation}
	where we also used \eqref{eq:Doob} (and H\"older's inequality if $p=1$). Since \eqref{eq:Mps} holds for all $s \in [0,t]$, we can take the supremum and get
	\begin{equation*}
		\overline{M}_p(t,\bm{\mu}^N) \le C\left(1+M_p(\overline{\mu})+K +K\int_0^t\overline{M}_p(\tau;\bm{\mu}^N)d\tau\right)e^{KT},
	\end{equation*}
	which in turn implies, by Gr\"onwall's inequality,
	\begin{equation}\label{eq:Mps2}
		\overline{M}_p(t,\bm{\mu}^N) \le C\left(1+K\right)e^{KT(1+e^{KT})}.
	\end{equation}
	Now consider any $N>C\left(1+K\right)e^{KT(1+e^{KT})}$. Then 
	\begin{equation*}
	\mathfrak{v}_N[t,\bm{\mu}^N](X_N(t),V_N(t))=\mathfrak{v}[t,\bm{\mu}^N](X_N(t),V_N(t)),	
	\end{equation*}
	thus $(X_N,V_N)$ solves \eqref{eq:MKVSDEauxdef}, while $\bm{\mu}^N={\rm Law}(X_N,V_N)$ solves \eqref{probmunl2-intro}. 
	This also proves the first bound in \eqref{eq:momentbounds}.
	
	\subsubsection{Uniqueness of the local solution}
	Let us first show that if $\bm{\mu}$ is a solution of \eqref{probmunl2-intro}, then there exists a strong solution $(X,V)$ of \eqref{eq:MKVSDEauxdef} such that $\bm{\mu}={\rm Law}(X,V)$. To do this, fix a solution $\bm{\mu}$ of \eqref{eq:MKVSDEauxdef} and define $F_{\bm{\mu}}(t,x,v)=\mathfrak{v}[t,\bm{\mu}](x,v)$. Then $\bm{\mu}$ solves the linear PDE
	\begin{equation}\label{probmulin}
		\begin{cases} \partial_t \mu_t = - v \cdot \nabla_x \mu_t + \sigma \Delta_v \mu_t - \mathrm{div}_v (F_{\bm{\mu}}(t,z) \mu_t) \quad &(t,x,v) \in  \R^+_0 \times\mathbb{R}^{2d}, \\
			\mu_0 = \bar{\mu}  \quad &(x,v) \in \R^{2d},
		\end{cases}
	\end{equation}
	hence, by Theorem~\ref{thm:existence} we know that $\bm{\mu}={\rm Law}(X,V)$ where $(X,V)$ is the strong solution of \eqref{eq:SDEaux}. The fact that $(X,V) \in L^p(\Omega;C([0,T];\R^{2d}))$ follows by \eqref{eq:Phi2} where $\Phi$ is the identity. Hence, in particular, by definition of $F_{\bm{\mu}}$, $(X,V)$ solves \eqref{eq:MKVSDEauxdef} as required.
	
	Now assume that $\bm{\mu}^1,\bm{\mu}^2 \in C([0,T];\W_p(\R^{2d}))$ are both solutions of \ref{probmunl2-intro} and let $Z_j=(X_j,V_j) \in L^p(\Omega;C([0,T];\R^{2d}))$ be the respective solutions of \eqref{eq:MKVSDEauxdef} with ${\rm Law}(X_j,V_j)=\bm{\mu}^j$, $j=1,2$. Define the auxiliary processes $\delta X(t)=X_1(t)-X_2(t)$, $\delta V(t)=V_1(t)-V_2(t)$ and $\delta Z(t)=Z_1(t)-Z_2(t)$. Consider any $t \in [0,T]$ and $s \in [0,t]$. For a fixed $\omega \in \Omega$ for which \eqref{eq:strongsol} holds, we have
	\begin{equation*}
		\delta V(s,\omega)=\int_0^s(\mathfrak{v}[\tau,\bm{\mu}^1](Z_1(\tau,\omega))-\mathfrak{v}[\tau,\bm{\mu}^2](Z_2(\tau,\omega)))d\tau,
	\end{equation*}
	i.e. $\delta V(\cdot,\omega)$ is absolutely continuous. As a consequence also $|\delta V(\cdot,\omega)|^p$ is absolutely continuous and by the chain rule and taking the expectation it holds
	\begin{equation*}
		\E\left[\left|\delta V(s)\right|^p\right]=p\E\left[\int_0^s\left|\delta V(\tau)\right|^{p-2}(\mathfrak{v}[\tau,\bm{\mu}^1](Z_1(\tau))-\mathfrak{v}[\tau,\bm{\mu}^2](Z_2(\tau)))\cdot \delta V(\tau)d\tau\right].
	\end{equation*}
	Now we want to use Fubini's theorem to exchange the expectation with the integral. To do this, we observe that
	\begin{align*}
		&\left|\delta V(\tau)\right|^{p-1}|\mathfrak{v}[\tau,\bm{\mu}^1](Z_1(\tau))-\mathfrak{v}[\tau,\bm{\mu}^2](Z_2(\tau))| \\
		&\qquad \le CK(|V_1(\tau)|^{p-1}+|V_2(\tau)|^{p-1})\left(1+|X_1(\tau)|^{\frac{\beta}{3}}+|V_1(\tau)|^\beta+|X_2(\tau)|^{\frac{\beta}{3}}+|V_2(\tau)|^\beta\right.\\
		&\left.\qquad \qquad +\left(\overline{M}_p(\tau,\bm{\mu}^1)\right)^{\frac{1}{p}}+\left(\overline{M}_p(\tau,\bm{\mu}^1)\right)^{\frac{1}{p}}\right)\\
		&\le CK(|V_1(\tau)|^{p-1}+|V_2(\tau)|^{p-1})(|X_1(\tau)|^{\frac{\beta}{3}}+|X_1(\tau)|^{\frac{\beta}{3}})\\
		&\quad +CK(|V_1(\tau)|^{p-1}+|V_2(\tau)|^{p-1}+|V_1(\tau)|^{p-1+\beta}+|V_2(\tau)|^{p-1+\beta})\\
		&\quad +CK(|V_1(\tau)|^{p-1}+|V_2(\tau)|^{p-1})\left(\left(\overline{M}_p(\tau,\bm{\mu}^1)\right)^{\frac{1}{p}}+\left(\overline{M}_p(\tau,\bm{\mu}^2)\right)^{\frac{1}{p}}\right)\\
		&\le CK\left[1+(|V_1(\tau)|^{p}+|V_2(\tau)|^{p})+(|X_1(\tau)|^{p}+|X_2(\tau)|^{p})+\left(\left(\overline{M}_p(\tau,\bm{\mu}^1)\right)+\left(\overline{M}_p(\tau,\bm{\mu}^2)\right)\right)\right],
	\end{align*}
	where in the last inequality we used Young's inequality. Taking the expectation it is not difficult to check that
	\begin{align*}
		\E\left[\left|\delta V(\tau)\right|^{p-1}|\mathfrak{v}[\tau,\bm{\mu}^1](Z_1(\tau))-\mathfrak{v}[\tau,\bm{\mu}^2](Z_2(\tau))|\right] \le CK(1+\overline{M}_p(T,\bm{\mu}^1)+\overline{M}_p(T,\bm{\mu}^2))
	\end{align*}
	where the right-hand side is independent of $\tau$. Hence, we can use Fubini's theorem to get
	 \begin{align}
	 	\E\left[\left|\delta V(s)\right|^p\right]&=p\int_0^s\E\left[\left|\delta V(\tau)\right|^{p-2}(\mathfrak{v}[\tau,\bm{\mu}^1](Z_1(\tau))-\mathfrak{v}[\tau,\bm{\mu}^2](Z_2(\tau)))\cdot \delta V(\tau)\right]d\tau\\
	 	&\le CD\int_0^s\sup_{0 \le r \le \tau}\E[|\delta Z(r)|^p]d\tau,\label{eq:bounduniq1}
	 \end{align}
	 where we used the dissipativity assumption $(\mathfrak{v}_3)$. We also have, clearly,
	 \begin{equation}\label{eq:bounduniq2}
	 	\E\left[|\delta X(s)|^p\right] \le C\int_0^s\E\left[|\delta V(\tau)|^p\right] d\tau \le C\int_0^s\sup_{0 \le r \le \tau}\E\left[|\delta Z(r)|^p\right] d\tau.
	 \end{equation}
	 Combining \eqref{eq:bounduniq1} and \eqref{eq:bounduniq2} and taking the supremum over $[0,t]$ we get
	 \begin{align}
	 	\sup_{0 \le s \le t}\E\left[\left|\delta Z(s)\right|^p\right]\le C(D+1)\int_0^t\sup_{0 \le r \le \tau}\E[|\delta Z(r)|^p]d\tau.
	 \end{align}
	 Since this holds for all $t \in [0,T]$, we can use Gr\"onwall's inequality to achieve \linebreak $\sup_{0 \le s \le t}\E\left[\left|\delta Z(s)\right|^p\right]=0$ for all $t \in [0,T]$, which in turn implies $Z_1=Z_2$ a.s. Hence $\bm{\mu}^1=\bm{\mu}^2$. Observe that this also shows that \eqref{eq:MKVSDEauxdef} admits a unique strong solution.
	 
	 \subsubsection{Constructon of the global solution}\label{step4}
	Now consider any increasing sequence $T_n \uparrow +\infty$ and denote by $\bm{\mu}^{n}$ the local solution of \eqref{probmunl2-intro} on $[0,T_n]$. Let $n_1<n_2$ and observe that $\bm{\mu}^{n_2}$ still solves \eqref{probmunl2-intro} on $[0,T_1]$. Hence, by uniqueness of the local solution, $\bm{\mu}^{n_1}=\bm{\mu}^{n_2}$ on $[0,T_1]$. For such a reason, if we denote $n(t)=\min\{n \ge 0: \ t \ge T_n\}$, we can define the continuous curve of probability measures $\bm{\mu} \in C(\R_0^+; \W_p(\R^{2d}))$ by setting
	 \begin{equation*}
	 	\mu_t=\mu_t^{n(t)}, \ \forall t \in \R_0^+.
	 \end{equation*}
	 It is clear that $\bm{\mu}$ is the unique global solution we are searching for. Analogously, if we set $(X_n,V_n)$ to be the respective solutions of \eqref{eq:MKVSDEauxdef}, a similar argument shows that the stochastic process $(X,V) \in L^p(\Omega;C(\R_0^+; \R^{2d}))$ defined as
	 \begin{equation*}
	 	(X(t),V(t))=(X_{n(t)}(t),V_{n(t)}(t))
	 \end{equation*}
	still solves \eqref{eq:MKVSDEauxdef} and $\bm{\mu}={\rm Law}(X,V)$.

\subsubsection{H\"older continuity of the solution}
For any $t \ge 0$, let $\Sigma_t$ be the $\sigma$-algebra generated by $(X_0,V_0)$ and $\{B(s); \ s \le t\}$. Then it is not difficult to check that $(X,V)$ is a stochastic process adapted to the filtration $\{\Sigma_t\}_{t \ge 0}$. Now let $\bm{\mu}$ be the solution fo \eqref{probmunl2-intro} and $Z=(X,V)$ the solution of \eqref{eq:MKVSDEauxdef} with ${\bm{\mu}}={\rm Law}(X,V)={\rm Law}(Z)$. Let $T>0$, $t \ge 0$ and consider $h>0$ such that $0 \le t < t+h \le T$, without loss of generality. Clearly, $(Z(t+h),Z(t))$ constitute a coupling of $\mu_{t+h}$ and $\mu_t$, hence
\begin{equation*}
	\W_p(\mu_{t+h},\mu_t) \le \left(\E\left[\left|Z(t+h)-Z(t)|^p\right|\right]\right)^{\frac{1}{p}}
\end{equation*}
Set $\widetilde{V}(h)=V(t+h)-V(t)$, $\widetilde{X}(h)=X(t+h)-X(t)$ and $\widetilde{Z}(h)=Z(t+h)-Z(t)$. Observe in particular that
\begin{equation*}
	\widetilde{X}(h)=\int_t^{t+h}V(s)ds
\end{equation*}
and then
\begin{equation}\label{eq:exptildeX}
	\E\left[|\widetilde{X}(h)|^p\right] \le h^{p-1}\E\left[\int_t^{t+h}|V(s)|^p ds\right] \le h^p \overline{M}_p(T;\bm{\mu}) \le C(1+K)e^{KT(1+e^{KT})}h^p, 
\end{equation}
where we used \eqref{eq:Mps2}. As a consequence, we get
\begin{align}
	\W_p(\mu_{t+h},\mu_t) &\le \left(\E\left[\left|\widetilde{Z}(h)|^p\right|\right]\right)^{\frac{1}{p}} \le C \left(\left(\E\left[\left|\widetilde{X}(h)|^p\right|\right]\right)^{\frac{1}{p}}+\left(\E\left[\left|\widetilde{V}(h)|^p\right|\right]\right)^{\frac{1}{p}}\right) \\
	&\le C \left((1+K)^{\frac{1}{p}}e^{\frac{KT(1+e^{KT})}{p}}h+\left(\E\left[\left|\widetilde{V}(h)|^p\right|\right]\right)^{\frac{1}{p}}\right),\label{eq:Holder1}
\end{align}
hence we only need to estimate $\E\left[\left|\widetilde{V}(h)|^p\right|\right]$. We have
\begin{equation*}
	\widetilde{V}(h)=\int_0^{h}\widetilde{\mathfrak{v}}[s,\bm{\mu}](\widetilde{X}(s)+X(t),\widetilde{V}(s)+V(t))\, ds+\sqrt{2\sigma}\widetilde{B}(h).
\end{equation*}
Where $\widetilde{\mathfrak{v}}[h,\bm{\mu}](z)=\mathfrak{v}[t+h,\bm{\mu}](z)$ and $\widetilde{B}(h)=B(t+h)-B(t)$. Fix $\varepsilon>0$ and set $H_{p,\varepsilon}(x)=(|x|^2+\varepsilon)^{\frac{p}{2}}$. Once we recall that $\widetilde{B}$ is still a Brownian motion and $H_{p,\varepsilon} \in C^2(\R^{d})$, we can use It\^{o}'s formula to write
\begin{align*}
	\begin{split}
	&H_{p,\varepsilon}(\widetilde{V}(h))=\varepsilon^{\frac{p}{2}}\\
	&+\int_0^{h}\left(\nabla H_{p,\varepsilon}(\widetilde{V}(\tau)) \cdot \widetilde{\mathfrak{v}}[\tau,\bm{\mu}](\widetilde{X}(\tau)+X(t),\widetilde{V}(\tau)+V(t))+\sigma \Delta H_{p,\varepsilon}(\widetilde{V}(\tau))\right)d\tau\\
	&\quad +\sqrt{2\sigma} \int_0^h \nabla H_{p,\varepsilon}(\widetilde{V}(\tau)) d\widetilde{B}(\tau).	
	\end{split}
\end{align*}
Next, for any $h \ge 0$, let $\widetilde{\Sigma}^t_h$ be the $\sigma$-algebra generated by $\{(X(s),V(s)), \ s \in [0,t]\}$ and $\{\widetilde{B}(\tau); \ \tau \in [0,h]\}$. Then $(\widetilde{X},\widetilde{V})$ is adapted to the filtration $\{\widetilde{\Sigma}^t_h\}_{h \ge 0}$ and then, for any $N \in \N$, the random variable
\begin{equation*}
	\mathcal{T}_N(h):=\min\{\inf \{\tau>0: \ |\widetilde{V}(\tau)| \ge N\},h\}
\end{equation*}
is a $\{\widetilde{\Sigma}^t_h\}_{h \ge 0}$-stopping time. Observe further that
\begin{align}
	&H_{p,\varepsilon}(\widetilde{V}(\mathcal{T}_N(h)))=\varepsilon^{\frac{p}{2}}\\
	&+\int_0^{\mathcal{T}_N(h)}\left(\nabla H_{p,\varepsilon}(\widetilde{V}(\tau)) \cdot \widetilde{\mathfrak{v}}[\tau,\bm{\mu}](\widetilde{X}(\tau)+X(t),\widetilde{V}(\tau)+V(t))+\sigma \Delta H_{p,\varepsilon}(\widetilde{V}(\tau))\right)d\tau\\
	&\quad +\sqrt{2\sigma} \int_0^{\mathcal{T}_N(h)} \nabla H_{p,\varepsilon}(\widetilde{V}(\min\{\tau,\mathcal{T}_N(h)\})) d\widetilde{B}(\tau).\label{eq:preexp}
\end{align}
Before taking the expectation on both sides, observe that 
\begin{equation}\label{eq:nablaH}
\nabla H_{p,\varepsilon}(x)=p(|x|^2+\varepsilon)^{\frac{p-2}{2}}x	
\end{equation}
and then
\begin{align*}
	|\nabla H_{p,\varepsilon}(\widetilde{V}(\min\{\tau,\mathcal{T}_N(h)\}))| \le pN(N^2+\varepsilon)^{\frac{p-2}{2}}.
\end{align*}
Thus the process $\int_0^{h}\nabla H_{p,\varepsilon}(\widetilde{V}(\min\{\tau,\mathcal{T}_N(h)\}))d\widetilde{B}(\tau)$ is a martingale \cite[Corollary 3.2.6]{oksendal2013stochastic}. Hence, since $\mathcal{T}_N(h)$ is a bounded stopping time, we can use the optional stopping theorem \cite[Theorem II.3.2]{revuzyor} to get
\begin{equation*}
	\E\left[\int_0^{\mathcal{T}_N(h)}\nabla H_{p,\varepsilon}(\widetilde{V}(\min\{\tau,\mathcal{T}_N(h)\}))d\widetilde{B}(\tau)\right]=0.
\end{equation*}
Taking the expectation in \eqref{eq:preexp} and then the absolute value we achieve
\begin{align}
	&\E\left[H_{p,\varepsilon}(\widetilde{V}(\mathcal{T}_N(h)))\right] \le \varepsilon^{\frac{p}{2}}\\
	&+\E\left[\int_0^{h}\left(\left|\nabla H_{p,\varepsilon}(\widetilde{V}(\tau)) \cdot \widetilde{\mathfrak{v}}[\tau,\bm{\mu}](\widetilde{X}(\tau)+X(t),\widetilde{V}(\tau)+V(t))\right|+\sigma \left|\Delta H_{p,\varepsilon}(\widetilde{V}(\tau))\right|\right)d\tau\right].\label{eq:exppost}
\end{align}
Now we use $(\mathfrak{v}_1)$ and \eqref{eq:nablaH} to get
\begin{align*}
	&\left|\nabla H_{p,\varepsilon}(\widetilde{V}(\tau)) \cdot \widetilde{\mathfrak{v}}[\tau,\bm{\mu}](\widetilde{X}(\tau)+X(t),\widetilde{V}(\tau)+V(t))\right| \\
	&\le pK(|\widetilde{V}(\tau)|^2+\varepsilon)^{\frac{p-2}{2}}|\widetilde{V}(\tau)|(1+|\widetilde{X}(\tau)+X(t)|^{\frac{\beta}{3}}+|\widetilde{V}(\tau)+V(t)|^{\beta}+\left(\overline{M}_p(t+\tau;\bm{\mu})\right)^{\frac{1}{p}})\\
	&\le CK(|\widetilde{V}(\tau)|^2+\varepsilon)^{\frac{p-1}{2}}(1+|\widetilde{X}(\tau)|^{\frac{\beta}{3}}+|X(t)|^{\frac{\beta}{3}}+|\widetilde{V}(\tau)|^\beta+|V(t)|^{\beta}+\left(\overline{M}_p(t+\tau;\bm{\mu})\right)^{\frac{1}{p}})\\
	&\le CKH_{p,\varepsilon}(\widetilde{V}(\tau))+CK(|\widetilde{V}(\tau)|^2+\varepsilon)^{\frac{p-1}{2}}(1+|\widetilde{X}(\tau)|^{\frac{\beta}{3}}+|X(t)|^{\frac{\beta}{3}}+|V(t)|^{\beta}+\left(\overline{M}_p(t+\tau;\bm{\mu})\right)^{\frac{1}{p}}).
\end{align*}
To handle the second summand, we use Young's inequality with exponent $p$, and then we achieve
\begin{align}
	&\left|\nabla H_{p,\varepsilon}(\widetilde{V}(\tau)) \cdot \widetilde{\mathfrak{v}}[\tau,\bm{\mu}](\widetilde{X}(\tau)+X(t),\widetilde{V}(\tau)+V(t))\right| \\
	&\quad \le CKH_{p,\varepsilon}(\widetilde{V}(\tau))+CK^p(1+|\widetilde{X}(\tau)|^{\frac{\beta p }{3}}+|X(t)|^{\frac{\beta p }{3}}+|V(t)|^{\beta p}+\overline{M}_p(t+\tau;\bm{\mu}))\\
	&\quad \le CKH_{p,\varepsilon}(\widetilde{V}(\tau))+CK(1+|\widetilde{X}(\tau)|^{p}+|Z(t)|^{p}+\overline{M}_p(T;\bm{\mu})). 
\end{align}
Taking the expectation this leads to
\begin{align}
	&\E\left[\left|\nabla H_{p,\varepsilon}(\widetilde{V}(\tau)) \cdot \widetilde{\mathfrak{v}}[\tau,\bm{\mu}](\widetilde{X}(\tau)+X(t),\widetilde{V}(\tau)+V(t))\right|\right] \\
	&\qquad \le CK\E\left[H_{p,\varepsilon}(\widetilde{V}(\tau))\right]+CK(1+(1+K)e^{KT(1+e^{KT})}), \label{eq:nablaH2}
\end{align}
where we also used \eqref{eq:Mps2} and \eqref{eq:exptildeX}. Plugging \eqref{eq:nablaH2} into \eqref{eq:exppost} we have
\begin{align}
	\E\left[H_{p,\varepsilon}(\widetilde{V}(\mathcal{T}_N(h)))\right] &\le \varepsilon^{\frac{p}{2}}+CK\int_0^{h}\E\left[H_{p,\varepsilon}(\widetilde{V}(\tau))\right]d\tau\\
	&+CK(1+(1+K)e^{KT(1+e^{KT})})h+\sigma\int_0^h \E\left[\left|\Delta H_{p,\varepsilon}(\widetilde{V}(\tau))\right|d\tau\right].\label{eq:exppost2}
\end{align}
Next, observe that since $V$ is a global solution of \eqref{eq:MKVSDEauxdef}, it cannot blow up in finite time, hence it must hold $\lim_{N \to \infty}\cT_N(h)=h$. Furthermore, for almost all $\omega \in \Omega$, $V$ (and then $\widetilde{V}$) is continuous. Hence, by Fatou's Lemma, taking the limit inferior in \eqref{eq:exppost2} we get
\begin{align}
	\E\left[H_{p,\varepsilon}(\widetilde{V}(h))\right] &\le \varepsilon^{\frac{p}{2}}+CK\int_0^{h}\E\left[H_{p,\varepsilon}(\widetilde{V}(\tau))\right]d\tau\\
	&+CK(1+(1+K)e^{KT(1+e^{KT})})h+\sigma\int_0^h \E\left[\left|\Delta H_{p,\varepsilon}(\widetilde{V}(\tau))\right|d\tau\right].\label{eq:exppost3}
\end{align}
Furthermore, notice that
\begin{equation*}
	\Delta H_{p,\varepsilon}(x)=p(|x|^2+\varepsilon)^{\frac{p-4}{2}}[(p+d-2)|x|^2+d\varepsilon],
\end{equation*}
hence
\begin{equation*}
	\left|\Delta H_{p,\varepsilon}(\widetilde{V}(h))\right| \le C(|\widetilde{V}(h)|^2+\varepsilon)^{\frac{p-2}{2}}
\end{equation*}
Now we have to distinguish among two cases. If $p>2$ we get, by Young's inequality with exponent $\frac{p}{p-2}$,
\begin{equation}\label{eq:DeltaHp}
	\left|\Delta H_{p,\varepsilon}(\widetilde{V}(h))\right| \le C(1+H_{p,\varepsilon}(|\widetilde{V}(h)|)).
\end{equation}
Hence, using \eqref{eq:DeltaHp} into \eqref{eq:exppost3}, we achieve
\begin{align}
	\E\left[H_{p,\varepsilon}(\widetilde{V}(h))\right] &\le \varepsilon^{\frac{p}{2}}+C(1+K)\int_0^{h}\E\left[H_{p,\varepsilon}(\widetilde{V}(\tau))\right]d\tau\\
	&\quad +CK(1+(1+K)e^{KT(1+e^{KT})})h.
\end{align}
Since $h \in [0,T-t]$ is arbitrary, we can use Gr\"onwall's inequality to get
\begin{align}
	\E\left[H_{p,\varepsilon}(\widetilde{V}(h))\right] \le \left(\varepsilon^{\frac{p}{2}}+CK(1+(1+K)e^{KT(1+e^{KT})})h\right)e^{C(1+K)h}.
\end{align}
It remains to take the limit as $\varepsilon \to 0$ to finally achieve, by the dominated convergence theorem,
\begin{align}
	\E\left[|\widetilde{V}(h)|^p\right] \le CK(1+(1+K)e^{KT(1+e^{KT})})e^{C(1+K)}h.
\end{align}
Using the latter inequality into \eqref{eq:Holder1}, we have the local $1/p$-H\"older continuity of $\bm{\mu}$ and the third bound in \eqref{eq:momentbounds}.

If $p \le 2$ instead we use
\begin{equation}\label{eq:DeltaHp2}
	\left|\Delta H_{p,\varepsilon}(\widetilde{V}(h))\right| \le C\varepsilon^{\frac{p-2}{2}}
\end{equation}
and then
\begin{align}
	\E\left[H_{p,\varepsilon}(\widetilde{V}(h))\right] &\le \varepsilon^{\frac{p}{2}}+CK\int_0^{h}\E\left[H_{p,\varepsilon}(\widetilde{V}(\tau))\right]d\tau\\
	&\quad +CK(1+(1+K)e^{KT(1+e^{KT})})h+C\varepsilon^{\frac{p-2}{2}}h.
\end{align}
Again, by Gr\"onwall's inequality we have
\begin{align}
	\E[|\widetilde{V}(h)|^p] \le \E\left[H_{p,\varepsilon}(\widetilde{V}(h))\right] &\le \left(\varepsilon^{\frac{p}{2}}+CK(1+(1+K)e^{KT(1+e^{KT})})h+C\varepsilon^{\frac{p-2}{2}}h\right)e^{CK}.
\end{align}
The latter inequality holds for any $\varepsilon>0$ and any $h \in [0,T-t]$, hence, we can take $\varepsilon=h$ so that
\begin{align}
	\E[|\widetilde{V}(h)|^p] \le \left(1+CK(1+(1+K)e^{KT(1+e^{KT})})\right)h^{\frac{p}{2}}e^{CK}.
\end{align}
Again, we get the local $1/2$-H\"older continuity of $\bm{\mu}$ and the third bound in \eqref{eq:momentbounds} by \eqref{eq:Holder1}.

\subsubsection{Higher moment estimate}
It remains to prove the second bound in \eqref{eq:momentbounds}. Again, consider the Young function $\Phi$ such that $\overline{\mu} \in \W_{\Phi_p}(\R^{2d})$. Observe that, by $(\mathfrak{v}_1)$ and \eqref{eq:Mps2}, we have, for all $t \in [0,T]$ and $z \in \R^{2d}$
\begin{equation*}
	\mathfrak{v}[t,\bm{\mu}](z) \le C(1+K)e^{KT(1+e^{KT})}(1+|x|^{\frac{\beta}{3}}+|v|^\beta).
\end{equation*}
Hence we can use \eqref{eq:Phi1} to get
\begin{equation*}
	\overline{M}_{\widetilde{\Phi}_p(\cdot;K)}(T,\bm{\mu}) \le C(1+(1+K)e^{KT(1+e^{KT})}),
\end{equation*}
where
\begin{equation*}
	\widetilde{\Phi}(r;K)=\Phi\left(Ce^{-C(1+K)e^{KT(1+e^{KT})}}r\right).
\end{equation*}
This ends the proof.
	
	\qed

\begin{rem}
	Notice that, to prove the existence of local solutions, Assumption $(\mathfrak{v}_3)$ is not needed. Furthermore, for uniqueness, one only needs that Assumption 
	$(\mathfrak{v}_3)$ holds on solutions of \eqref{eq:MKVSDEauxdef}.
\end{rem}

\subsection{Proof of Corollary \ref{cor:stab}}
From now on, we fix the time horizon $T>0$.
	Let $Z_j=(X_j,V_j)$ and $Z=(X,V)$ be the solutions of \eqref{eq:MKVSDEauxdef} associated with $\mathfrak{v}_j$ and $\mathfrak{v}$ respectively. Then, clearly, $\mu_t^j={\rm Law}(X_j(t),V_j(t))$ and $\mu_t={\rm Law}(X(t),V(t))$, hence 
	\begin{align*}
		\sup_{0 \le t \le T}\left(\W_p(\mu_t^j,\mu_t)\right)^p \le C\sup_{0 \le t \le T}\left(\E[|Z_j(t)-Z(t)|^p]\right).
	\end{align*}
	Let $0 \le t \le T$. Once we notice that $\delta V_j(t):=V_j(t)-V(t)$ is absolutely continuous, we can use the chain rule to get for any $s \in [0,t]$
	\begin{align*}
		\E\left[\left|\delta V_j(s)\right|^p\right]=p\E\left[\int_0^s\left|\delta V_j(\tau)\right|^{p-2}(\mathfrak{v}_j[\tau,\bm{\mu}^j](Z_j(\tau))-\mathfrak{v}[\tau,\bm{\mu}](Z(\tau)))\cdot \delta V_j(\tau)d\tau\right].
	\end{align*}
	Now notice that, by Young's inequality and $(\mathfrak{v}_1)$ we get
	\begin{multline*}
		\left|\delta V_j(\tau)\right|^{p-1}\left|\mathfrak{v}_j[\tau,\bm{\mu}^j](Z_j(\tau))-\mathfrak{v}[\tau,\bm{\mu}](Z(\tau))\right|\\
		 \le C(K+1)\left(1+\left|Z_j(\tau)\right|^{p}+|Z(\tau)|^p+\overline{M}_p(T;\bm{\mu}_j)+\overline{M}_p(T;\bm{\mu}_j)\right),
	\end{multline*}
	and then
	\begin{align*}
		\int_{0}^{s} &\E\left[\left|\delta V_j(\tau)\right|^{p-1}\left|\mathfrak{v}_j[\tau,\bm{\mu}^j](Z_j(\tau))-\mathfrak{v}[\tau,\bm{\mu}](Z(\tau))\right|\right]d\tau \\
		&\le C(K+1)T\left(1+\overline{M}_p(T;\bm{\mu}_j)+\overline{M}_p(T,\bm{\mu})\right)\\
		&\le C(K+1)T\left(1+H_T(K)\right),
	\end{align*}
	where we also used \eqref{eq:momentbounds}. Hence we can employ Fubini's theorem to get
	\begin{align}
		\E\left[\left|\delta V_j(s)\right|^p\right]&=p\int_0^s\E\left[\left|\delta V_j(\tau)\right|^{p-2}(\mathfrak{v}_j[\tau,\bm{\mu}^j](Z_j(\tau))-\mathfrak{v}[\tau,\bm{\mu}](Z(\tau)))\cdot \delta V_j(\tau)\right]d\tau\\
		&=p\int_0^s\E\left[\left|\delta V_j(\tau)\right|^{p-2}(\mathfrak{v}_j[\tau,\bm{\mu}^j](Z_j(\tau))-\mathfrak{v}_j[\tau,\bm{\mu}](Z(\tau)))\cdot \delta V_j(\tau)\right]d\tau\\
		&\quad +p\int_0^s\E\left[\left|\delta V_j(\tau)\right|^{p-2}(\mathfrak{v}_j[\tau,\bm{\mu}](Z(\tau))-\mathfrak{v}[\tau,\bm{\mu}](Z(\tau)))\cdot \delta V_j(\tau)\right]d\tau\\
		&=I_1+I_2.\label{eq:stab1}
	\end{align}
	For the first term, we use $(\mathfrak{v}_3)$ to get
	\begin{align}\label{eq:stab2}
		I_1 \le D\int_0^s \sup_{0 \le r \le \tau}\E[|\delta Z_j(r)|^p]d\tau,
	\end{align}
	where $\delta Z_j(r):=Z_j(r)-Z(r)$. Next, we move to $I_2$. To handle it we use Young's inequality as follows
	\begin{align}
		I_2 &\le p\int_0^s\E\left[\left|\delta V_j(\tau)\right|^{p-1}|\mathfrak{v}_j[\tau,\bm{\mu}](Z(\tau))-\mathfrak{v}[\tau,\bm{\mu}](Z(\tau))|\right]d\tau\\
		&\le C\left(\int_0^s\sup_{0 \le r \le \tau}\E\left[\left|\delta Z_j(r)\right|^{p}\right]d\tau+\int_0^s\E\left[|\mathfrak{v}_j[\tau,\bm{\mu}](Z(\tau))-\mathfrak{v}[\tau,\bm{\mu}](Z(\tau))|^p\right]d\tau\right).\label{eq:stab3}
	\end{align}
	We denote
	\begin{equation*}
		\mathcal{R}_j(s)=\int_0^s\E\left[|\mathfrak{v}_j[\tau,\bm{\mu}](Z(\tau))-\mathfrak{v}[\tau,\bm{\mu}](Z(\tau))|^p\right]d\tau
	\end{equation*}
	and we notice that it is a nondecreasing function. Hence we get, combining \eqref{eq:stab1}, \eqref{eq:stab2} and \eqref{eq:stab3}
	\begin{align}
		\E\left[\left|\delta V_j(s)\right|^p\right]&\le C(1+D)\int_0^s\sup_{0 \le r \le \tau}\E\left[\left|\delta Z_j(r)\right|^{p}\right]d\tau+C\mathcal{R}_j(s).\label{eq:stab4}
	\end{align}
	Next, we observe that
	\begin{equation}\label{eq:stabi5}
		\E\left[|\delta X_j(s)|^p\right] \le T^{p-1}\int_0^s \E\left[|\delta V_j(\tau)|^p\right]d\tau \le T^{p-1}\int_0^s \sup_{0 \le r \le \tau}\E\left[|\delta Z_j(r)|^p\right]d\tau,
	\end{equation}
	where $\delta X_j(t):=X_j(t)-X(t)$. Combining \eqref{eq:stab4} and \eqref{eq:stabi5} and taking the supremum on $[0,t]$ we have
	\begin{align}
		\sup_{0 \le s \le t}\E\left[\left|\delta Z_j(s)\right|^p\right]&\le C(1+D)\int_0^t\sup_{0 \le r \le \tau}\E\left[\left|\delta Z_j(r)\right|^{p}\right]d\tau+C\mathcal{R}_j(t),
	\end{align}
	that, by Gr\"onwall's inequality, implies
	\begin{align}
		\sup_{0 \le s \le T}\E\left[\left|\delta Z_j(s)\right|^p\right]&\le e^{C(1+D)}\mathcal{R}_j(T)
	\end{align}
	and then 
	\begin{align*}
		\sup_{0 \le t \le T}\left(\W_p(\mu_t^j,\mu_t)\right)^p \le e^{C(1+D)}\mathcal{R}_j(T).
	\end{align*}
	It remains to show that the right-hand side converges to $0$. To do this, it is sufficient to observe that, by $(\mathfrak{v}_1)$ and Young's inequality
	\begin{equation*}
		\left|\mathfrak{v}_j[\tau,\bm{\mu}](Z(\tau))-\mathfrak{v}_j[\tau,\bm{\mu}](Z(\tau))\right| \le CK\left(1+|Z(\tau)|^p+\overline{M}_p(T,\bm{\mu})\right),
	\end{equation*}
	where
	\begin{equation*}
		\int_0^T\E\left[CK\left(1+|Z(\tau)|^p+\overline{M}_p(T,\bm{\mu})\right)\right]d\tau \le CK(1+\overline{M}_p(T,\bm{\mu}))<\infty.
	\end{equation*}
	Hence, by the dominated convergence theorem and the fact that $\bm{\mu}$ satisfies \eqref{eq:momentbounds}, \linebreak  $\lim_{j \to \infty}\mathcal{R}_j(T)=0$. This concludes the proof. \hfill $\square$


\section{The PDE-ODE system and a related optimal control problem}\label{sec:optimal}
In this section we apply the previous results to a mean field sparse optimal control problem. Precisely, we consider the following PDE-ODE system \cite{albi2016invisible,albi2017mean}
\begin{equation}\label{PDEODE}
	\begin{cases}
		\partial_t \mu_t=-v\cdot \nabla_x \mu_t+\sigma\Delta_v\mu_t-{\rm div}_v((\mathfrak{v}[t,\bm{\mu}](z)\\
				\qquad \qquad\quad \qquad \, \, \,+\mathfrak{w}[t,\bm{H}](z))\mu_t)) & (t,x,v) \in (0,T] \times \R^{2d} \\
		\dot{\mathbf{Y}}(t)=\mathbf{W}(t)=F[t,\bm{\mu}](\bm{Y})+\mathbf{u}(t,\bm{\mu}) & t \in (0,T]\\
		\mu_0=\overline{\mu}, \qquad \mathbf{Y}(0)=\overline{\mathbf{Y}},
	\end{cases}
\end{equation}
where $\bm{\mu} \in C([0,T];\W_p(\R^{2d}))$, $\bm{H}=(\mathbf{Y},\mathbf{W}):[0,T] \to \R^{2m}$, $\mathfrak{v}:[0,T]\times C([0,T];\W_p(\R^{2d}))\times \R^{2d} \to \R$, $\mathfrak{w}:[0,T] \times C([0,T];\R^{2m}) \times \R^{2d} \to \R^{d}$, $F:[0,T] \times C([0,T];\W_p(\R^{2d})) \times C([0,T];\R^{m}) \to \R^m$, $\mathbf{u}: [0,T] \times C([0,T];\W_p(\R^{2d})) \to \R^m$, $\overline{\mu} \in \W_p(\R^{2d})$, and $\overline{\bm{H}}=(\overline{\bm{Y}},\overline{\bm{W}}) \in \R^{2m}$.
\begin{defn}\label{def:solnlPDEODE}
	We say that $(\bm{\mu},\bm{Y}) \in C([0,T];\W_p(\R^{2d}) \times \R^{2m})$ is a solution of \eqref{PDEODE} if and only if $\bm{\mu}$ is solution of
	\begin{equation*}
		\begin{cases}
			\partial_t \mu_t=-v\cdot \nabla_x \mu_t+\sigma\Delta_v\mu_t-{\rm div}_v((\mathfrak{v}[t,\bm{\mu}](z)+\mathfrak{w}[t,\bm{H}](z))\mu_t)) & (t,x,v) \in (0,T] \times \R^{2d} \\
			\mu_0=\overline{\mu},
		\end{cases}
	\end{equation*}
	in the sense of Definition \ref{solutionnl} and  $\bm{Y}$ satisfies  
	\begin{align*}
		\mathbf{Y}(t)=\overline{\bm{Y}}+\int_0^t\left(F[s,\bm{\mu}](\bm{Y})+\mathbf{u}(s,\bm{\mu})\right)\, ds && t \in (0,T]
	\end{align*}
\end{defn}
From now on, we assume that $\mathfrak{v}$ satisfies Assumptions \ref{ass:v}. Furthermore we consider the following assumptions of $\mathfrak{w}$ and $F$:

\begin{ProblemSpecBox}{Assumptions on $\mathfrak{w}$: $(\mathfrak{w})$}{$(\mathfrak{w})$}\label{ass:w}
	\begin{itemize}
		\item[$(\mathfrak{w}_0)$] $\mathfrak{w}:[0,T]\times C([0,T];\R^{2m})\times \R^{2d} \to \R^d$ is a Carath\'eodory map, i.e., it is measurable in the variable $t \in [0,T]$ and continuous in $(\bm{H},z) \in C([0,T];\R^{2m})\times \R^{2d}$.
		\item[$(\mathfrak{w}_1)$] There exists a constant $K_{\mathfrak{w}}>0$ such that for any $t \in [0,T]$, $z \in \R^{2d}$ and $\bm{H} \in C([0,T];\R^{2m})$ it holds
		\begin{equation*}
			\left|\mathfrak{w}[t,\bm{H}](z)\right| \le K_{\mathfrak{w}}(1+\sup_{0 \le s \le t}|\bm{H}(s)|+|x|^{\frac{\beta}{3}}+|v|^\beta),
		\end{equation*}
		where $\beta$ is the exponent in $(\mathfrak{v}_1)$.
		\item[$(\mathfrak{w}_2)$] There exists a constant $L_{\mathfrak{w}}$ such that for any $t \in [0,T]$, $\bm{H}^j \in C([0,T];\R^{2m})$ and $z_j \in \R^{2d}$, $j=1,2$, it holds
		\begin{equation*}
			\left|\mathfrak{w}[t,\bm{H}^1](z_1)-\mathfrak{w}[t,\bm{H}^2](z_2)\right| \le L_{\mathfrak{w}}\left(\sup_{0 \le s \le t}|\bm{H}^1-\bm{H}^2|+|z_1-z_2|\right).
		\end{equation*}
	\end{itemize}
\end{ProblemSpecBox}

\begin{ProblemSpecBox}{Assumptions on $F$: $(F)$}{$(F)$}\label{ass:F}
	\begin{itemize}
		\item[$(F_0)$] $F:[0,T]\times C([0,T];\W_p(\R^{2d})) \times C([0,T];\R^{m}) \to \R^m$ is a Carath\'eodory map, i.e. is measurable in $t \in [0,T]$ and continuous in $(\bm{\mu},\bm{Y}) \in C([0,T];\W_p(\R^{2d})) \times C([0,T];\R^{m})$.
		\item[$(F_1)$] There exists a constant $K_F>0$ such that
		\begin{equation*}
			|F(t,\bm{\mu})(\bm{Y})| \le K_F\left(1+\sup_{0 \le s \le t}|\bm{Y}(s)|+\left(\overline{M}_p(T;\bm{\mu})\right)^{\frac{1}{p}}\right),
		\end{equation*}
		for all $\bm{\mu} \in C([0,T];\W_p(\R^{2d}))$, $\bm{Y} \in C([0,T];\R^{m})$ and $t \in [0,T]$.
		\item[$(F_2)$] There exists a constant $L_F>0$ such that
		\begin{equation*}
			|F[t,\bm{\mu}](\bm{Y}^1)-F[t,{\bm{\nu}}](\bm{Y}^2)| \le L_F\left(\sup_{0 \le s \le t}|\bm{Y}^1(s)-\bm{Y}^2(s)|+\sup_{0 \le s \le t}\W_p(\mu_s,\nu_s)\right),
		\end{equation*}
		for all $\bm{\mu},\bm{\nu} \in C([0,T];\W_p(\R^{2d}))$, $\bm{Y}^j \in C([0,T];\R^{m})$, $j=1,2$ and $t \in [0,T]$.
	\end{itemize}
\end{ProblemSpecBox}

Lastly, concerning the controls $\bm{u}=(u_j)_{j=0,\dots,m}$, we fix two constants $M_{\bm{u}}$ and $L_{\bm{u}}$ and we denote by $\cA$ the set of admissible controls, characterized as follows:
\begin{ProblemSpecBox}{Assumptions on $\bm{u} \in \mathcal{A}$: $(\mathcal{A})$}{$(\mathcal{A})$}\label{ass:A}
	\begin{itemize}
		\item[$(\mathcal{A}_0)$] $\bm{u}:[0,T]\times C([0,T];\W_p(\R^{2d})) \to \R^m$ is a Carath\'eodory map, i.e., it is measurable in $t \in [0,T]$ and continuous in $\bm{\mu} \in C([0,T];\W_p(\R^{2d}))$.
		\item[$(\mathcal{A}_1)$] For all $t \in [0,T]$ it holds
		\begin{equation*}
			|\bm{u}(t,\bm{\delta}_0)| \le M_{\bm{u}},
		\end{equation*}
		where $\bm{\delta}_0 \in \W_p(C([0,T];\R^{2d}))$ is the Dirac delta measure concentrated in the constant function $0 \in C([0,T];\R^{2d})$.
		\item[$(\cA_2)$] For all $\bm{\mu},\bm{\nu} \in C([0,T];\W_p(\R^{2d}))$, $j=1,\dots,m$ and $t \in [0,T]$, it holds
		\begin{equation*}
			|u_j(t,\bm{\mu})-u_j(t,\bm{\nu})| \le \frac{L_{\bm{u}}}{m}\sup_{0 \le s \le t}\W_p(\mu_s,\nu_s).
		\end{equation*}
	\end{itemize}
\end{ProblemSpecBox}

\medskip

By Assumptions $(\cA_1)$ and $(\cA_2)$, for all $\bm{u} \in \cA$ it holds
\begin{equation*}
	\left|\bm{u}(t,\bm{\mu})\right| \le \left|\bm{u}(t,\bm{\mu})-\bm{u}(t,\bm{\delta}_0)\right|+\left|\bm{u}(t,\bm{\delta}_0)\right| \le L_{\bm{u}}(\overline{M}_p(t;\bm{\mu}))^{\frac{1}{p}}+M_0,
\end{equation*}
where we used the fact that $(\overline{M}_p(t;\bm{\mu}))^{\frac{1}{p}}=\sup_{0 \le s \le t}\W_p(\mu_s,\delta_0)$. Hence, if we set
\begin{equation*}
	K_{\bm{u}}=\max\{L_{\bm{u}},M_{\bm{u}}\},
\end{equation*}
then we have that for all $\bm{\mu} \in C([0,T];\W_p(\R^{2d}))$ and all $t \in [0,T]$ it holds
\begin{equation}\label{eq:sublin}
	\left|\bm{u}(t,\bm{\mu})\right| \le K_{\bm{u}}(1+(\overline{M}_p(t;\bm{\mu}))^{\frac{1}{p}}).
\end{equation}

\subsection{Well-posedness of the PDE-ODE system \eqref{PDEODE}}
Before setting the control problem, let us show that the system \eqref{PDEODE} is well-posed.
Our proof relies on some preliminary results that are proved in Appendix \ref{wp-pre}.
First of all, for any fixed $T>0$, $\bm{u} \in \cA$ and $\bm{\mu} \in C([0,T];\W_p(\R^{2d}))$ we introduce the system
\begin{equation}\label{PDEODE2-sec}
		\begin{cases}
			\dot{\mathbf{Y}}_{\bm{\mu}}(t)=\mathbf{W}_{\bm{\mu}}(t)=F[t,\bm{\mu}](\bm{Y}_{\bm{\mu}})+\mathbf{u}(t,\bm{\mu}) & t \in (0,T]\\
			\mu_0=\overline{\mu}, \qquad \mathbf{Y}_{\bm{\mu}}(0)=\overline{\mathbf{Y}},
		\end{cases}
\end{equation}
which admits a unique solution $\bm{Y}_{\mu}$ thanks to Lemma \ref{lem:PDEODEWell1}.
Then we define the map $\mathcal{S}:C([0,T];\W_p(\R^{2d})) \times \cA \times [0,T] \to \R^{2m}$ as follows:
\begin{equation*}
	\mathcal{S}[\bm{\mu},\bm{u}](t)=\bm{H}_{\bm{\mu}}(t)=(\bm{Y}_{\bm{\mu}}(t),\bm{W}_{\bm{\mu}}(t)),
\end{equation*}
where $\bm{Y}_{\bm{\mu}}$ is the unique solution of \eqref{PDEODE2-sec}. For any fixed $\bm{u} \in \cA$, we also define the map
\begin{equation*}
	G_{\bm{u}}:[0,T]\times C([0,T];\W_p(\R^{2d})) \times \R^{2d} \to \R^{d}
\end{equation*}
as follows:
\begin{equation*}
G_{\bm{u}}[t,\bm{\mu}](z)=\mathfrak{v}[t,\bm{\mu}](z)+\mathfrak{w}[t,\mathcal{S}[\bm{\mu},\bm{u}]](z),
\end{equation*}
which satisfies assumptions $(\mathfrak{v}_0)$, $(\mathfrak{v}_1)$, $(\mathfrak{v}_2)$ and $(\mathfrak{v}_3)$ thanks to Lemma \ref{lem:PDEODEWell2}.

Once this is established, we set
\begin{equation*}
	K_{G}:=K_{\mathfrak{v}}+K_{\mathfrak{w}}(1+K_F)(1+C_1+K_F+K_{\bm{u}}),
\end{equation*}
where $C_1$ is defined in Lemma \ref{lem:PDEODEWell1}, and we denote by
\begin{equation*}
	\mathcal{K}:=\left\{\bm{\mu} \in C([0,T];\W_p(\R^{2d})): \ \overline{M}_p(T,\bm{\mu})+\overline{M}_{\widetilde{\Phi}_p(\cdot;K_{G})}(T;\bm{\mu})+\sup_{\substack{0 \le s,t \le T \\ t \not = s }}\frac{\W_p(\mu_t,\mu_s)}{|t-s|^{\gamma_p}} \le \mathcal{C}(K_G)\right\},
\end{equation*}
where $\widetilde{\Phi}$ and $\mathcal{C}$ are the functions defined in Theorem \ref{thm:main1} and $\gamma_p=\frac{1}{\max\{2,p\}}$. We also set
\begin{equation*}
	\mathcal{C}_1(K):=(C_1+C_3)\left(1+(\mathcal{C}(K_G))^{\frac{1}{p}}\right),
\end{equation*}
where $C_1$ and $C_3$ are defined in Lemma \ref{lem:PDEODEWell1}, and
\begin{equation*}
	\mathcal{K}_1:=\left\{\bm{Y} \in C([0,T];\R^{m}): \ \sup_{0 \le s \le T}|\bm{Y}(s)| +\sup_{\substack{0 \le t,s \le T \\ t \not = s }}\frac{|\bm{Y}(t)-\bm{Y}(s)|}{|t-s|} \le \mathcal{C}_1(K_G)\right\}.
\end{equation*}
By Lemma \ref{lem:PDEODEWell3} we know both $\mathcal{K}$ and $\mathcal{K}_1$ are compact subsets and $\cK^\prime=\cK \times \cK_1$. 
Then we are in a position to prove \eqref{PDEODE} is well-posed.

\begin{thm}\label{thm:PDEODEWell}
	For any fixed $\bm{u} \in \cA$ there exists a unique solution $(\bm{\mu},\bm{Y}) \in C([0,T];\W_p(\R^{2d}))\times C([0,T];\R^{m})$ of \eqref{PDEODE}. In particular, if $B$ is a $d$-dimensional Brownian motion and $(X_0,V_0) \in L^p(\Omega;\R^{2d})$ is independent of $B$ and $\overline{\mu}={\rm Law}(X_0,V_0)$, then $\bm{\mu}={\rm Law}(X,V)$, where $(X,V)$ is the unique global strong solution in $L^p(\Omega;C(\R_0^+;\R^{2d}))$ of
		\begin{equation}\label{eq:MKVSDEaux2}
			\begin{cases}
				dX(t)=V(t)dt & t \in [0,T]\\
				dV(t)=\left(\mathfrak{v}[t,\bm{\mu}](X(t),V(t))+\mathfrak{w}[t,\bm{H}](X(t),V(t))\right)dt+\sqrt{2\sigma} dB(t) & t \in [0,T]\\
				\dot{\bm{Y}}(t)=\bm{W}(t)=F[t,\bm{\mu}](\bm{Y})+\bm{u}(t,\bm{\mu}) & t \in (0,T]\\
				X(0)=X_0, \qquad V(0)=V_0, \qquad \bm{\mu}={\rm Law}(X,V)\\
				\bm{Y}(0)=\overline{\bm{Y}}
			\end{cases}
		\end{equation}
	 Moreover, $(\bm{\mu},\bm{Y}) \in \cK^\prime$. Finally, if $\{\bm{u}_j\}_{j \in \N} \subset \cA$ is a sequence of admissible controls such that for all $t \in [0,T]$ and $\bm{\nu} \in \cK$ 
	\begin{equation*}
		\lim_{j \to \infty}\int_0^t \bm{u}_j(s,\bm{\nu})\, ds=\int_0^t \bm{u}(s,\bm{\nu})\, ds
	\end{equation*}
	and $(\bm{\mu}^j,\bm{Y}^j)$ are the respective solutions of \eqref{PDEODE}, then,
	\begin{equation*}
		\lim_{j \to \infty}\sup_{0 \le t \le T}\left(\W_p(\mu_t^j,\mu_t)+|\bm{H}^j(t)-\bm{H}(t)|\right)=0.
	\end{equation*}
\end{thm}
\begin{proof}
	Fix $\bm{u} \in \cA$ and consider the equation
	\begin{equation}\label{eq:PDEaux}
		\begin{cases}
			\partial_t\mu_t=-v\cdot \nabla_x\mu_t+\sigma \Delta_v\mu_t-{\rm div}_v\left(G_{\bm{u}}[t,\bm{\mu}](z)\mu_t\right) & (t,z) \in (0,T] \times \R^{2d}\\
			\mu_0=\overline{\mu} & z \in \R^{2d}.
		\end{cases}
	\end{equation}
	Since $G_{\bm{u}}$ satisfies Assumptions \ref{ass:v}, then \eqref{eq:PDEaux} admits a unique solution $\bm{\mu} \in C^{\gamma_p}([0,T];\W_p(\R^{2d}))$ that can be expressed as $\bm{\mu}={\rm Law}(X,V)$ as in the statement, by Theorem \ref{thm:main1}. In particular, $\bm{\mu}$ belongs to $\cK$ by \eqref{eq:momentbounds}. Furthermore, setting $\bm{H}=\mathcal{S}[\bm{\mu},\bm{u}]$, we get that $(\bm{\mu},\bm{Y})$ clearly satisfies \eqref{PDEODE} and $\bm{Y} \in \cK_1$ by Lemma \ref{lem:PDEODEWell1}. This shows the existence of a solution for \eqref{PDEODE}. To prove uniqueness, let $(\bm{\mu}^\prime,\bm{Y}^\prime)$ be another solution. Then, by Definition \ref{def:solnlPDEODE} and uniqueness of the solution of \eqref{PDEODE2-sec}, we know that $\bm{H}^\prime=\mathcal{S}[\bm{\mu}^\prime,\bm{u}]$. However, this means that $\bm{\mu}^\prime$ is solution of \eqref{eq:PDEaux}, hence by Lemma \ref{lem:PDEODEWell1}, then $\bm{\mu}^\prime=\bm{\mu}$ and $\bm{H}^\prime=\mathcal{S}[\bm{\mu}^\prime,\bm{u}]=\mathcal{S}[\bm{\mu},\bm{u}]=\bm{H}$.
	
	Now, let $\{\bm{u}_j\}_{j \in \N}\subset \cA$ be a sequence of controls assumed as in the assumption. 
	Then, by Lemma \ref{lem:PDEODEWell4} and Corollary \ref{cor:stab}, we get that, for all $t \in [0,T]$,
	\begin{equation}\label{eq:limits2}
		\lim_{j \to \infty}\sup_{0 \le s \le t}\W_p(\mu_s^j,\mu_s)=0.
	\end{equation}
	Furthermore, we observe that by definition
	\begin{align*}
		\left|\bm{H}^j(t)-\bm{H}(t)\right| &\le \left|\mathcal{S}[\bm{\mu}^j,\bm{u}^j](t)-\mathcal{S}[\bm{\mu},\bm{u}^j](t)\right|+\left|\mathcal{S}[\bm{\mu},\bm{u}^j](t)-\mathcal{S}[\bm{\mu},\bm{u}](t)\right|\\
		&\le C\sup_{0 \le s \le T}\W_p(\mu_s^j,\mu_s)+\left|\mathcal{S}[\bm{\mu},\bm{u}^j](t)-\mathcal{S}[\bm{\mu},\bm{u}](t)\right|,
	\end{align*}
	where we applied \eqref{eq:bounds3} and the assumptions on $F$ and $\cA$. We get the desired statement by taking the supremum, then the limit on both sides of the previous inequality and using \eqref{eq:limits} and \eqref{eq:limits2}.
\end{proof}

\subsection{The Control Problem}\label{subsec:CP}
Now, we want to set a control problem on \eqref{PDEODE}. Precisely, we consider the cost functional
\begin{equation*}
	\F[\bm{u}]=\int_0^T\cL(t,\bm{\mu},\bm{Y})\, dt+\int_0^T \Psi(\bm{u}(t,\bm{\mu}))\, dt
\end{equation*}
where $(\bm{\mu},\bm{Y})$ is the solution of \eqref{PDEODE} with given control $\bm{u} \in \cA$, 
$\mathcal{L}$ is a lagrangian functional accounting for closedness to the decided target and $\Psi$ is a convex control cost. We consider the following assumptions on the cost functional $\F$.

\begin{ProblemSpecBox}{Assumptions on $\F$: $(\mathcal{F})$}{$(\mathcal{F})$}\label{ass:calF}
	\begin{itemize}
	\item[$(\mathcal{F}_0)$] $\cL:[0,T]\times C([0,T];\W_p(\R^{2d})) \times C([0,T];\R^{m})$ and $\Psi:\R^m \to \R$ are measurable and bounded from below.
	\item[$(\mathcal{F}_1)$] $\cL$ is continuous in the variables $(\bm{\mu},\bm{Y})$.
	\item[$(\F_2)$] There exists a function $M_{\cK^\prime} \in L^1(0,T)$ such that for a.a. $t \in [0,T]$ and all $(\bm{\mu},\bm{Y}) \in \cK^\prime$
	\begin{equation*}
		|\cL(t,\bm{\mu},\bm{Y})| \le M_{\cK^\prime}(t)
	\end{equation*} 
	\item[$(\F_3)$] $\Psi:\R^m \to \R$ is convex.
\end{itemize}
\end{ProblemSpecBox}
%
We are now ready to prove Theorem~\ref{main2}.
\subsection{Proof of Theorem \ref{main2}}



	By Theorem \ref{thm:PDEODEWell}, for any fixed $\bm{u} \in \cA$ the solution $(\bm{\mu},\bm{Y})$ of \eqref{PDEODE} belongs to 
	$\cK^\prime$.
	First, we show that $\F[\bm{u}]<\infty$ for all $\bm{u} \in \cA$. Indeed, by \eqref{eq:sublin} and letting $(\bm{\mu},\bm{Y})$ be the solution of 
	\eqref{PDEODE} with control $\bm{u}$, we have
	\begin{equation*}
		\left|\bm{u}(t,\bm{\mu})\right| \le K_{\bm{u}}(1+\mathcal{C}(K_G)^{\frac{1}{p}}).
	\end{equation*}
	Now, let $M_\Psi=\sup_{|x| \le K_{\bm{u}}(1+\mathcal{C}(K_G)^{\frac{1}{p}})}\Psi(x)$, which exists since $\Psi$ is convex and thus continuous. Hence
	\begin{equation*}
		\F[\bm{u}] \le \int_0^T M_{\cL}(t)\, dt+ M_\Psi T<\infty.
	\end{equation*}
	
	Now, we show that for any two controls $\bm{u}_1,\bm{u}_2 \in \cA$ such that for all $\bm{\nu} \in \cK$ and $t \in [0,T]$ it holds $\bm{u}_1(t,\bm{\nu})=\bm{u}_2(t,\bm{\nu})$, we have $\F[\bm{u}_1]=\F[\bm{u}_2]$. To do this, let $(\bm{\mu}^1,\bm{Y}^1)$ and $(\bm{\mu}^2,\bm{Y}^2)$ be the solutions of \eqref{PDEODE} with controls $\bm{u}_1$ and $\bm{u}_2$. Since $\bm{\mu}^1 \in \mathcal{K}$, then $\bm{u}_1(t,\bm{\mu}^1)=\bm{u}_2(t,\bm{\mu}^1)$, and hence we have
	\begin{align*}
		\bm{Y}^1(t)&=\overline{\bm{Y}}+\int_0^t(F[s,\bm{\mu}^1](\bm{Y}^1)+\bm{u}_1(s,\bm{\mu}^1))\, ds
		=\overline{\bm{Y}}+\int_0^t(F[s,\bm{\mu}^1](\bm{Y}^1)+\bm{u}_2(s,\bm{\mu}^1))\, ds.
	\end{align*}
	Thus, $\bm{Y}^1$ also solves \eqref{PDEODE2-sec} with control $\bm{u}_2$ and measure $\bm{\mu}^1$, i.e.
	\begin{equation*}
	\mathcal{S}[\bm{\mu}^1,\bm{u}_2]=\bm{H}^1=\mathcal{S}[\bm{\mu}^1,\bm{u}_2].
	\end{equation*}
 However, this means that
	\begin{align*}
		G_{\bm{u}_1}[t,\bm{\mu}^1](z)&=\mathfrak{v}[t,\bm{\mu}^1](z)+\mathfrak{w}[t,\mathcal{S}[\bm{\mu}^1,\bm{u}_1]](z)\\
		&=\mathfrak{v}[t,\bm{\mu}^1](z)+\mathfrak{w}[t,\mathcal{S}[\bm{\mu}^1,\bm{u}_2]](z)=G_{\bm{u}_2}[t,\bm{\mu}^1](z)
	\end{align*}
	and then $\bm{\mu}^1$ solves \eqref{eq:PDEaux} with field $G_{\bm{u}_2}$. However, also $\bm{\mu}^2$ solves \eqref{eq:PDEaux} with the field $G_{\bm{u}_2}$. Hence, since for a given field $G_{\bm{u}_2}$ \eqref{eq:PDEaux} admits a unique solution, $\bm{\mu}^1=\bm{\mu}^2$ and also $\bm{H}^1=\bm{H}^2$. Then, it is clear, by definition of $\F$, that $\F[\bm{u}_1]=\F[\bm{u}_2]$. Thus we are in a position to consider the equivalence relation $\sim$ over $\cA$ defined as
	\begin{equation*}
		\bm{u}_1 \sim \bm{u}_2 \quad \Leftrightarrow \quad \bm{u}_1(t,\bm{\mu})=\bm{u}_2(t,\bm{\mu}) \ \forall t \in [0,T], \ \forall \bm{\mu} \in \cK,
	\end{equation*}
	the quotient set $\widetilde{\cA}=\cA/\sim$ and the functional $\widetilde{\F}:\widetilde{\cA} \to \R$ defined as
	\begin{equation*}
		\widetilde{\F}[[\bm{u}]_{\sim}]=\F[\bm{u}], \quad \forall [\bm{u}]_{\sim} \in \widetilde{\cA}.
	\end{equation*}
	To actually work with $\widetilde{\F}$, we need to provide a suitable representation of the quotient set $\widetilde{\cA}$. To do this, first notice that for any $\bm{u} \in \cA$ we consider a function $\bm{U}:[0,T] \to C(\cK;\R^{m})$ defined as $\bm{U}(t)=\bm{u}(t,\cdot)$ for $t \in [0,T]$, so that $\bm{U}(t)(\bm{\mu})=\bm{u}(t,\bm{\mu})$ for any $\bm{\mu} \in \cK$. Since
	\begin{equation*}
		|\bm{u}(t,\bm{\mu})| \le K_{\bm{u}}(1+(\mathcal{C}(K_G))^{\frac{1}{p}}),
	\end{equation*}
	and then
	\begin{equation*}
		\int_0^T \sup_{\bm{\mu} \in \cK}|\bm{u}(t,\bm{\mu})|\, dt \le K_{\bm{u}}(1+(\mathcal{C}(K_G))^{\frac{1}{p}})T,
	\end{equation*}
	it is clear that $\bm{U} \in L^1([0,T];C(\cK;\R^{m}))$. Let $\Pi:\cA \to L^1([0,T];C(\cK;\R^{m}))$ be the function that maps $\bm{u}$ into $\bm{U}$ as before. If $\bm{u}_1 \sim \bm{u}_2$, then $\Pi\bm{u}_1=\Pi\bm{u}_2$, i.e., $\Pi$ is compatible with the equivalence relation $\sim$. In particular, this means that we can define the map $\widetilde{\Pi}:\widetilde{\cA} \to \Pi\cA$ such that $\widetilde{\Pi}[\bm{u}]_{\sim}=\Pi\bm{u}$. By definition, this map is surjective. Furthermore, assume that $\widetilde{\Pi}[\bm{u}_1]_{\sim}=\widetilde{\Pi}[\bm{u}_2]_{\sim}$. Then $\Pi\bm{u}_1=\Pi\bm{u}_2$ hence $\bm{u}_1(t,\bm{\mu})=\bm{u}_2(t,\bm{\mu})$ for all $t \in [0,T]$ and $\bm{\mu} \in \cK$, which in turn implies $\bm{u}_1 \sim \bm{u}_2$ and $[\bm{u}_1]_{\sim}=[\bm{u}_2]_{\sim}$. In particular, we have shown that $\widetilde{\Pi}$ is a bijection. From now on, without loss of generality, we identify $\widetilde{\cA}$ with $\Pi\cA$ and
	\begin{equation*}
		\F[\bm{u}]=\widetilde{\F}[\Pi\bm{u}].
	\end{equation*}
	We can then rewrite Problem \ref{prob:1} as follows:
	\begin{ProblemSpecBox}{Problem $1^\star$}{$1^\star$}\label{prob:1star}
		Find $\bm{U}^\star \in \widetilde{\cA}$ such that
		\begin{equation*}
			\widetilde{\F}[\bm{U}^\star]=\min_{\bm{U} \in \widetilde{\cA}}\widetilde{\F}[\bm{U}].
		\end{equation*}
	\end{ProblemSpecBox}
	
	Now let us consider a minimizing sequence $\{\bm{U}^n\}_{n \in \N}\subset \widetilde{\cA}$ for $\widetilde{\mathcal{F}}$. We will prove that this minimizing sequence admits (up to a subsequence) a limit $\bm{U}^\star$ in some sense that will be specified later. To do this, however, we preliminarily need to prove that $\{\bm{U}^n\}_{n \in \N}$ is uniformly integrable and uniformly tight, i.e.
	\begin{equation}\label{eq:unifintU}
		\lim_{R \to +\infty} \sup_{n \in \N}\int_{\{t \in [0,T]: \ \sup_{\bm{\mu} \in \cK}|\bm{U}^n(t)(\bm{\mu})|>R\}}\sup_{\bm{\mu} \in \cK}|\bm{U}^n(t)(\bm{\mu})|\, dt=0,
	\end{equation}
	and for all $\varepsilon>0$ there exists a compact-valued multifunction $\Gamma_\varepsilon:[0,T] \to 2^{C(\cK;\R^{m})}$ 
	with measurable graph such that
	\begin{equation*}
		\sup_{n \in \N}|\{t \in [0,T]: \bm{U}^n(t) \not \in \Gamma_\varepsilon(t)\}<\varepsilon,
	\end{equation*}
	respectively.
	Let us first show the uniform integrability. Indeed, for all $n \in \N$, there exists $\bm{u}^n \in \cA$ such that $\bm{U}^n=\Pi\bm{u}^n$ and then, for all $t \in [0,T]$ and $\bm{\mu} \in \cK$,
	\begin{equation*}
		|\bm{U}^n(t)(\bm{\mu})|=|\bm{u}^n(t,\bm{\mu})| \le K_{\bm{u}}(1+(\mathcal{C}(K_G))^{\frac{1}{p}}).
	\end{equation*}
	Hence, for $R>K_{\bm{u}}(1+(\mathcal{C}(K_G))^{\frac{1}{p}})$, we clearly have
	\begin{equation*}
		\int_{\{t \in [0,T]: \ \sup_{\bm{\mu} \in \cK}|\bm{U}^n(t)(\bm{\mu})|>R\}}\sup_{\bm{\mu} \in \cK}|\bm{U}^n(t)(\bm{\mu})|\, dt=0
	\end{equation*}
	and then \eqref{eq:unifintU} holds. 
	
	Now we show the uniform tightness. To do this, we fix $t \in [0,T]$ and notice that, by Assumption $(\cA_2)$ it holds
	\begin{equation*}
		|\bm{U}^n (t)(\bm{\nu}^1)-\bm{U}^n (t)(\bm{\nu}^2)| \le L_{\bm{u}}\sup_{0 \le t \le T}\W_p(\mu^1_t,\mu^2_t), \, 
		\quad \forall \bm{\mu}^1,\bm{\mu}^2 \in \cK.
	\end{equation*}
	Notice that the uniform topology of $C([0,T];\W_p(\R^{2d}))$ is generated by the metric
	\begin{equation*}
		\mathfrak{d}(\bm{\mu}^1,\bm{\mu}^2)=\sup_{t \in [0,T]}\W_p(\mu^1_t,\mu^2_t), \ 
		\quad \forall \bm{\mu}^1,\bm{\mu}^2 \in C([0,T];\W_p(\R^{2d})).
	\end{equation*}
	Hence by Lemma \ref{lem:PDEODEWell3}, $\cK$ equipped with $\mathfrak{d}$ is a compact Hausdorff space. Defining
	\begin{equation*}
		\widetilde{\cK}=\left\{f \in C(\cK;\R^m): \ \sup_{\bm{\mu} \in \cK}\left|f(\bm{\mu})\right|+\sup_{\substack{\bm{\mu}^1,\bm{\mu}^2 \\ \bm{\mu}^1 \not = \bm{\mu}^2}}\frac{|f(\bm{\mu}^1)-f(\bm{\mu}^2)|}{\mathfrak{d}(\bm{\mu}^1,\bm{\mu}^2)} \le L_{\bm{u}}+K_{\bm{u}}(1+(\mathcal{C}(K_G))^{\frac{1}{p}})\right\},
	\end{equation*}
	which is a compact subset of $C(\cK;\R^{m})$ by the Arzel\'a-Ascoli Theorem, we have $\bm{U}^n(t) \in \widetilde{\cK}$ for all $t \in [0,T]$ and $n \in \N$. Hence, to get the uniform tightness of $\{\bm{U}^n\}_{n \in \N}$ it is sufficient to set $\Gamma_\varepsilon(t)=\widetilde{\cK}$  for all $\varepsilon>0$ and $t \in [0,T]$.
	
	Now, we can use a suitable generalization of the Dunford-Pettis theorem, as in \cite[Theorem 1.3]{balder1997extension}, which tells us that there exists a function $\bm{U}^\star \in L^1([0,T];C(\cK,\R^m))$ such that, up to a non-relabelled subsequence, for all measurable subsets $E \subset [0,T]$
	\begin{equation}\label{eq:limiting}
		\lim_{n \to \infty}\sup_{\bm{\mu} \in \cK}\frac{1}{|E|}\left|\int_{E} (\bm{U}^n(t)(\bm{\mu})-\bm{U}^\star(t)(\bm{\mu}))\, dt\right|=0.
\end{equation}
	Then we need to show that $\bm{U}^\star \in \widetilde{\cA}$. To do this, first observe that for any $\bm{\mu}^1,\bm{\mu}^2 \in \cK$ it holds, by Assumption $(\cA_2)$,
	\begin{equation*}
		\frac{1}{|E|}\left|\int_{E} (U^n(t)_j(\bm{\mu}^1)-U^n_j(t)(\bm{\mu}^2))\, dt\right| \le \frac{L_{\bm{u}}}{m}\sup_{0 \le s \le \sup E}\W_p(\mu^1_s,\mu^2_s),
	\end{equation*}
	for any $j=1,\dots,m$, where $\bm{U}^n=(U^n_1,\cdots,U^n_j)$. Taking the limit as $n \to \infty$ and using \eqref{eq:limiting} we get
	\begin{equation}\label{eq:Lipschitz1}
		\frac{1}{|E|}\left|\int_{E} (U^\star_j(t)(\bm{\mu}^1)-U^\star_j(t)(\bm{\mu}^2))\, dt\right| \le \frac{L_{\bm{u}}}{m}\sup_{0 \le s \le \sup E}\W_p(\mu^1_s,\mu^2_s),
	\end{equation}
	where $\bm{U}^\star=(U^\star_1,\cdots,U^\star_m)$. Since $\bm{U}^\star \in L^1([0,T];C(\cK;\R^m))$, by \cite[Proposition 1.2.2]{arendt2011vector}, we know that almost any $t \in [0,T]$ is a Lebesgue point for $\bm{U}^\star$, i.e. there exists a set $E_{\sf Leb}\subset [0,T]$ such that $|[0,T]\setminus E_{\sf Leb}|=0$ and for all $t \in E_{\sf Leb}$ it holds
	\begin{equation*}
		\lim_{h \to 0}\frac{1}{h}\int_t^{t+h}\sup_{\bm{\mu} \in \cK}|\bm{U}^\star(s)(\bm{\mu})-\bm{U}^\star(t)(\bm{\mu})|\, ds=0.
	\end{equation*}
	In particular, if we fix $\bm{\mu}\in \cK$ for all $t \in E_{\sf Leb}$ we have
	\begin{equation*}
		\lim_{h \to 0}\left|\frac{1}{h}\int_t^{t+h}\bm{U}^\star(s)(\bm{\mu}) \, ds-\bm{U}^\star(t)(\bm{\mu})\right| \le \lim_{h \to 0}\frac{1}{h}\int_t^{t+h}\sup_{\bm{\mu} \in \cK}|\bm{U}^\star(s)(\bm{\mu})-\bm{U}^\star(t)(\bm{\mu})|\, ds=0.
	\end{equation*}
	Fix $t \in E_{\sf Leb}$ and observe that by \eqref{eq:Lipschitz1} for any $h>0$ it holds
	\begin{align*}
		\left|\frac{1}{h}\int_{t}^{t+h} (U_j^\star(s)(\bm{\mu}^1)-U_j^\star(s)(\bm{\mu}^2))\, ds\right| \le \frac{L_{\bm{u}}}{m}\sup_{0 \le s \le t+h}\W_p(\mu^1_s,\mu^2_s).
	\end{align*}
	Now we take the limit as $h \to 0$ on both sides: on the left-hand side, we use the fact that $t \in E_{\sf Leb}$, while on the right-hand side we use the continuity of $\bm{\mu}^1,\bm{\mu}^2 \in \cK$. Finally, for any $j=1, \ldots, m$ we achieve
	\begin{align}\label{eq:Lipschitz2}
		\left|U_j^\star(t)(\bm{\mu}^1)-U_j^\star(t)(\bm{\mu}^2)\right| \le \frac{L_{\bm{u}}}{m}\sup_{0 \le s \le t}\W_p(\mu^1_s,\mu^2_s).
	\end{align}
	For $t \not \in E_{\sf Leb}$, without loss of generality, we set $\bm{U}^\star(t) \equiv 0$, since $|[0,T]\setminus E_{\sf Leb}|=0$ and \eqref{eq:Lipschitz2} still holds. Hence, we have $\bm{U}^\star$ satisfies \eqref{eq:Lipschitz2} for all $t \in [0,T]$. Next, notice that, since $\bm{\delta}_0 \in \cK$ and by Assumption $(\cA_1)$, for any measurable subset $E \subset [0,T]$ we have
	\begin{equation*}
		\frac{1}{|E|}\left|\int_{E}\bm{U}^n(t)(\bm{\delta}_0)\, dt\right| \le M_{\bm{u}}.
	\end{equation*}
	With the same argument as before, we get, for all $t \in E_{\sf Leb}$,
	\begin{equation}\label{eq:boundedness}
		\left|\bm{U}^\star(t)(\bm{\delta}_0)\right| \le M_{\bm{u}}.
	\end{equation}
	Combining this with the fact that we set $\bm{U}^\star(t)\equiv 0$ whenever $t \not \in E_{\sf Leb}$, we infer \eqref{eq:boundedness} holds for all $t \in [0,T]$. 
	
	Now, we fix $t \in [0,T]$ and introduce another equivalence relation $\sim_t$ on $\cK$, as follows:
	\begin{equation*}
		\bm{\mu}^1 \sim_t \bm{\mu}^2 \Leftrightarrow \mu_s^1=\mu_s^2, \ \forall s \in [0,t].
	\end{equation*}
	Notice that, by \eqref{eq:Lipschitz2}, if $\bm{\mu}^1 \sim_t \bm{\mu}^2$ then $U_j^\star(t)(\bm{\mu}^1)=U_j^\star(t)(\bm{\mu}^2)$ for all $j=1,\dots,m$. This means that $U_j^\star(t)$ is compatible with the equivalence relation $\sim_t$ and then we can define the function $\widetilde{U}_j^\star(t):\widetilde{\cK}_t \to \R$, where $\widetilde{\cK}_t=\cK/\sim_t$ is the quotient of $\cK$ with respect to the equivalence relation $\sim_t$, as
	\begin{equation*}
		\widetilde{U}_j^\star(t)([\bm{\mu}]_{\sim_t})=U_j^\star(t)(\bm{\mu}).
	\end{equation*}
	Let us also denote by $\pi_t:\cK \mapsto \widetilde{\cK}_t$ the projection, i.e. $\pi_t\bm{\mu}=[\bm{\mu}]_{\sim_t}$, and endow $\widetilde{\cK}_t$ with the quotient topology. Notice also that the function $\Pi_t : \cK \to C([0,t];\W_p(\R^{2d}))$ defined as $(\Pi_t\bm{\mu})_s=\mu_s$ for all $s \in [0,t]$ (i.e., $\Pi_t$ is the function that maps $\bm{\mu}$ to its restriction on the interval $[0,t]$) is compatible with $\sim_t$, i.e., if $\bm{\mu}^1 \sim_t \bm{\mu}^2$ then $\Pi_t \bm{\mu}^1=\Pi_t \bm{\mu}^2$. We endow $\Pi_t\cK$ with the uniform metric $\mathfrak{d}_t$ on $C([0,t];\W_p(\R^{2d}))$, i.e.,
	\begin{equation*}
		\mathfrak{d}_t(\bm{\mu}^1,\bm{\mu}^2)=\sup_{0 \le s \le t}\W_p(\mu^1_s,\mu^2_s), \ \bm{\mu}^1,\bm{\mu}^2 \in \Pi_t\cK
	\end{equation*}
	It is clear that $\Pi_t$ is non-expansive, i.e., for all $\bm{\mu}^1,\bm{\mu}^2 \in \cK$ it holds
	\begin{equation*}
		\mathfrak{d}_t(\Pi_t\bm{\mu}^1,\Pi_t\bm{\mu}^2) \le \mathfrak{d}(\bm{\mu}^1,\bm{\mu}^2).
	\end{equation*}
	Since $\Pi_t$ is compatible with $\sim_t$, we define the map $\widetilde{\Pi}_t:\widetilde{\cK}_t \to \Pi_t\cK$ as $\widetilde{\Pi}_t[\bm{\mu}]_{\sim_t}=\Pi_t\bm{\mu}$, which is clearly continuous with respect of the quotient topology on $\widetilde{\cK}_t$. Furthermore, notice that if $\Pi_t\bm{\mu}^1=\Pi_t\bm{\mu}^2$, then, by definition, $\bm{\mu}^1 \sim_t \bm{\mu}^2$. Hence, $\widetilde{\Pi}_t$ is bijective. Finally, consider an open set $E \subset \widetilde{\cK}_t$. We want to prove that $\widetilde{\Pi}_tE$ is an open set. To do this, fix $[\bm{\mu}]_{\sim_t} \in E$. We want to fix a special representative $\bm{\mu} \in [\bm{\mu}]_{\sim_t}$. Precisely, let $\widetilde{\bm{\mu}}=\widetilde{\Pi}_t[\bm{\mu}]_{\sim_t}$ and define $\bm{\mu} \in C([0,T];\W_p(\R^{2d}))$ by setting $\mu_s=\widetilde{\mu}_{\min\{s,t\}}$ for $s \in [0,T]$. It is not difficult to check that $\bm{\mu} \in \cK$. In particular $\bm{\mu} \in \pi_t^{-1}(U)$, and notice that since $\pi_t^{-1}(E)$ is an open set in $\cK$, there exists $\varepsilon>0$ such that $B_\varepsilon(\bm{\mu}):=\{\bm{\nu} \in \cK: \ \mathfrak{d}(\bm{\mu},\bm{\nu})<\varepsilon\} \subset \pi_t^{-1}(E)$. 
	
	Now, we define $B_\varepsilon(\Pi_t\bm{\mu};t):=\{\bm{\nu} \in \Pi_t\cK: \ \mathfrak{d}_t(\Pi_t\bm{\mu},\bm{\nu})<\varepsilon\}$. Since $\Pi_t$ is non-expansive, we have $\Pi_t B_\varepsilon(\bm{\mu}) \subseteq B_\varepsilon(\Pi_t\bm{\mu};t)$. On the other hand, consider any $\widetilde{\bm{\nu}} \in B_\varepsilon(\Pi_t\bm{\mu};t)$ and define $\bm{\nu} \in C([0,T];\W_p(\R^{2d}))$ by setting $\nu_s=\widetilde{\nu}_{\min\{s,t\}}$ for $s \in [0,T]$. Again, it is not difficult to check that $\bm{\nu} \in \cK$ and $\mathfrak{d}(\bm{\nu},\bm{\mu})<\varepsilon$. Hence, $\widetilde{\bm{\nu}}=\Pi_t\bm{\nu}$ with $\bm{\nu} \in B_\varepsilon(\bm{\mu})$ and then $\widetilde{\bm{\nu}} \in \Pi_tB_\varepsilon(\Pi_t\bm{\mu};t)$. This shows that $\Pi_t B_\varepsilon(\bm{\mu})=B_\varepsilon(\Pi_t\bm{\mu};t)$. Recalling that $\Pi_t\bm{\mu}=\widetilde{\Pi}_t[\bm{\mu}]_{\sim_t}$, we get
	\begin{equation*}
		B_\varepsilon(\Pi_t\bm{\mu};t)=\Pi_t B_\varepsilon(\bm{\mu})=\widetilde{\Pi}_t\pi_t B_\varepsilon(\bm{\mu}) \subseteq \widetilde{\Pi}_t E, 
	\end{equation*}
	where we used the fact that since $B_\varepsilon(\bm{\mu}) \subseteq \pi_t^{-1}(E)$. Since $[\bm{\mu}]_{\sim_t} \in E$ is arbitrary, $\widetilde{\Pi}_t E$ is an open set. In conclusion, $\widetilde{\Pi}_t :\widetilde{\cK}_t \to \Pi_t\cK$ is a homeomorphism.
	
	Now we notice that $\widetilde{U}_j^\star(t)\circ \widetilde{\Pi}_t^{-1}$ is $\frac{L_{\bm{u}}}{m}$-Lipschitz. Indeed, for any $\widetilde{\bm{\mu}}^1,\widetilde{\bm{\mu}}^2 \in \Pi_t\cK$, letting $\bm{\mu}^j \in \Pi_t^{-1}\widetilde{\bm{\mu}}^j$ for $j=1,2$, it holds
	\begin{align*}
		\left|\widetilde{U}_j^{\star}(t)(\widetilde{\Pi}_t^{-1}(\widetilde{\bm{\mu}}^1))-\widetilde{U}_j^{\star}(t)(\widetilde{\Pi}_t^{-1}(\widetilde{\bm{\mu}}^2))\right|&=\left|\widetilde{U}_j^{\star}(t)([\bm{\mu}^1]_{\sim_t})-\widetilde{U}_j^{\star}(t)([\bm{\mu}^2]_{\sim_t})\right|\\
		&=\left|U_j^\star(t)(\bm{\mu}^1)-U_j^\star(t)(\bm{\mu}^2)\right|\\
		& \le \frac{L_{\bm{u}}}{m}\sup_{0 \le s \le t}\W_p(\mu^1_s,\mu^2_s)=\frac{L_{\bm{u}}}{m}\mathfrak{d}(\widetilde{\bm{\mu}}^1,\widetilde{\bm{\mu}}^2).
	\end{align*}
	Let us recall that also $\Pi_t\cK$ is a compact subset of $C([0,t];\W_p(\R^{2d}))$. Now we can use McShane's Extension Theorem (see \cite{mcshane1934extension}) to state that the function $\widetilde{u}_j^\star(t,\cdot)$ defined as
	\begin{equation*}
		\widetilde{u}_j^\star(t,\bm{\nu})=\min_{\bm{\mu} \in \Pi_t\cK}\left\{\widetilde{U}_j^\star(t)(\widetilde{\Pi}_t^{-1}\bm{\mu})+\frac{L_{\bm{u}}}{m}\mathfrak{d}(\bm{\mu},\bm{\nu})\right\}, \quad \bm{\nu} \in C([0,t];\W_p(\R^{2d})),
	\end{equation*}
	is the maximal $\frac{L_{\bm{u}}}{m}$-Lipschitz extension of $\widetilde{U}_j^\star(t) \circ \widetilde{\Pi}_t^{-1}$ on $C([0,t];\W_p(\R^{2d}))$. Finally, we set, for all $t \in [0,T]$ and $\bm{\mu} \in C([0,T];\W_p(\R^{2d}))$, 
	\begin{equation*}
		u_j^\star(t,\bm{\mu})=\widetilde{u}_j^\star(t,\Pi_t\bm{\mu})
	\end{equation*}
	and $\bm{u}^\star=(u_j^\star)_{j=1,\dots,m}$. We can clearly rewrite $u_j^\star$ as follows:
		\begin{equation*}
		u_j^\star(t,\bm{\nu})=\min_{\bm{\mu} \in \cK}\left\{U_j^\star(t)(\bm{\mu})+\frac{L_{\bm{u}}}{m}\sup_{0 \le s \le t}\W_p(\bm{\mu},\bm{\nu})\right\}, \quad \bm{\nu} \in C([0,T];\W_p(\R^{2d})),
	\end{equation*}
	so that it is clear that $u_j^\star$ is measurable in the variable $t$, as it is the minimum of a family of measurable functions. Concerning the continuity on $\bm{\mu}$, notice that for all $\bm{\mu}^1,\bm{\mu}^2 \in C([0,T];\W_p(\R^{2d}))$,
	\begin{equation*}
		\left|u_j^\star(t,\bm{\mu}^1)-u_j^\star(t,\bm{\mu}^2)\right|=\left|\widetilde{u}_j^\star(t,\Pi_t\bm{\mu}^1)-\widetilde{u}_j^\star(t,\Pi_t\bm{\mu}^2)\right| \le \frac{L_{\bm{u}}}{m}\sup_{0 \le s \le t}\W_p(\mu^1_s,\mu^2_s).
	\end{equation*}
	This shows that $\bm{u}^\star$ is continuous in the $\bm{\mu}$ variable, so that $(\cA_0)$ holds, and that $(\cA_2)$ also holds. Furthermore, notice that
	\begin{multline*}
		\left|\bm{u}^\star(t,\bm{\delta}_0)\right|^2=\sum_{j=1}^{m}\left|\widetilde{u}_j^\star(t,\Pi_t\bm{\delta}_0)\right|^2=\sum_{j=1}^{m}\left|\widetilde{U}_j^\star(t)(\widetilde{\Pi}_t^{-1}\Pi_t\bm{\delta}_0)\right|^2\\
		=\sum_{j=1}^{m}\left|\widetilde{U}_j^\star(t)([\bm{\delta}_0]_{\sim_t})\right|^2=\sum_{j=1}^{m}\left|U_j^\star(t)(\bm{\delta}_0)\right|^2 \le M_0^2,
	\end{multline*}
	by \eqref{eq:boundedness}, thus $(\cA_1)$ also holds. As a consequence, $\bm{u}^\star \in \cA$. However, if $\bm{\mu} \in \cK$, we have
	\begin{equation*}
		u_j^\star(t,\bm{\mu})=\widetilde{u}_j^\star(t,\Pi_t\bm{\mu})=\widetilde{U}_j^\star(t)(\widetilde{\Pi}_t^{-1}\Pi_t\bm{\mu})\\
		=\widetilde{U}_j^\star(t)([\bm{\mu}]_{\sim_t})=U_j^\star(t)(\bm{\mu}),
	\end{equation*}
	i.e. $\bm{U}^\star=\Pi \bm{u}^\star \in \widetilde{\cA}$.
	
	It remains to show that $\bm{U}^\star$ is solution of Problem \ref{prob:1star}. Recall that $\{\bm{U}^n\}_{n \in \N} \subset \widetilde{\cA}$, hence there exists a sequence of controls $\{\bm{u}^n\}_{n \in \N} \subset \cA$ such that $\bm{U}^n=\Pi \bm{u}^n$. Let $(\bm{\mu}^n, \bm{Y}^n)$ be the respective solutions of \eqref{PDEODE}. Furthermore, let $(\bm{\mu}^\star,\bm{Y}^\star)$ be the solution of \eqref{PDEODE} with control $\bm{u}^\star$. Observe that $\Psi:\R^m \to \R$ is a convex function and let $\Psi^\ast$ be its Legendre transform, i.e.
	\begin{equation*}
		\Psi^\ast(x)=\sup_{y \in \R^m}(y \cdot x-\Psi(y)), \quad x \in \R^m.
	\end{equation*}
	Recall that $\Psi^\ast:\R^m \to \R$ is still a convex function and that the Legendre transform is an involution, i.e., 
	\begin{equation*}
		\Psi(x)=\Psi^{\ast \ast}(x)=\sup_{y \in \R^m}(y \cdot x-\Psi^\ast(y)), \quad x \in \R^m.
	\end{equation*}
	Fix $y \in \R^m$ and let $\Psi(x;y)=y \cdot x-\Psi^\ast(y)$, which is an affine function. Then we have, for any measurable subset $E \subset [0,T]$
	\begin{align}
		\int_E \Psi(\bm{U}^n(t)(\bm{\mu}^n);y)\, dt&=y \cdot \int_E \bm{U}^n(t)(\bm{\mu}^n)\, dt-\Psi^\ast(y)|E|\\
		&=y \cdot \int_E \bm{U}^n(t)(\bm{\mu}^\star)\, dt-\Psi^\ast(y)|E|+y \cdot \mathcal{R}_n(E),\label{eq:prelimit1}
	\end{align}
	where
	\begin{equation*}
		\mathcal{R}_n(E)=\int_E \left(\bm{U}^n(t)(\bm{\mu}^n)-\bm{U}^n(t)(\bm{\mu}^\star)\right)\, dt.
	\end{equation*}
	Let us show that $\mathcal{R}_n \to 0$. To do this, first notice that
	\begin{equation*}
		\left|\mathcal{R}_n(E)\right| \le \int_E \left|\bm{u}^n(t,\bm{\mu}^n)-\bm{u}^n(t,\bm{\mu}^\star)\right| \, dt \le L_{\bm{u}}|E| \sup_{0 \le s \le T}\W_p(\bm{\mu}^,\bm{\mu}^\star).
	\end{equation*}
	By \eqref{eq:limiting}, we know that
	\begin{equation*}
		\lim_{n \to \infty}\int_E\bm{u}^n(t,\bm{\mu})\, dt=\int_E\bm{u}^\star(t,\bm{\mu})\, dt, \ \forall \bm{\mu} \in \cK,
	\end{equation*}
	hence we can use Theorem \ref{thm:PDEODEWell} to conclude that
	\begin{equation*}
		\lim_{n \to +\infty} \left|\mathcal{R}_n(E) \right|=0.
	\end{equation*}
	Taking the limit as $n \to \infty$ in \eqref{eq:prelimit1} we get
	\begin{equation*}
		\lim_{n \to +\infty}\int_E \Psi(\bm{U}^n(t)(\bm{\mu}^n);y)\, dt=y \cdot \int_E \bm{U}^\star(t)(\bm{\mu}^\star)\, dt-\Psi^\ast(y)|E|=\int_E \Psi(\bm{U}^\star(t)(\bm{\mu}^\star);y)\, dt.
	\end{equation*}
	In particular, taking the supremum only on the left-hand side, we have
	\begin{align*}
		\int_E \Psi(\bm{U}^\star(t)(\bm{\mu}^\star);y)\, dt &\le \sup_{y \in \R^m}\lim_{n \to +\infty}\int_E \Psi(\bm{U}^n(t)(\bm{\mu}^n);y)\, dt \\
		&\le \liminf_{n \to +\infty} \sup_{y \in \R^m}\int_E \Psi(\bm{U}^n(t)(\bm{\mu}^n);y)\, dt\\
		& \le \liminf_{n \to +\infty}\int_E \Psi(\bm{U}^n(t)(\bm{\mu}^n))\, dt=:\mathcal{G}(E).
	\end{align*}
	Notice that $\mathcal{G}$ is superadditive. Indeed, if we consider two measurable sets $E_1,E_2 \subset [0,T]$ with $E_1 \cap E_2=\emptyset$, then
	\begin{equation*}
		\int_{E_1 \cup E_2} \Psi(\bm{U}^n(t)(\bm{\mu}^n))\, dt=\int_{E_1} \Psi(\bm{U}^n(t)(\bm{\mu}^n))\, dt+\int_{E_2} \Psi(\bm{U}^n(t)(\bm{\mu}^n))\, dt
	\end{equation*}
	and taking the limit inferior we achieve
	\begin{equation*}
		\mathcal{G}(E_1 \cup E_2) \ge \mathcal{G}(E_1)+\mathcal{G}(E_2).
	\end{equation*}
	Hence, by the Localization Lemma (see \cite[Proposition 1.16]{braides1998approximation}), we have
	\begin{align}\label{eq:semicont1}
		\int_E \Psi(\bm{U}^\star(t)(\bm{\mu}^\star))\, dt &=\int_E \sup_{y \in \R^m}\Psi(\bm{U}^\star(t)(\bm{\mu}^\star);y)\, dt \\
		&\le \mathcal{G}(E)= \liminf_{n \to +\infty}\int_E \Psi(\bm{U}^n(t)(\bm{\mu}^n)).
	\end{align}
	Next, recall that by Theorem \ref{thm:PDEODEWell} we have
	\begin{equation*}
		\lim_{j \to \infty}\sup_{0 \le t \le T}(\W_p(\mu^j_t,\mu^\star_y)+|\bm{H}^j(t)-\bm{H}^\star(t)|)=0
	\end{equation*}
	hence, by $(\F_1)$,
	\begin{equation*}
		\lim_{j \to \infty}\cL(t,\bm{\mu}^j,\bm{Y}^j)=\cL(t,\bm{\mu}^\star,\bm{Y}^\star).
\end{equation*}
	We also recall that, still by Theorem \ref{thm:PDEODEWell}, $(\bm{\mu}^j,\bm{H}^j) \in \cK^\prime$ for all $j \in \N$. Thus, by $(\F_2)$ and dominated convergence,
	\begin{equation}\label{eq:semicont2}
		\lim_{j \to \infty}\int_0^T\cL(t,\bm{\mu}^j,\bm{Y}^j)\, dt=\int_0^T \cL(t,\bm{\mu}^\star,\bm{Y}^\star) \, dt.
	\end{equation}
	Finally, combining \eqref{eq:semicont1} and \eqref{eq:semicont2} we achieve
	\begin{equation*}
		\widetilde{\F}[\bm{U}^\star] \le \lim_{n \to \infty}\widetilde{\F}[\bm{U}^n]=\inf_{\bm{U} \in \widetilde{\cA}}\F[\bm{U}],
	\end{equation*}
	i.e., $\bm{U}^\star$ is solution of Problem \ref{prob:1star}. \qed

\newpage

\appendix

\section{Some sufficient conditions to the validity of assumption $(\mathfrak{v}_3)$} \label{appendix:equivalence}
As already underlined in the introduction, it could be difficult to verify that $(\mathfrak{v}_3)$ actually holds. For such  reason, let us give some sufficient conditions for $(\mathfrak{v}_3)$ to be true.

\begin{prop}\label{prop:equivlence}
	Let Assumptions \ref{ass:v} (except for $(\mathfrak{v}_3)$) hold true and denote
	\begin{equation*}
		{\sf D}([0,T];\R^{2d})=\{(X,V) \in C([0,T];\R^{2d}): \ X'=V\}.
	\end{equation*}
	Consider the following properties:
	\begin{itemize}
		\item[$(\mathfrak{v}_3^{\prime})$] For any $T>0$ there exists a constant  $D \ge 0$ such that for any $t \in [0,T]$, any couple of measures ${\bm{\mu}}^1,{\bm{\mu}}^2 \in \W_p(C([0,T];\R^{2d}))$ supported on ${\sf D}([0,T];\R^{2d})$ with $\mu^1_0=\mu^2_0$ and any $\bm{\gamma} \in \Pi({\bm{\mu}}^1,{\bm{\mu}}^2)$ it holds
		\begin{multline}\label{eq:condD4}
			\int_{\R^{2d}\times \R^{2d}}|v_1-v_2|^{p-2}(\mathfrak{v}[t,{\bm{\mu}}^1](z_1)-\mathfrak{v}[t,{\bm{\mu}}^2](z_2))\cdot(v_1-v_2)\, d{\gamma}_t(z_1,z_2) \\ \le D \sup_{0 \le s \le t}\left(\int_{\R^{2d}\times \R^{2d}}|z_1-z_2|^p \, d{\gamma}_s(z_1,z_2)\right)
		\end{multline}
		\item[$(\mathfrak{v}_3^{\prime\prime})$] For any $T>0$ there exists a constant $D \ge 0$ such that for any $t \in [0,T]$, any ${\bm{\mu}}^1,{\bm{\mu}}^2 \in \W_p(C([0,T];\R^{2d}))$ supported on ${\sf D}([0,T];\R^{2d})$ with $\mu^1_0=\mu_0^2$ and $z_1,z_2 \in \R^{2d}$
		\begin{equation*}
			\left|\mathfrak{v}[t,{\bm{\mu}}^1](z_1)-\mathfrak{v}[t,{\bm{\mu}}^2](z_2)\right| \le D(\sup_{0 \le s \le t}\W_p(\mu_s^1,\mu_s^2)+|z_1-z_2|).
		\end{equation*}
	\end{itemize}
	Then $(\mathfrak{v}_3^{\prime \prime}) \Rightarrow (\mathfrak{v}_3^{\prime}) \Rightarrow (\mathfrak{v}_3)$. 
\end{prop}
\begin{proof}
	To show that $(\mathfrak{v}_3^{\prime}) \Rightarrow (\mathfrak{v}_3)$, consider any two processes $Z_j \in L^p(\Omega;C([0,T];\R^{2d}))$ as in Assumption $(\mathfrak{v}_3)$ and let $\bm{\mu}^j={\rm Law}(Z_j)$ for $j=1,2$. Then $\bm{\mu}^j$ is supported on ${\sf D}([0,T];\R^{2d})$ and $\mu_0^1=\mu_0^2$. Set also $\gamma={\rm Law}(Z_1,Z_2) \in \Pi(\bm{\mu}^1,\bm{\mu}^2)$. Then, by $(\mathfrak{v}_3^\prime)$, we have
	\begin{multline*}
		\E\left[|V_1(t)-V_2(t)|^{p-2}(\mathfrak{v}[t,\bm{\mu}^1](Z_1(t))-\mathfrak{v}[t,\bm{\mu}^1](Z_2(t)))\cdot (V_1(t)-V_2(t))\right] \\
		\le D \sup_{0 \le s \le t}\E[|Z_1(s)-Z_2(s)|^p].
	\end{multline*}
	Integrating on $[0,T]$ we get $(\mathfrak{v}_3)$.

	Now let us show that $(\mathfrak{v}_3^{\prime \prime}) \Rightarrow (\mathfrak{v}_3^{\prime})$. Fix ${\bm{\mu}}^1,{\bm{\mu}}^2 \in \W_p (C([0,T];\R^{2d}))$ and $\bm{\gamma} \in \Pi({\bm{\mu}}^1,{\bm{\mu}}^2)$. First consider the case $p=1$. Observe that 
	\begin{align*}
		\int_{\R^{2d}\times \R^{2d}}&|v_1-v_2|^{-1}(\mathfrak{v}[t,{\bm{\mu}}^1](z_1)-\mathfrak{v}[t,{\bm{\mu}}^2](z_2))\cdot(v_1-v_2)\, d\gamma_t (z_1,z_2) \\
		&\le \int_{\R^{2d}\times \R^{2d}}|\mathfrak{v}[t,{\bm{\mu}}^1](z_1)-\mathfrak{v}[t,{\bm{\mu}}^2](z_2)|\, d\gamma_t(z_1,z_2)\\
		&\le D\int_{\R^{2d}\times \R^{2d}} \sup_{0\le s\le t}\W_1(\mu^1_s,\mu^2_s)\, d\gamma_t(z_1,z_2)+D\int_{\R^{2d}\times \R^{2d}}|z_1-z_2|\, d\gamma_t(z_1,z_2)\\
		&\le D \sup_{0\le s\le t}\W_1(\mu^1_s,\mu^2_s) + D \sup_{0\le s\le t} \int_{\R^{2d}\times \R^{2d}}|z_1-z_2|\, d\gamma_s(z_1,z_2)\\
		&\le 2 D \sup_{0\le s\le t} \int_{\R^{2d}\times \R^{2d}}|z_1-z_2|\, d\gamma_s(z_1,z_2)
	\end{align*}
	On the other hand, if $p>1$ we get 
	\begin{align*}
		\int_{\R^{2d}\times \R^{2d}}&|v_1-v_2|^{p-2}(\mathfrak{v}[t,{\bm{\mu}}^1](z_1)-\mathfrak{v}[t,{\bm{\mu}}^2](z_2))\cdot(v_1-v_2)\, d{\gamma}_t(z_1,z_2) \\
		&\le \int_{\R^{2d}\times \R^{2d}}|v_1-v_2|^{p-1}|\mathfrak{v}[t,{\bm{\mu}}^1](z_1)-\mathfrak{v}[t,{\bm{\mu}}^2](z_2)|\, d{\gamma}_t(z_1,z_2)\\
		&\le D\int_{\R^{2d}\times \R^{2d}}|v_1-v_2|^{p-1}\sup_ {0 \le s \le t} \W_p(\mu^1_s,\mu^2_s)\, d{\gamma}_t(z_1,z_2)\\
		&\quad +D\int_{\R^{2d}\times \R^{2d}}|v_1-v_2|^{p-1}|z_1-z_2|\, d{\gamma}_t(z_1,z_2)\\
		&\le D\frac{p-1}{p}\int_{\R^{2d}\times \R^{2d}}|v_1-v_2|^{p} d{\gamma}_t(z_1,z_2)+\frac{D}{p}\int_{\R^{2d}\times \R^{2d} } \sup_{0 \le s \le t} \W_p^p(\mu^1_s,\mu^2_s) d{\gamma}_t(z_1,z_2)\\
		&\quad +D\int_{\R^{2d} \times \R^{2d}}|z_1-z_2|^p d{\gamma}_t(z_1,z_2) \le 2D \sup_{0 \le s \le t} \int_{\R^{2d}\times \R^{2d}}|z_1-z_2|^p\, d{\gamma}_s(z_1,z_2).
	\end{align*}
	This ends the proof.
\end{proof}

\newpage

\section{Preliminaries for the well-posedness of the PDE-ODE system} \label{wp-pre}
In this subsection, we collect some useful preliminary results to prove the well-posedness of the PDE-ODE system in \eqref{PDEODE}.
First of all, we consider an auxiliary system and we show the existence and uniqueness of its solution, alongside with some additional
qualitative properties.

\begin{lem}\label{lem:PDEODEWell1} 
	Let Assumptions \ref{ass:v}, \ref{ass:w}, \ref{ass:F} and \ref{ass:A} hold.
	Let $T>0$, $\bm{u} \in \cA$ and $\bm{\mu} \in C([0,T];\W_p(\R^{2d}))$ be fixed, then the system 
	\begin{equation}\label{PDEODE2}
		\begin{cases}
			\dot{\mathbf{Y}_{\bm{\mu}}}(t)=\mathbf{W}_{\bm{\mu}}(t)=F[t,\bm{\mu}](\bm{Y}_{\bm{\mu}})+\mathbf{u}(t,\bm{\mu}) & t \in (0,T]\\
			\mu_0=\overline{\mu}, \qquad \mathbf{Y}_{\bm{\mu}}(0)=\overline{\mathbf{Y}},
		\end{cases}
	\end{equation}
	admits a unique solution $\bm{Y}_{\bm{\mu}}$. Furthermore, setting $\bm{H}_{\mu}=(\bm{Y}_{\bm{\mu}},\bm{W}_{\bm{\mu}})$:
	\begin{itemize}
		\item[(i)] there exist three constants $C_j$, $j=1,2,3$, depending on $p,K_F,K_{\bm{u}},L_F,L_{\bm{u}},T$ and $\overline{\bm{H}}$ such that
			\begin{align}\label{eq:bounds1}
				\sup_{0 \le s \le T}|\bm{Y}_{\bm{\mu}}(s)| &\le C_1(1+\overline{M}_p(T;\bm{\mu})^{\frac{1}{p}}),  \\ \label{eq:bounds3}
				 \sup_{0 \le s \le T}|\bm{Y}_{\bm{\mu}}(s)-\bm{Y}_{\bm{\nu}}(s)| &\le C_2 \sup_{0 \le s \le T}\W_p(\mu_s,\nu_s), \\ \label{eq:bounds2}
				|\bm{Y}_{\bm{\mu}}(s)-\bm{Y}_{\bm{\mu}}(t)| &\le C_3(1+\overline{M}_p(T;\bm{\mu})^{p})|t-s|, \ \forall t \in [0,T];
			\end{align}
		\item[(ii)] if there exists a sequence $\{\bm{u}_j\}_{j \in \N} \subset \cA$ such that
	\begin{equation*}
		\lim_{j \to \infty}\int_{0}^{t}\bm{u}_j(s,\bm{\mu})dt=\int_{0}^{t}\bm{u}(s,\bm{\mu})dt, \ \forall t \in [0,T],
	\end{equation*}
	then
	\begin{equation}\label{eq:limits}
		\lim_{j \to \infty}\sup_{t \in [0,T]}|\bm{H}_{\bm{\mu}}^j(t)-\bm{H}_{\bm{\mu}}(t)|=0,
	\end{equation}
	where $\bm{H}_{\bm{\mu}}^j$ is the solution of \eqref{PDEODE2} with $\bm{u}_j$ in place of $\bm{u}$;
	\item[(iii)] for any $t \in [0,T]$, if $\bm{\mu}^1,\bm{\mu}^2 \in C([0,T];\W_p(\R^{2d}))$ are such that $\mu^1_s=\mu^2_s$ for all $s \in [0,t]$, 
	then $\bm{H}_{\bm{\mu}^1}(t)=\bm{H}_{\bm{\mu}^2}(t)$.
	\end{itemize}
\end{lem}
\begin{proof}
	First of all, existence and uniqueness of the solution of \eqref{PDEODE2} is proved by means of the Picard iteration method, 
	and for this reason we do not recall the proof here.
	
	\smallskip
	
	{\bf (i)} Now, by $(F_1)$ and \eqref{eq:sublin}, it follows
	\begin{equation*}
		\sup_{0 \le s \le t}|\bm{Y}_{\bm{\mu}}(s)| \le C\left(1+(K_F+K_{\bm{u}})(1+\overline{M}_p(T;\bm{\mu})^{\frac{1}{p}})+K_F\int_0^t\sup_{0 \le z \le \tau}|\bm{Y}_{\bm{\mu}}(z)|d\tau\right)
	\end{equation*}
	and then, by Gr\"onwall inequality,	
	\begin{equation*}
		\sup_{0 \le s \le t}|\bm{Y}_{\bm{\mu}}(s)| \le C\left(1+(K_F+K_{\bm{u}})(1+\overline{M}_p(T;\bm{\mu})^{\frac{1}{p}})\right)e^{(K_F+1)T}.
	\end{equation*}
	leading to \eqref{eq:bounds1}. 
	
	Next, let us consider any $\bm{\mu}, \bm{\nu} \in C([0,T];\W_p(\R^{2d}))$. Then, by $(F_2)$ and $(\cA_2)$ we have
	\begin{equation*}
		\sup_{0 \le s \le t}\left|\bm{Y}_{\bm{\mu}}(s)-\bm{Y}_{\bm{\nu}}(s)\right| \le L_F\int_0^t\sup_{0 \le r \le \tau}\left|\bm{Y}_{\bm{\mu}}(r)-\bm{Y}_{\bm{\nu}}(r)\right|d\tau+(L_F+L_{\bm{u}})\int_0^t \sup_{0 \le r \le \tau}\W_p(\mu_r,\nu_r)d\tau,
	\end{equation*}
	Hence, once again, by Gr\"onwall's inequality, we infer
	\begin{equation*}
		\sup_{0 \le s \le t}\left|\bm{Y}_{\bm{\mu}}(s)-\bm{Y}_{\bm{\nu}}(s)\right| \le T(L_F+L_{\bm{u}})e^{L_FT}\sup_{0 \le s \le t}\W_p(\mu_s,\nu_s).
	\end{equation*}
	that is \eqref{eq:bounds3}.
	
	Finally, we fix $0 \le t \le s \le T$ and observe that
	\begin{align*}
		|\bm{Y}_{\bm{\mu}}(t)-\bm{Y}_{\bm{\mu}}(s)| &\le \int_s^t |F[\tau,\bm{\mu}](\bm{Y}_{\bm{\mu}})|\, d\tau+\int_s^t|\bm{u}(\tau,\bm{\mu})|\, d\tau\\
		&\le (K_F+K_{\bm{u}}+C_1(K_F+1))(1+(\overline{M}_p(T,\bm{\mu}))^{\frac{1}{p}})|t-s|.
	\end{align*}

	\smallskip
	
	\textbf{(ii)} By the assumptions of $F$ and $\bm{u}$, it is sufficient to verify the uniform convergence on $\bm{Y}$. We consider a fixed $\bm{\mu}$ and $\bm{u}$, and we assume there exists a sequence $\{\bm{u}^j\}_{j \in \N}$ satisfying the assumptions. 
		Furthermore, we denote by $(\bm{Y}^j_{\bm{\mu}},\bm{W}^j_{\bm{\mu}})$ and $(\bm{Y}_{\bm{\mu}},\bm{W}_{\bm{\mu}})$ the solutions of the ODE 
		\eqref{PDEODE2} respectively with $\bm{u}_j$ and $\bm{u}$. Then, 
		for all $0 \le s \le t \le T$
		\begin{align*}
			\left|\bm{Y}^j_{\bm{\mu}}(s)-\bm{Y}_{\bm{\mu}}(s)\right| &\le L_F\int_0^s\sup_{0 \le r \le \tau}\left|\bm{Y}^j_{\bm{\mu}}(r)-\bm{Y}_{\bm{\mu}}(r)\right|\, d\tau +\left|
			\int_0^s\left(\bm{u}^j(\tau,\bm{\mu})-\bm{u}(\tau,\bm{\mu})\right)ds\right|,
		\end{align*}
	Now we take the supremum over $[0,t]$ and we achieve 
	\begin{align*}
		\sup_{0 \le s \le t}\left|\bm{Y}^j_{\bm{\mu}}(s)-\bm{Y}_{\bm{\mu}}(s)\right| &\le (L_F+1)\int_0^t \sup_{0 \le r \le \tau}\left|\bm{Y}^j_{\bm{\mu}}(r)-\bm{Y}_{\bm{\mu}}(r)\right|\, d\tau +\mathcal{R}_j(t),
	\end{align*}
	where
	\begin{equation*}
		\mathcal{R}_j(t):=\sup_{0 \le s \le t}\left|\int_0^s\left(\bm{u}^j(s,\bm{\mu})-\bm{u}(s,\bm{\mu})\right)ds\right|,
	\end{equation*}
	which, by Gr\"onwall's inequality, implies 
	\begin{equation*}
		\sup_{0 \le s \le t}|\bm{Y}_{\bm{\mu}}^j(s)-\bm{Y}_{\bm{\mu}}(s)| \le \mathcal{R}_j(t)e^{(L_F+1)T}.
	\end{equation*}
	We only need to prove that $\lim_{j \to \infty}\mathcal{R}_j(T)=0$. To do this, consider the sequence of functions
	\begin{equation*}
		\bm{J}_j(t)=\int_0^t(\bm{u}^j(s,\bm{\mu})-\bm{u}(s,\bm{\mu}))\, ds.
	\end{equation*}
	By assumption, we have that for all $t \in [0,T]$ it holds $\lim_{j \to \infty}\bm{J}_j(t)=0$. Furthermore, it is clear that $\mathcal{R}_j(T)=\sup_{0 \le t \le T}\left|\bm{J}_j(t)\right|$, hence we want to prove that $\bm{J}_j \to 0$ uniformly in $[0,T]$. However, this is clear since the sequence $\bm{J}_j$ is equi-Lipschitz. Indeed, for all $0 \le s \le t \le T$,
	\begin{align*}
	\left|\bm{J}_j(t)-\bm{J}_j(s)\right|&=\left|\int_s^t(\bm{u}^j(\tau,\bm{\mu})-\bm{u}(\tau,\bm{\mu}))\, d\tau\right| \le \int_s^t\left|\bm{u}^j(\tau,\bm{\mu})-\bm{u}(\tau,\bm{\mu})\right| \, d\tau\\
	&\le 2K_{\bm{u}}\left(1+\left(\overline{M}_p(t;\bm{\mu})\right)^{\frac{1}{p}}\right)|t-s|.
	\end{align*}
	This shows \eqref{eq:limits}.
	
	\smallskip

	{\bf (iii)} The last statement follows by uniqueness of the solution of \eqref{PDEODE2}, together with assumptions $(F_2)$ and $(\cA_2)$, 
	which in turn imply that if $\mu_s^1=\mu_s^2$ for all $s \in [0,t]$, then $F[s,\bm{\mu}^1](y)=F[s,\bm{\mu}^2](y)$ and 
	$\bm{u}(s,\bm{\mu}^1)=\bm{u}(s,\bm{\mu}^2)$ for all $s \in [0,t]$ and $h \in \R^{2m}$.
\end{proof}

Then, for reader's convenience, we recall the definition of the map $\mathcal{S}:C([0,T];\W_p(\R^{2d})) \times \cA \times [0,T] \to \R^{2m}$:
\begin{equation*}
	\mathcal{S}[\bm{\mu},\bm{u}](t)=\bm{H}_{\bm{\mu}}(t),
\end{equation*}
where $\bm{H}_{\bm{\mu}}$ is the unique solution of \eqref{PDEODE2}. Furthermore, as already done in Section \ref{sec:optimal}, 
for any fixed $\bm{u} \in \cA$, we also define the map
\begin{equation*}
	G_{\bm{u}}:[0,T]\times C([0,T];\W_p(\R^{2d})) \times \R^{2d} \to \R^{d}
\end{equation*}
as follows:
\begin{equation*}
G_{\bm{u}}[t,\bm{\mu}](z)=\mathfrak{v}[t,\bm{\mu}](z)+\mathfrak{w}[t,\mathcal{S}[\bm{\mu},\bm{u}]](z).
\end{equation*}
Concerning $G_{\bm{u}}$, we can prove the following result.

\begin{lem}\label{lem:PDEODEWell2}
	Fix $\bm{u} \in \cA$. Then $G_{\bm{u}}$ satisfies Assumptions \ref{ass:v}, where the constants are independent of the choice of $\bm{u} \in \cA$.
\end{lem}
\begin{proof}
	Since $\mathfrak{v}$ already satisfies Assumptions \ref{ass:v}, it is sufficient to verify that $\mathfrak{w}[\cdot,\mathcal{S}[\cdot,\bm{u}]](\cdot)$ satisfies the same assumptions. First, we notice that $(\mathfrak{v}_0)$ is satisfied. Indeed, by \eqref{eq:bounds3} and the assumptions on $F$, we know that $\mathcal{S}[\cdot,\bm{u}]$ is continuous. Since by $(\mathfrak{w}_0)$ is a Carath\'eodory map, the composition $\mathfrak{w}[\cdot,\mathcal{S}[\cdot,\bm{u}]](\cdot)$ is also a Carath\'eodory map. Next, we prove $(\mathfrak{v}_1)$. To do this, observe that by $(\mathfrak{w}_1)$ 
	\begin{equation*}
		|\mathfrak{w}[t,\mathcal{S}[\bm{\mu},\bm{u}]](z)| \le K_{\mathfrak{w}}(1+\sup_{0 \le s \le t}|\mathcal{S}[\bm{\mu},\bm{u}](s)|+|x|^{\frac{\beta}{3}}+|v|^{\beta}),
	\end{equation*}
	and then, by \eqref{eq:bounds1} and the assumptions on $F$ and $\cA$ we have
	\begin{equation*}
		|\mathfrak{w}[t,\mathcal{S}[\bm{\mu},\bm{u}]](z)| \le K_{\mathfrak{w}}(1+K_F)(1+C_1+K_F+K_{\bm{u}})(1+\overline{M}_p(T;\bm{\mu})^{\frac{1}{p}}+|x|^{\frac{\beta}{3}}+|v|^{\beta}),
	\end{equation*}
	Assumption $(\mathfrak{v}_2)$ clearly follows by $(\mathfrak{w}_2)$. Now we prove $(\mathfrak{v}_3)$. To do this, let $\bm{\mu}^1,\bm{\mu}^2 \in \W_p(C([0,T];\R^{2d}))$, $z_1,z_2 \in \R^{2d}$ and observe that, by $(\mathfrak{w}_2)$
	\begin{align*}
		|\mathfrak{w}[t,\mathcal{S}[\bm{\mu}^1,\bm{u}]](z_1)-&\mathfrak{w}[t,\mathcal{S}[\bm{\mu}^2,\bm{u}]](z_2)| \le L_{\mathfrak{w}}(\sup_{0 \le s \le t}|\mathcal{S}[\bm{\mu}^1,\bm{u}](s)-\mathcal{S}[\bm{\mu}^2,\bm{u}](s)|+|z_1-z_2|)\\
		&\le L_{\mathfrak{w}}(1+L_F+L_{\bm{u}}(1+C_2)(\sup_{0 \le s \le T}\W_p(\mu_s^1,\mu_s^2)+|z_1-z_2|),
	\end{align*}
	where we also used \eqref{eq:bounds3} and the assumptions on $F$ and $\cA$. The latter inequality implies $(\mathfrak{v}_3)$ (see Appendix~\ref{appendix:equivalence} for more details). Finally, $(\mathfrak{v}_4)$ follows once we observe that if $\mu^1_s=\mu^2_s$ for all $s \in [0,t]$, then $\mathcal{S}[\bm{\mu}^1,\bm{u}](s)=\mathcal{S}[\bm{\mu}^2,\bm{u}](s)$ for all $s \in [0,t]$ by Lemma \ref{lem:PDEODEWell1}.
\end{proof}

%
%
The following lemma is a consequence of the Arzel\'a-Ascoli Theorem (see \cite[Page 81]{kelley1963linear}) and the characterization of compact subsets of $\W_p(\R^{2d})$ (see \cite[Proposition 7.1.5]{ambrosio2005gradient}). It will be useful in the setting of the control problem.
\begin{lem}\label{lem:PDEODEWell3}
	$\mathcal{K}$ and $\mathcal{K}_1$ are compact subsets respectively of $C([0,T];\W_p(\R^{2d}))$ and $C([0,T];\R^{m})$.
\end{lem}
\begin{proof}
	Concerning $\cK_1$, we notice that the functions belonging to it are equibounded and equi-Lipschitz, hence we can apply the Arzel\'a-Ascoli Theorem. Concerning $\cK$, let us first observe that the curves of probability mesures belonging to it are equi-H\"older with respect to the Wasserstein distance, hence equi-continuous. Furthermore, it is clear that $\cK$ is closed. Hence, it is only necessry to show that for for all $t \in [0,T]$ the set
	\begin{equation*}
		\cK_t:=\{\mu_t \in \W_p(\R^{2d}): \ \bm{\mu} \in \cK\}
		\end{equation*} 
	is relatively compact. Notice that $\cK_t \subset \W_p(\R^{2d})$. We now prove that $\cK_t$ is uniformly tight. Fix $\varepsilon>0$ and let $\delta=\frac{2(\mathcal{C}(K_G))^{\frac{1}{p}}}{\varepsilon}$. Then, by Markov's inequality, setting $B_\delta:=\{z \in \R^{2d}: \ |z|<\delta\}$,
	\begin{equation*}
		\mu_t(\R^{2d}\setminus B_\delta) \le \frac{M_1(\mu_t)}{\delta} \le \frac{M_p(\mu_t)^{\frac{1}{p}}}{\delta} \le \frac{(\mathcal{C}(K_G))^{\frac{1}{p}}}{\delta}<\varepsilon.
	\end{equation*}
	Since $\delta$ is independent of the choice of $\mu_t \in \cK_t$, this shows that $\cK_t$ is uniformly tight. Now we show that $\cK_t$ has uniformly integrable $p$-moments. To do this, we claim that there exists a probability measure $\widetilde{{\sf P}}$ on $(\R^{2d},\cB(\R^{2d}))$ such that for all $\mu_t \in \cK_t$ there exists a random variable $Z$ on the probability space $(\R^{2d},\cB(\R^{2d}),\widetilde{{\sf P}})$ with ${\rm Law}(Z)=\mu_t$. Despite this property is well-known in measure theory, we provide here a proof for completeness. Recall that $(\R^{2d},\cB(\R^{2d}))$ and $([0,1], \cB([0,1]))$ are uncountable Polish spaces, hence by Kuratowski's Isomorphism Theorem (see \cite[Theorem 15.6]{kechris2012classical}) we know that there exists a Borel isomorphism $f:[0,1] \to \R^{2d}$. We equip $([0,1], \cB([0,1]))$ with the Lebesgue measures ${\sf Leb}$ and we define $\widetilde{\sf P}=f \sharp {\sf Leb}$, i.e. the pushforward of ${\sf Leb}$ on $(\R^{2d},\cB(\R^{2d}))$ through the Borel isomorphism $f$. Furthermore, for any $\mu_t \in \cK_t$, denote $\widetilde{\mu}_t=f^{-1}\sharp \mu_t$ the pushforward of $\mu_t$ on $([0,1],\cB([0,1]))$ through the Borel isomorphism $f^{-1}$. Notice that both ${\sf P}$ and $\widetilde{\mu}_t$ are probability measures. In particular, by Skorokhod's representation theorem (see \cite[Section 3.12]{williams1991probability}) we know that there exists a random variable $\widetilde{Z}$ on $([0,1],\cB([0,1]),{\sf Leb})$ such that $\widetilde{\mu}_t={\rm Law}(\widetilde{Z})$. We set $Z=f\circ \widetilde{Z}$, which is measurable since it is composition of measurable functions. Furthermore, for any Borel set $B \in \cB(\R^{2d})$, we have
	\begin{equation*}
		\widetilde{\sf P}(Z \in B)={\sf Leb}(\widetilde{Z} \in f^{-1}(B))=\widetilde{\mu}_t(f^{-1}(B))=\mu_t(B),
	\end{equation*}
	i.e. ${\rm Law}(Z)=\mu_t$. Let 
	\begin{equation*}
		\widetilde{\cK}_t:=\{Z \in \cM(\R^{2d},\cB(\R^{2d}),\widetilde{\sf P}): \ {\rm Law}(Z) \in \cK_t\}.
	\end{equation*}
	Then, by the previous argument, $\widetilde{\cK}_t \not = \emptyset$ and for all $\mu_t \in \cK_t$ there exists $Z \in \widetilde{\cK}_t$ such that ${\rm Law}(Z)=\mu_t$. Hence, to prove that $\cK_t$ has uniformly integrable $p$-moments, it is sufficient to show that $\widetilde{\cK}_t$ is uniformly $L^p$-integrable, or, equivalently, that $\{|Z|^p, \ Z \in \widetilde{\cK}_t\}$ is uniformly integrable. However, it is clear that, for any $Z \in \widetilde{\cK}_t$ with ${\rm Law}(Z)=\mu_t$ for some $\bm{\mu} \in \cK$,
	\begin{equation*}
		\E\left[\widetilde{\Phi}_p(|Z|;K_G)\right] \le \overline{M}_{\widetilde{\Phi}_p(\cdot;K_G)}(T,\bm{\mu}) \le \mathcal{C}(K_G),
	\end{equation*}
	where the right-hand side of the previous inequality is independent of $Z$. Hence, by the de la Vall\'ee-Poussin Theorem, $\{|Z|^p, \ Z \in \widetilde{\cK}_t\}$ is uniformly integrable and then $\cK_t$ has uniformly integrable $p$-moments. By \cite[Proposition 7.1.5]{ambrosio2005gradient}, we know that $\cK_t$ is relatively compact, hence $\cK$ is equicontinuous, closed and pointwise relatively compact, which in turn implies, by the Arzel\'a-Ascoli theorem, that $\cK$ is compact. 
\end{proof}

Finally, we observe that $\cK^\prime:=\cK \times \cK_1$ is a compact subset of $C([0,T];\W_p(\R^{2d}))\times C([0,T];\R^{m})$. Then we are able to describe the behaviour of the field $G_{\bm{u}}$ with respect to suitable variations of the control $\bm{u} \in \cA$.
\begin{lem}\label{lem:PDEODEWell4}
	Let $\{\bm{u}_j\}_{j \in \N}\subset \cA$ and $\bm{u} \in \cA$ be such that
	\begin{equation*}
		\lim_{j \to \infty}\int_0^t \bm{u}_j(s,\bm{\mu})\, ds=\int_0^t \bm{u}(s,\bm{\mu})\, ds, \quad \forall t \in [0,T], \ \forall \bm{\mu} \in \cK.
	\end{equation*}
	Then, for any $t \in [0,T]$, $\bm{\mu} \in \cK$ and $z \in \R^{2d}$
	\begin{equation*}
		\lim_{j \to \infty}G_{\bm{u}_j}[t,\bm{\mu}](z)=G_{\bm{u}}[t,\bm{\mu}](z).
	\end{equation*}
\end{lem}
\begin{proof}
	Fix $\bm{\mu} \in \cK$, $t \in [0,T]$ and $z \in \R^{2d}$. Then we have
	\begin{align*}
		|G_{\bm{u}_j}[t,\bm{\mu}](z)-G_{\bm{u}}[t,\bm{\mu}](z)|&=|\mathfrak{w}[t,\mathcal{S}[\bm{\mu},\bm{u}_j]](z)-\mathfrak{w}[t,\mathcal{S}[\bm{\mu},\bm{u}]](z)|\\
		&\le L_{\mathfrak{w}}\sup_{0 \le s \le t}|\mathcal{S}[\bm{\mu},\bm{u}_j](s)-\mathcal{S}[\bm{\mu},\bm{u}](s)|.
	\end{align*}
	Taking the limit as $j \to \infty$, we get the statement by Lemma \ref{lem:PDEODEWell1}.
\end{proof}

\newpage

\section{Other Possible Control Classes}\label{sec:opcc}
It could be possible, for modelling restrictions, that we can apply controls with some more restrictive assumptions than the ones considered for $\cA$. Hence, in this section, we consider three possible subclasses of $\cA$. Since the proofs of the existence of the solutions to the related optimal control problems differ from Theorem \ref{main2} only of some technical details, we will only underline the parts in which the proof are actually different.

\subsection{Controls satisfying the Oscillation Restriction Criterion}

First, let us assume that the controls cannot \textit{oscillate too much}. Precisely, we consider the class of admissible controls $\cA_{\rm ORC} \subset \cA$ satisfying the additional condition
\begin{itemize}
	\item[$(\cA_{\rm ORC})$] For all compact sets $\cK \subset C([0,T];\W_p(\R^{2d}))$, all measurable sets $A \subseteq [0,T]$ and all $\varepsilon>0$, there exists a measurable set $B \subset A$ such that
	\begin{equation*}
		\frac{1}{|B|}\int_{B} \sup_{\bm{\mu} \in \cK}\left|\bm{u}(t,\bm{\mu})-\frac{1}{|B|}\int_{B}\bm{u}(s,\bm{\mu})\, ds\right|\, dt < \varepsilon.
	\end{equation*}
\end{itemize}
Assumption $(\cA_{\rm ORC})$ can be recognized as a special form of the \textit{Oscillation Reduction Criterion}, introduced in \cite{girardi1991weak} for real-valued functions and then discussed in \cite{j1994weak} for Bochner integrals. With this assumption, we can actually relax the requirements on the cost functional $\F$. Precisely, we assume that $\F$ satisfies $(\F_0)$, $(\F_1)$, $(\F_2)$ and
\begin{itemize}
	\item[$(\F_{\rm ORC})$] $\Psi:\R^m \to \R$ is continuous.
\end{itemize}
We are interested in the following optimal control problem:

\begin{ProblemSpecBox}{Problem $2$}{$2$}\label{prob:2}
	Find $\bm{u}^\star \in \cA_{\rm ORC}$ such that
	\begin{equation*}
		\F[\bm{u}^\star]=\min_{\bm{u} \in \cA_{\rm ORC}}\F[\bm{u}].
	\end{equation*}
\end{ProblemSpecBox}

\begin{thm}\label{thm:ORC}
	Under $(\F_0)$, $(\F_1)$, $(\F_2)$ and $(\F_{\rm ORC})$, Problem \ref{prob:2} admits at least a solution.
\end{thm}
\begin{proof}
	We proceed as in the proof of Theorem \ref{main2} except that this time $\bm{U}^n \to \bm{U}^\star$ in measure by \cite[Theorem 1.4]{balder1997extension}. Up to a non-relabelled subsequence, we can assume $\bm{U}^n(t) \to \bm{U}^\star(t)$ in $C(\cK;\R^{m})$ for a.a. $t \in [0,T]$. Moreover, if we denote by $(\bm{\mu}^n,\bm{Y}^n)$ the solutions of \eqref{PDEODE} with controls $\bm{u}^n$ and $(\bm{\mu}^\star,\bm{Y}^\star)$ the solution of \eqref{PDEODE} with control $\bm{u}^\star$, we also know that, by Theorem \ref{thm:PDEODEWell}, $\bm{\mu}^n \to \bm{\mu}^\star$ in $C([0,T];\W_p(\R^{2d}))$. Hence $\lim_{n \to \infty}\bm{U}^n(t)(\bm{\mu}^n)=\bm{U}^\star(t)(\bm{\mu}^\star)$ for a.a. $t \in [0,T]$. Furthermore, recall that $\left|\Psi(\bm{U}^n(t)(\bm{\mu}^n))\right| \le M_\Psi$, where $M_\Psi$ is defined in the proof of Theorem \ref{main2}. Hence, by dominated convergence, we get
	\begin{equation*}
		\lim_{n \to \infty}\int_0^T \Psi(\bm{U}^n(t)(\bm{\mu}^n))\, dt = \int_0^T \Psi(\bm{U}^\star(t)(\bm{\mu}^\star))\, dt.
	\end{equation*}
	The remainder of the proof is exactly the same sa the one of Theorem \ref{main2}.
\end{proof}

\subsection{Controls with current-state dependence on $\bm{\mu}$}
It could be the case that the control we want to apply cannot depend on the whole trajectory of $\bm{\mu}$, but only on the current state. To model this case, we consider the class of admissible controls $\cA_{\rm EV} \subset \cA$ satisfying the additional condition:
\begin{itemize}
	\item[$(\cA_{\rm EV})$] There exists a function $\bm{u}^\sharp: [0,T]\times \W_p(\R^{2d}) \to \R^m$ such that $\bm{u}(t,\bm{\mu})=\bm{u}^\sharp(t,\mu_t)$ for all $\bm{\mu} \in C([0,T];\W_p(\R^{2d}))$ and
	\begin{equation*}
		|\bm{u}^\sharp(t,\mu)-\bm{u}^\sharp(t,\nu)| \le \frac{L_{\bm{u}}}{m}\W_p(\mu,\nu)
	\end{equation*}
	for all $t \in [0,T]$ and $\mu,\nu \in \W_p(\R^{2d})$.
\end{itemize}
We consider then the optimal control problem:
\begin{ProblemSpecBox}{Problem $3$}{$3$}\label{prob:3}
	Find $\bm{u}^\star \in \cA_{\rm EV}$ such that
	\begin{equation*}
		\F[\bm{u}^\star]=\min_{\bm{u} \in \cA_{\rm EV}}\F[\bm{u}].
	\end{equation*}
\end{ProblemSpecBox}

\begin{thm}\label{thm:EV}
	Under $(\F_0)$, $(\F_1)$, $(\F_2)$ and $(\F_3)$, Problem \ref{prob:3} admits at least a solution.
\end{thm}
\begin{proof}
	Recall that, arguing as in Theorem \ref{main2}, if we consider a minimizing sequence $\{\bm{u}^n\}_{n \in \N}\subset \cA_{\rm EV}$, then there exists a function $\bm{u}^\star \in \cA$ such that
	\begin{equation*}
		\lim_{n \to \infty}\sup_{\bm{\mu} \in \cK}\frac{1}{|E|}\left|\int_{E} \left(\bm{u}^n(t,\bm{\mu})-\bm{u}^\star(t,\bm{\mu})\right)\right| \, dt =0. 
	\end{equation*}
	Hence, to prove the statement, it is sufficient to show that $\bm{u}^\star \in \cA_{\rm EV}$. To do this, consider $\bm{\mu},\bm{\nu} \in \cK$ and observe that, by $(\cA_{\rm EV})$, for any $j=1,\dots,m$, $t \in E_{\sf Leb}$ and $h>0$ small enough,
	\begin{equation*}
		\frac{1}{h}\left|\int_{t}^{t+h} \left(u_j^n(s,\bm{\mu})-u_j^\star(s,\bm{\nu})\right) \, ds\right| \le \frac{L_{\bm{u}}}{m}\frac{1}{h}\int_t^{t+h}\W_p(\mu_s,\nu_s)\, ds. 
	\end{equation*}
	Taking the limit as $n \to \infty$ and then as $h \downarrow 0$, we get
	\begin{equation}\label{eq:Lipschitz1ujst}
		\left|u_j^\star(t,\bm{\mu})-u_j^\star(t,\bm{\nu})\right| \le \frac{L_{\bm{u}}}{m}\W_p(\mu_t,\nu_t). 
	\end{equation}
	This holds for all $t \in [0,t]$ once we assume that $u_j^\star(t,\cdot)\equiv 0$ for $t \not \in E_{\sf Leb}$. Now observe that if $\bm{\mu} \in \cK$, then $\mu_t \in \cK^\sharp$, where
	\begin{equation*}
		\cK^\sharp=\left\{\mu \in \W_p(\R^{2d}): \ M_p(\mu)+M_{\widetilde{\Phi}(\cdot;K_G)}(\mu)\le \mathcal{C}(K_G)\right\}.
	\end{equation*}
	On the other hand, if $\mu \in \cK^\sharp$, then $\bm{\mu}^\sharp \in C([0,T];\W_p(\R^{2d}))$ defined as $\mu^\sharp_t \equiv \mu$ for all $t \in [0,t]$ belongs to $\cK$. Hence, we define $\bm{u}^{\sharp,\star}:\cK^\sharp \to \R^m$ as follows
	\begin{equation*}
		\bm{u}^{\sharp,\star}(t,\mu)=\bm{u}^\star(t,\bm{\mu}^\sharp).
	\end{equation*}
	Then, for all $\bm{\mu} \in \cK$ and $t \in [0,T]$, by a simple application of \eqref{eq:Lipschitz1ujst}, it must hold
	\begin{equation}\label{eq:ustar}
		\bm{u}^\star(t,\bm{\mu})=\bm{u}^{\sharp,\star}(t,\mu_t).
	\end{equation}
	Using McShane's Extension Theorem as in the proof of Theorem \ref{main2}, we extend $\bm{u}^{\sharp,\star}(t,\cdot)$ to the whole space $\W_p(\R^{2d})$ and then we define $\bm{u}^\star(t,\cdot)$ on $C([0,T];\W_p(\R^{2d}))$ by means of \eqref{eq:ustar}. As a consequence, $\bm{u}^\star \in \cA_{\rm EV}$, concluding the proof.
\end{proof}

\subsection{Controls with separated variables} 
Among the controls in $\mathcal{A}_{\rm EV}$, one could consider the ones whose dependence on $t$ and $\mu_t$ is \textit{separated}. To do this, let $\ell \in \N$ and fix three positive constants $M_{\bm{h}}$, $M_{\bm{g}}$ and $L_{\bm{g}}$. We define the set $\mathcal{A}^\prime$ as the class of couples $(\bm{h},\bm{g})$ satisfying the following assumptions:
\begin{ProblemSpecBox}{Assumptions on $(\bm{h},\bm{g}) \in \mathcal{A}^\prime$: $(\mathcal{A}^\prime)$}{$(\mathcal{A}^\prime)$}\label{ass:Asep}
	\begin{itemize}
		\item[$(\mathcal{A}^\prime_0)$] $\bm{h} \in L^1([0,T];\R^{m \times \ell})$ and $\bm{g} \in C(\W_p(\R^{2d});\R^{\ell})$.
		\item[$(\mathcal{A}^\prime_1)$] For all $t \in [0,T]$ it holds
		\begin{equation*}
			|\bm{g}(\delta_0)| \le M_{\bm{g}},
		\end{equation*}
		where $\bm{\delta}_0 \in \W_p(\R^{2d})$ is the Dirac delta measure concentrated on $0 \in\R^{2d}$.
		\item[$(\cA^\prime_2)$] It holds
		\begin{equation*}
			|g_j(\mu)-g_j(\nu)| \le \frac{L_{\bm{g}}}{m}\W_p(\mu,\nu),
		\end{equation*}
		for all $\mu,\nu \in \W_p(\R^{2d})$, $j=1,\dots,m$ and $t \in [0,T]$.
		\item[$(\cA^\prime_3)$] For almost all $t \in [0,T]$ it holds
		\begin{equation*}
			|\bm{h}(t)| \le M_{\bm{h}},
		\end{equation*}
		where $|\bm{h}(t)|$ is the Frobenius norm of $\bm{h}(t)$.
	\end{itemize}
\end{ProblemSpecBox}
Now fix $(\bm{h},\bm{g}) \in \cA^\prime$ and consider the function $\bm{u}: (t,\bm{\mu}) \mapsto \bm{h}(t)\bm{g}(\mu_t)$. Clearly, $\bm{u}$ is a Carath\'eodory map. Furthermore, since the Frobenius norm is sub-multiplicative, for any $t \in [0,T]$
\begin{equation*}
	|\bm{u}(t,\bm{\delta}_0)| \le |\bm{h}(t)||\bm{g}(\delta_0)| \le M_{\bm{h}}M_{\bm{g}},
\end{equation*}
while, denoting by $\bm{h}_j$ the $j$-th row of $\bm{h}$, for all $t \in [0,T]$ and $\bm{\mu},\bm{\nu} \in C([0,T];\W_p(\R^{2d}))$,
\begin{equation*}
	|u_j(t,\bm{\mu})-u_j(t,\bm{\nu})| \le |\bm{h}_j(t)||\bm{g}(\mu_t)-\bm{g}(\nu_t)| \le M_{\bm{h}}L_{\bm{g}}\sup_{0 \le t \le T}\W_p(\mu_t,\nu_t),
\end{equation*}
hence for $M_{\bm{u}}=M_{\bm{h}}M_{\bm{g}}$ and $L_{\bm{u}}=mM_{\bm{h}}L_{\bm{g}}$, we have that $\bm{u} \in \cA_{\rm EV}$. Setting $M_{\bm{u}}$ and $L_{\bm{u}}$ as declared, we can define the call of admissible controls $\cA_{\rm SV} \subset \cA_{\rm EV}$ satisfying the additional condition
\begin{itemize}
	\item[$(\cA_{\rm SV})$] There exists $(\bm{h},\bm{g}) \in \cA^\prime$ such that $\bm{u}(t,\bm{\mu})=\bm{h}(t)\bm{g}(\mu_t)$.
\end{itemize}

We are interested in the following optimal control problem:

\begin{ProblemSpecBox}{Problem $4$}{$4$}\label{prob:4}
	Find $\bm{u}^\star \in \cA_{\rm SV}$ such that
	\begin{equation*}
		\F[\bm{u}^\star]=\min_{\bm{u} \in \cA_{\rm SV}}\F[\bm{u}].
	\end{equation*}
\end{ProblemSpecBox}

Before proceeding with the proof of the existence of a solution for Problem \ref{prob:4}, let us notice that, with a similar argument as in the proof of Theorem \ref{thm:EV}, it is possible to reduce the problem on $\cA_{\rm SV}$ to an analogous one on $\cA^\prime$. Indeed, consider the function $\Psi^\prime:\R^{m \times \ell} \times \R^{\ell} \to \R$ and the functional $\F^\prime:\cA^\prime \to \R$ defined as follows:
\begin{equation}\label{eq:Psiprime}
	\Psi^\prime(\bm{h},\bm{g})=\Psi(\bm{h}\bm{g}), \quad \forall (\bm{h},\bm{g}) \in \R^{m \times \ell} \times \R^\ell
\end{equation}
and
\begin{equation*}
	\F^\prime[\bm{h},\bm{g}]=\int_0^T \cL(t,\bm{\mu},\bm{Y})\, dt+\int_0^T\Psi^\prime(\bm{h}(t), \bm{g}(\mu_t))\, dt, \quad \forall (\bm{h},\bm{g}) \in \cA^\prime,
\end{equation*}
where $(\bm{\mu},\bm{Y})$ is the solution of \eqref{PDEODE} with control $\bm{u}(t,\bm{\mu})=\bm{h}(t)\bm{g}(\mu_t)$. Notice that if $\bm{u} \in \cA_{\rm SV}$, then $\bm{u}(t,\bm{\mu})=\bm{h}(t)\bm{g}(\mu_t)$ for some $(\bm{h},\bm{g}) \in \cA^\prime$ and $\F[\bm{u}]=\F^\prime[\bm{h},\bm{g}]$. On the other hand, if $(\bm{h},\bm{g}) \in \cA^\prime$, then $\bm{u}(t,\bm{\mu})=\bm{h}(t)\bm{g}(\mu_t)$ belongs to $\cA_{\rm SV}$ and $\F[\bm{u}]=\F^\prime[\bm{h},\bm{g}]$. Hence, solving Problem \ref{prob:4} is equivalent to solving

\begin{ProblemSpecBox}{Problem $5$}{$5$}\label{prob:5}
	Find $(\bm{h}^\star, \bm{g}^\star) \in \cA^\prime$ such that
	\begin{equation*}
		\F^\prime[\bm{h}^\star,\bm{g}^\star]=\min_{(\bm{h},\bm{g}) \in \cA^\prime}\F^\prime[\bm{u}].
	\end{equation*}
\end{ProblemSpecBox}

Notice that if we start from Problem \ref{prob:4}, then $\Psi^\prime$ must satisfy \eqref{eq:Psiprime} in Problem \ref{prob:5}. However, we could directly consider Problem \ref{prob:5} and be less restrictive on the dependence of $\Psi^\prime$ on the separated variables $\bm{h}$ and $\bm{g}$. Indeed, we can assume that $\F^\prime$ satisfies the following assumptions:
\begin{ProblemSpecBox}{Assumptions on $\F^\prime$: $(\F^\prime)$}{$(\F^\prime)$}\label{ass:Fprime}
	\begin{itemize}
		\item[$(\mathcal{F}^\prime_0)$] $\cL:[0,T]\times C([0,T];\W_p(\R^{2d})) \times C([0,T];\R^{2m})$ and $\Psi^\prime:\R^{m \times \ell} \times \R^\ell \to \R$ are measurable and bounded from below;
		\item[$(\mathcal{F}_1)$] $\cL$ is continuous in the variables $(\bm{\mu},\bm{Y})$.
		\item[$(\F_2)$] There exists a function $M_{\cL} \in L^1(0,T)$ such that for a.a. $t \in [0,T]$ and all $(\bm{\mu},\bm{Y}) \in \cK^\prime$
		\begin{equation*}
			|\cL(t,\bm{\mu},\bm{Y})| \le M_{\cK^\prime}(t)
		\end{equation*} 
		\item[$(\F^\prime_3)$] $\Psi:\R^{m \times \ell} \times \R^\ell \to \R$ is convex in the first variable and continuous in the second one.
	\end{itemize}
\end{ProblemSpecBox}

\begin{thm}\label{thm:SV}
	Problem \ref{prob:5} admits at least a solution.
\end{thm}
\begin{proof}
	Arguing as in Theorem \ref{main2}, since for all $\bm{\mu} \in \cK$ it holds $\mu_t \in \cK^\prime$ for any $t \in [0,T]$, solving Problem \ref{prob:5} is equivalent to solving  
	\begin{ProblemSpecBox}{Problem $5^\star$}{$5^\star$}\label{prob:5star}
		Find $(\bm{h}^\star, \bm{g}^\star) \in \widetilde{\cA}^\prime$ such that
		\begin{equation*}
			\widetilde{\F}^\prime[\bm{h}^\star,\bm{g}^\star]=\min_{(\bm{h},\bm{g}) \in \widetilde{\cA}^\prime}\widetilde{\F}^\prime[\bm{h},\bm{g}].
		\end{equation*}
	\end{ProblemSpecBox}
	where $\widetilde{\cA}^\prime=\Pi \cA^\prime$, $\Pi:C([0,T];\R^{m \times \ell}) \times C(\W_p(\R^{2d});\R^\ell) \to C([0,T];\R^{m \times \ell})\times C(\cK^\prime;\R^{\ell})$ is the \textit{domain restriction map}, i.e. for all $\mu \in \cK^\prime$ and $t \in [0,T]$ it holds $\Pi (\bm{h},\bm{g})(t,\mu)=(\bm{h}(t),\bm{g}(\mu))$, and for all $\Pi(\bm{h},\bm{g}) \in \widetilde{\cA}^\prime$ it holds
	\begin{equation*}
		\widetilde{\F}^\prime[\Pi(\bm{h},\bm{g})]=\F^\prime[\bm{h},\bm{g}].
	\end{equation*}
	Now consider a minimizing sequence $\{(\bm{h}^n,\bm{g}^n)\}_{n \in \N}\subseteq \widetilde{\cA}^\prime$. Since $|\bm{h}^n(t)| \le M_{\bm{h}}$ for all $n \in \N$ and almost all $t \in [0,T]$, then by the Dunford-Pettis theorem (see \cite[Theorem 4.30]{brezis2011functional}) we know that there exists a function $\bm{h}^\star \in L^1([0,T];\R^{m \times \ell})$ such that, up to a non-relabelled subsequence, $\bm{h}^n \rightharpoonup \bm{h}^\star$ in $L^1([0,T];\R^{m \times \ell})$ and it is clear, arguing again as in Theorem \ref{main2}, that $|\bm{h}^\star(t)| \le M_{\bm{h}}$ for almost all $t \in [0,T]$. Next, notice that for all $\mu \in \cK^\prime$
	\begin{equation*}
		\left|\bm{g}^n(\mu)\right| \le L_{\bm{g}}(M_p(\mu))^{\frac{1}{p}}+M_{\bm{g}} \le L_{\bm{g}}(\mathcal{C}(K_G))^{\frac{1}{p}}+M_{\bm{g}}=:\widetilde{M}_{\bm{g}}.
	\end{equation*}
	Hence $\{\bm{g}^n\}_{n \in \N} \subset C(\cK^\prime;\R^\ell)$ are equibounded and equiLipschitz (by $(\cA^\prime_2)$). Hence, up to a non-relabelled subsequence, by the Arzel\'a-Ascoli theorem, there exists $\bm{g}^\star \in C(\cK^\prime;\R^{\ell})$ such that $\bm{g}^n \to \bm{g}^\star$. With the same arguments as in Theorem \ref{main2}, one can use the McShane extension theorem to extend $\bm{g}^\star$ to a function in $C(\W_p(\R^{2d});\R^\ell)$ whose components are still $\frac{L_{\bm{g}}}{m}$-Lipschitz, hence proving that $(\bm{h}^\star,\bm{g}^\star) \in \widetilde{\cA}^\prime$. Furthermore, for all $\bm{\mu} \in \cK$, $t \in [0,T]$ and $n \in \N$, set $\bm{u}^n(t,\bm{\mu})=\bm{h}(t)\bm{g}^n(\mu_t)$ and $\bm{u}^\star(t,\bm{\mu})=\bm{h}(t)\bm{g}^\star(\mu_t)$. Then, for all $t \in [0,T]$ and $\bm{\mu} \in \cK$ we have
	\begin{align*}
		\left|\int_0^t \bm{u}^n(s,\bm{\mu})\, ds\right.&\left.-\int_0^t \bm{u}^\star(s,\bm{\mu})\, ds\right| \\
		&\le \int_0^t\left|\bm{h}^n(s)\right|\left|\bm{g}^n(\mu_s)-\bm{g}^\star(\mu_s)\right|\, ds +\left|\int_0^t\bm{g}^\star(\mu_s)\left(\bm{h}^n(s)-\bm{h}^\star(s)\right)\, ds\right|\\
		&\le M_{\bm{h}}T\sup_{\mu \in \cK}|\bm{g}^n(\mu)-\bm{g}^\star(\mu)|+\left|\int_0^t\bm{g}^\star(\mu_s)\left(\bm{h}^n(s)-\bm{h}^\star(s)\right)\, ds\right|.
	\end{align*}
	Hence, recalling that the function $s \in [0,T] \mapsto \bm{g}^\star(\mu_s)$ belongs to $L^\infty([0,T];\R^{\ell})$, we can take the limit to get
	\begin{equation*}
		\lim_{n \to \infty}\left|\int_0^t \bm{u}^n(s,\bm{\mu})\, ds-\int_0^t \bm{u}^\star(s,\bm{\mu})\, ds\right|=0.
	\end{equation*}
	As a consequence, if we denote by $(\bm{\mu}^n,\bm{Y}^n)$ the solutions of \eqref{PDEODE} with controls $\bm{u}^n$ and $(\bm{\mu}^\star,\bm{Y}^\star)$ the solution of \eqref{PDEODE} with control $\bm{u}^\star$, it holds, by Theorem \ref{thm:PDEODEWell}, $(\bm{\mu}^n,\bm{Y}^n) \to (\bm{\mu}^\star,\bm{Y}^\star)$ in $C([0,T];\W_p(\R^{2d})) \times C([0,T];\R^{m})$. Furthermore, notice that

	\begin{multline*}
		\int_0^T \Psi^\prime(\bm{h}^n(t),\bm{g}^n(\mu_t^n))\, dt=\int_0^T\left(\Psi^\prime(\bm{h}^n(t),\bm{g}^n(\mu_t^n))-\Psi^\prime(\bm{h}^n(t),\bm{g}^\star(\mu_t^\star))\right)\, dt \\
		+ \int_0^T \Psi^\prime(\bm{h}^n(t),\bm{g}^\star(\mu_t^\star))\, dt.
	\end{multline*}
	To handle the first integral, we observe that
	\begin{equation*}
		\int_0^T\left|\Psi^\prime(\bm{h}^n(t),\bm{g}^n(\mu_t^n))-\Psi^\prime(\bm{h}^n(t),\bm{g}^\star(\mu_t^\star))\right|\, dt \le \int_0^T\sup_{|\bm{h}| \le M_{\bm{h}}}\left|\Psi^\prime(\bm{h},\bm{g}^n(\mu_t^n))-\Psi^\prime(\bm{h},\bm{g}^\star(\mu_t^\star))\right|\, dt
	\end{equation*}
	and recall that $\left|\bm{g}^n(\mu_t^n)\right|,\left|\bm{g}^\star(\mu_t^\star)\right| \le \widetilde{M}_{\bm{g}}$, so that, by dominated convergence,
	\begin{align}
		\lim_{n \to \infty}&\int_0^T\left|\Psi^\prime(\bm{h}^n(t),\bm{g}^n(\mu_t^n))-\Psi^\prime(\bm{h}^n(t),\bm{g}^\star(\mu_t^\star))\right|\, dt\\
		&=\lim_{n \to \infty} \int_0^T\sup_{|\bm{h}| \le M_{\bm{h}}}\left|\Psi^\prime(\bm{h},\bm{g}^n(\mu_t^n))-\Psi^\prime(\bm{h},\bm{g}^\star(\mu_t^\star))\right|\, dt=0.\label{eq:Psiprime1}
	\end{align}
	On the other hand, arguing exactly as in Theorem \ref{thm:PDEODEWell}, we have
	\begin{equation}\label{eq:Psiprime2}
		\int_0^T \Psi^\prime(\bm{h}^\star(t),\bm{g}^\star(\mu_t^\star))\, dt \le \liminf_{n \to \infty}\int_0^T \Psi^\prime(\bm{h}^n(t),\bm{g}^\star(\mu_t^\star))\, dt.
	\end{equation}
	Combining \eqref{eq:Psiprime1} and \eqref{eq:Psiprime2} we get
	\begin{equation*}
		\int_0^T \Psi^\prime(\bm{h}^\star(t),\bm{g}^\star(\mu_t^\star))\, dt \le \liminf_{n \to \infty}\int_0^T \Psi^\prime(\bm{h}^n(t),\bm{g}^n(\mu_t^\star))\, dt
	\end{equation*}
	and then we proceed as in the proof of Theorem \ref{thm:PDEODEWell}.
\end{proof}
\begin{rem}
	Notice that these are the controls considered in \cite{Ascione20236965} for the first order case, under more restrictive assumptions on the system \eqref{PDEODE}.
\end{rem}

\newpage

\section{An example of application of Theorem \ref{main2} } \label{appendix:example}

In this appendix, we provide an example of second order mean field control problem of the form of Problem \ref{prob:1} arising from a second order multi-agent system. Precisely, let $N,m \in \N$ be respectively the number of followers and leaders. Denote by $(X_N^j,V_N^j)$ the position and velocity of the generic follower, for $j=1,\dots,N$, and by $(Y_N^i,W_N^i)$ the position and velocity of the generic leader, for $i=1,\dots,m$. We assume the dynamics of the leader-follower system is described by the following set of \textit{controlled} SDEs:
\begin{equation}\label{eq:prelimitN}
	\begin{cases}
		dX_N^j(t)=V_N^j(t)dt & \begin{split} j=1,\dots,N, \\ t \in [0,T]\end{split}\\
		\displaystyle \begin{split}
			dV_N^j(t)&=\left(\frac{1}{N}\sum_{i=1}^{N}K_{1,1}(X_N^i(t)-X_N^j(t),V_N^i(t)-V_N^j(t))\right.\\
			&\left.\qquad +\frac{1}{m}\sum_{i=1}^{m}K_{1,2}(Y_N^i(t)-X_N^j(t),W_N^i(t)-V_N^j(t))\right)dt\\
			&\qquad+\sqrt{2\sigma}dB^j(t)
		\end{split} &  \begin{split} j=1,\dots,N, \\ t \in [0,T] \end{split}\\
		\displaystyle
		\begin{split}
		\dot{Y}_N^i(t)&=W_N^i(t)=\frac{1}{N}\sum_{k=1}^{N}K_{2,1}(X_N^k(t)-Y_N^i(t))
			\\
			&+\frac{1}{m}\sum_{k=1}^{m}K_{2,2}(Y_N^k(t)-Y_N^i(t))
			+u_N^i(t,\bm{X}_N(t),\bm{V}_N(t))
		\end{split} 
	 &  \begin{split} i=1,\dots,m, \\ t \in [0,T] \end{split}\\
		X_N^j(0)=X_0^j,\  V_N^j(0)=V_0^j & j=1,\dots,N \\ 
		Y_N^i(0)=Y_0^i & i=1,\dots,m,
	\end{cases}
\end{equation}
where
\begin{itemize}
	\item $\bm{X}_N:=(X^1_N,\dots,X^N_N)$, $\bm{V}_N:=(V^1_N,\dots,V^N_N)$ and $\bm{Z}_N:=(\bm{X}_N,\bm{V}_N)$;
	\item $K_{1,j}:\R^{2d} \to \R^d$ and $K_{2,j}:\R^{d} \to \R^d$ are $L_{\rm Ker}$-Lipschitz and $M_{\rm Ker}$-bounded interaction kernels
	such that $K_{i,j}(-x)=K_{i,j}(x)$;
	\item the initial data $(X_0^j,V_0^j)_{j \in \N}$ are independent and identically distributed (i.i.d.) random variables belonging to $\W_p(\R^{2d})$ for some $p>1$; 
	\item for all $i=1,\dots,m$ there exists a function $u^i:[0,T]\times \W_1(\R^{2d}) \to \R^d$ such that 
	\begin{equation*}
		u_N^i(t,\bm{x},\bm{v})=u^i\left(\frac{1}{N}\sum_{j=1}^{N}\delta_{x_j,v_j}\right),
	\end{equation*}
	where $\delta_{x_j,v_j}$ are Dirac delta measures concentrated in $(x_j,v_j) \in \R^{2d}$ for all $j=1,\dots,N$. In particular, $u^i$ is such that $|u^i(t,\delta_0)| \le M_{\bm{u}}$ for every $i=1, \ldots, m$ and
	\begin{equation*}
		\left|u^i_k(t,\mu)-u^i_k(t,\nu)\right| \le \frac{L_{\bm{u}}}{md}\W_1(\mu,\nu), 
	\end{equation*}
	for all $k=1, \ldots, N$, $t \in [0,t]$ and $\mu,\nu \in \W_1(\R^{2d})$;
	\item the initial data $\{Y_0^i\}_{i=1,\dots,m} \subset \R^{d}$ are deterministic;
	\item $\bm{B}=(B^1(t),\dots,B^N(t))$ is a $dN$-dimensional standard Brownian motion.
\end{itemize}

Notice that by standard theory of SDEs (see \cite[Theorem 5.2.1]{oksendal2013stochastic}) for all $N \in \N$ there exists a unique strong solution $(\bm{X}_N,\bm{V}_N,\bm{Y}_N,\bm{W}_N)$, where $\bm{X}_N=(X_N^j)_{j=1,\dots,N}$, $\bm{V}_N=(V_N^j)_{j=1,\dots,N}$, $\bm{Y}_N=(Y_N^j)_{j=1,\dots,m}$ and $\bm{W}_N=(W_N^j)_{j=1,\dots,m}$. 
We would like to send, in some sense, $N \to \infty$. Heuristically, we expect the dynamic of a single follower to converge towards the solution of the following McKean-Vlasov system of SDEs/ODEs:
\begin{equation}\label{eq:overlineXV}
	\begin{cases}
		d\overline{X}(t)=\overline{V}(t)dt & t \in [0,T]\\
		\displaystyle \begin{split}
			d\overline{V}(t)&=\left(\left(K_{1,1} \ast \mu_t\right)(\overline{X}(t),\overline{V}(t))\vphantom{\frac{1}{m}\sum_{i=1}^{m}K_{1,2}(Y_N^i(t)-\overline{X}(t),W_N^i(t)-\overline{V}(t))}\right.\\
			&\left.\qquad +\frac{1}{m}\sum_{i=1}^{m}K_{1,2}(Y_N^i(t)-\overline{X}(t),W_N^i(t)-\overline{V}(t))\right)dt\\
			&\qquad+\sqrt{2\sigma}dB(t)
		\end{split} &  t \in [0,T]\\
		\displaystyle \begin{split}
			\dot{\overline{Y}}^j(t)&=W_N^j(t)=\left(K_{2,1} \ast \mu_t\right) (\overline{Y}^j(t))
			\\
			&\qquad
			+\frac{1}{m}\sum_{i=1}^{m}K_{2,2}(\overline{Y}^i(t)-\overline{Y}^j(t))
			+u^j(t,\mu_t)
		\end{split} &  \begin{split}
		j=1,\dots,m, \\
		t \in [0,T]
		\end{split} \\
		\overline{X}(0)=\overline{X}_0,\  \overline{V}(0)=\overline{V}_0 \\ 
		\overline{Y}^j(0)=Y_0^j & j=1,\dots,m, \\
		\bm{\mu}={\rm Law}(\overline{X},\overline{V}),
	\end{cases}
\end{equation}
where
\begin{align*}
	(K_{1,1}\ast \mu_t)(x,v)&=\int_{\R^{2d}}K_{1,j}(\xi-x,\nu-v)d\mu_t(\xi,\nu),\\
	(K_{2,1}\ast \mu_t)(y)&=\int_{\R^{2d}}K_{2,j}(\xi-y)d\mu_t(\xi,\nu).
\end{align*}
We set now $\overline{Z}=(\overline{X},\overline{V})$, $\overline{\bm{W}}(t)=(\overline{W}^j(t))_{j=1,\dots,m}$, $\overline{\bm{Y}}(t)=(\overline{Y}^j(t))_{j=1,\dots,m}$, $\overline{\bm{H}}=(\overline{\bm{Y}},\overline{\bm{W}})$, $\bm{u}(t,\bm{\mu})=(u^j(t,\mu_t))_{j=1,\dots,m}$,
\begin{align*}
	F[t,\bm{\mu}](\overline{\bm{Y}})&:=\left(\left(K_{2,1} \ast \mu_t\right) (\overline{Y}^j(t)) +\frac{1}{m}\sum_{i=1}^{m}K_{2,2}(\overline{Y}^i(t)-\overline{Y}^j(t))\right)_{j=1,\dots,m},\\
	\mathfrak{v}[t,\bm{\mu}](\overline{Z})&:=\left(K_{1,1} \ast \mu_t\right)(\overline{X}(t),\overline{V}(t)),\\
	\mathfrak{w}[t,\overline{\bm{H}}](\overline{Z})&:=\frac{1}{m}\sum_{i=1}^{m}K_{1,2}(Y_N^i(t)-\overline{X}(t),W_N^i(t)-\overline{V}(t)),
\end{align*}
and $\bm{Y}_0=(Y_0^j)_{j=1,\dots,m}$. With this notation, we get exactly equation \eqref{eq:MKVSDEaux2}. Notice that, clearly, $\mathfrak{v}$ satisfies $(\cB_0)$ and $(\cB_4)$ and in particular
\begin{align}\label{eq:boundedv}
	\left|\mathfrak{v}[t,\bm{\mu}](z)\right| \le \int_{\R^{2d}}\left|K_{1,1}(\xi-x,\nu-v)\right|d\mu_t(\xi,\nu) \le M_{\rm Ker},
\end{align}
that implies $(\cB_1)$, while
\begin{align}
	\left|\mathfrak{v}[t,\bm{\mu}](z_1)-\mathfrak{v}[t,\bm{\mu}](z_2)\right| &\le \int_{\R^{2d}}\left|K_{1,1}(\xi-x_1,\nu-v_1)-K_{1,1}(\xi-x_2,\nu-v_2)\right|d\mu_t(\xi,\nu) \\
	&\le L_{\rm Ker}(|x_1-x_2|+|v_1-v_2|), \label{eq:Lipschitzv}
\end{align}
which, combined with \eqref{eq:boundedv}, leads to $(\cB_2)$. Furthermore, for any $\bm{\mu}^1,\bm{\mu}^2 \in \W_1(C([0,T];\R^{2d}))$ let $\gamma \in \Pi(\bm{\mu}^1,\bm{\mu}^2)$ and denote by $(Z_1,Z_2) \in L^1(\Omega;C([0,T];\R^{4d}))$ a random variable such that $\gamma={\rm Law}(Z_1,Z_2)$. It holds
\begin{align*}
	\left|\mathfrak{v}[t,\bm{\mu}^1](z)-\mathfrak{v}[t,\bm{\mu}^2](z)\right| &\le \E\left[\left|K_{1,1}(X_1(t)-x,V_1(t)-v)-K_{1,1}(X_2(t)-x,V_2(t)-v)\right|\right]\\
	&\le L_{\rm Ker}\E\left[|Z_1(t)-Z_2(t)|\right].
\end{align*}
Since $\gamma \in \Pi(\bm{\mu}^1,\bm{\mu}^2)$ is arbitrary, taking the infimum we get
\begin{align}\label{eq:preB3}
	\left|\mathfrak{v}[t,\bm{\mu}^1](z)-\mathfrak{v}[t,\bm{\mu}^2](z)\right| \le L_{\rm Ker}\W_1\left(\bm{\mu}_1,\bm{\mu}_2\right). 
\end{align}
Hence, \eqref{eq:preB3} combined with \eqref{eq:Lipschitzv} and Proposition \ref{prop:equivlence} shows that $\mathfrak{v}$ satisfies $(\mathfrak{v}_3)$. It is also clear that $\mathfrak{w}$ satisfies Assumptions \ref{ass:w}. Concerning $F$, the fact that it satisfies Assumptions \ref{ass:F} follows in the same way as we did for $\mathfrak{v}$. Finally, notice that the control $\bm{u} \in \cA_{\rm EV}$ for some fixed constants $M_{\bm{u}}$ and $L_{\bm{u}}$. Hence, by Theorem \ref{thm:PDEODEWell}, there exist a unique solution $(\overline{\bm{\mu}},\overline{\bm{Y}})$ of \eqref{PDEODE} and $\overline{\bm{\mu}}={\rm Law}(\overline{X},\overline{V})$ where $\overline{Z}=(\overline{X},\overline{V})$ solves \eqref{eq:overlineXV} together with $\overline{\bm{Y}}$. The link between the solutions of \eqref{eq:prelimitN} and \eqref{eq:overlineXV} is underlined by the following \textit{propagation of chaos} result, whose proof is omitted since it is exactly the same as the one of \cite[Theorem 3.9]{Ascione20236965}.
\begin{thm}\label{thm:propchaos}
	Assume that ${\rm Law}(X_0^j,V_0^j)={\rm Law}(\overline{X},\overline{V})$ for all  $j \in \N$ and $(\overline{Z}^j)_{j \in \N}=(\overline{X}^j,\overline{Y}^j)_{j \in \N}$ be a sequence of i.i.d. copies of the solution $\overline{Z}=(\overline{X},\overline{Y})$ of \eqref{eq:overlineXV}. Then it holds
	\begin{equation*}
		\E\left[\max_{1 \le n \le N}\max_{0 \le t \le T}|Z^n_N(t)-\overline{Z}^n(t)|+\max_{0 \le t \le T}|\bm{H}_N(t)-\overline{\bm{H}}(t)|\right]\le P_{d,p}(N),
	\end{equation*}
	where $Z^n_N=(X^n_N,V^n_N)$ for all $n=1,\dots,N$, $H^n_N=(Y^n_N,W^n_N)$ for all $n=1,\dots,m$, $\bm{H}_N=(H^n_N)_{n=1,\dots,m}$ and $P_{d,p}$ is a function depending only on $d$ and $p$ such that $\lim_{N \to \infty}P_{d,p}(N)=0$. Furthermore, it holds
		 \begin{equation}\label{LLN}
		\lim_{N \to \infty} \E\left[\sup_{0 \le t \le T}\W_1(\mu^N_t,\overline{\mu}_t)\right]=0, \qquad \mbox{ where } \qquad 		\mu^N_t:=\frac{1}{N}\sum_{j=1}^{N}\delta_{(X_N^n(t),V_N^n(t))}.
	\end{equation}
%
\end{thm}
Hence, we also have, in \eqref{LLN}, a \textit{Law of Large Numbers} which guarantees that the empirical measure $\bm{\mu}^N$ of the followers in multiagent system \eqref{eq:prelimit1} converges towards the solution of mean field limit PDE-ODE system \eqref{PDEODE}. Furthermore, with the notation introduced in Theorem \ref{thm:propchaos}, we can rewrite \eqref{eq:prelimitN} as follows:
\begin{equation}\label{eq:prelimitN2}
	\begin{cases}
		d\bm{X}_N(t)=\bm{V}_N(t)dt & t \in [0,T],\\
		\displaystyle \begin{split}
			d\bm{V}_N(t)&=\left(\mathfrak{v}[t,\bm{\mu}^N](Z_N^j(t))+\mathfrak{w}[t,\bm{H}_N](Z_N^j(t))\right)_{j=1,\dots,N}dt\\
			&\quad+\sqrt{2\sigma}d\bm{B}(t)
		\end{split} &  t \in [0,T],\\
		\dot{\bm{Y}}_N(t)=\bm{W}_N(t)=F[t,\bm{\mu}^N](\bm{H}_N)+\bm{u}(t,\bm{\mu}^N) & t \in [0,T],\\
		\bm{X}_N(0)=(X_0^j)_{j=1,\dots,N},\  \bm{V}_N(0)=(V_0^j)_{j=1,\dots,N},\\ 
		\bm{Y}_N(0)=\bm{Y}_0 .
	\end{cases}
\end{equation}

Now assume that we wanted to control \eqref{eq:prelimitN2}. Precisely, for any $N \in \N$ we consider the functional $\F^N:\cA_{\rm EV} \to \R$ defined as
\begin{equation*}
	\F^N[\bm{u}]=\E\left[\int_0^T \cL(t,\bm{\mu}^N,\bm{Y}^N)\, dt+\int_0^T \Psi\left(\bm{u}(t,\bm{\mu}^N)\right)\right]
\end{equation*}
under assumptions $(\F_0)$, $(\F_1)$, $(\F_2)$ and $(\F_3)$. We are interested in the following optimal control problem:

\begin{ProblemSpecBox}{Problem $\mbox{MA}_N$}{$\mbox{MA}_N$}\label{prob:MAN}
	Find $\bm{u}^N \in \cA_{\rm EV}$ such that
	\begin{equation*}
		\F^N[\bm{u}^N]=\min_{\bm{u} \in \cA_{\rm EV}}\F^N[\bm{u}].
	\end{equation*}
\end{ProblemSpecBox}

To show that Problem \ref{prob:MAN} is well-posed, one needs some preliminaries. First of all, let us define
\begin{equation*}
	\cK_n:=\left\{\bm{\mu} \in C([0,T];\W_1(\R^{2d})): \ \overline{M}_p(T;\bm{\mu})+\sup_{\substack{0 \le s,t \le T \\ s \not = t}}\frac{\W_1(\mu_s,\mu_t)}{|t-s|^{\frac{1}{4}}}\le n\right\}.
\end{equation*}
The fact that $\cK_n \subset C([0,T];\W_1(\R^{2d}))$ is compact follows as in Lemma \ref{lem:PDEODEWell3}. Then, we need the following stability result.
\begin{lem}\label{lem:stabilN}
	Let $\{\bm{u}^j\}_{j \in \N}\subset \cA_{\rm EV}$ and $\bm{u}^\star \in \cA_{\rm EV}$ such that
	\begin{equation*}
		\lim_{j \to \infty}\int_0^t \bm{u}^j(s,\bm{\mu})\, ds=\int_0^t \bm{u}(s,\bm{\mu})\, ds \ \forall t \in [0,T], \ \forall \bm{\mu} \in \bigcup_{n \in \N}\cK_n.
	\end{equation*}
	Denote by $\bm{\mu}^{N,j}$ and $\bm{\mu}^{N,\star}$ the empirical measures of the followers in \eqref{eq:prelimitN2} with controls respectively $\bm{u}^j$ and $\bm{u}^\star$ and by $\bm{H}_N^j$ and $\bm{H}^\star$ the dynamics of the respective leaders. Then
	\begin{equation}\label{eq:Nstab}
	\lim_{j \to \infty}\E\left[\sup_{0 \le t \le T}\W_1(\bm{\mu}^{N,j}_t,\bm{\mu}^{N,\star}_t)+\sup_{0 \le t \le T}|\bm{H}_N^j(t)-\bm{H}_N^\star(t)|\right]=0.
	\end{equation}
\end{lem}
\begin{proof}
Arguing as in Lemma \ref{lem:PDEODEWell1}, we have
	\begin{align}\label{eq:bound1Rj}
		\sup_{0 \le s \le t}\left|\bm{Y}_N^j(s)-\bm{Y}_N^\star(s)\right| \le \mathcal{R}^j(t)e^{(L_F+1)T},
\end{align}
where this time
\begin{align}
	\mathcal{R}^j(t)&:=\sup_{0 \le s \le t}\left|\int_0^s (\bm{u}^j(\tau,\bm{\mu}^{N,j})-\bm{u}^\star(\tau,\bm{\mu}^{N,\star}))\, ds\right|\\
	&\le \int_0^t |\bm{u}^j(s,\bm{\mu}^{N,j})-\bm{u}^j(s,\bm{\mu}^{N,\star})|\, ds+\sup_{0 \le s \le t}\left|\int_0^s (\bm{u}^j(\tau,\bm{\mu}^{N,\star})-\bm{u}^\star(\tau,\bm{\mu}^{N,\star}))\, ds\right|\\
	&\le L_{\bm{\mu}}\int_0^t\W_1(\mu_s^{N,j},\mu^{N,\star}_s)\, ds+\sup_{0 \le s \le t}\left|\int_0^s (\bm{u}^j(\tau,\bm{\mu}^{N,\star})-\bm{u}^\star(\tau,\bm{\mu}^{N,\star}))\, ds\right|\\
	&=:L_{\bm{\mu}}\int_0^t\W_1(\mu_s^{N,j},\mu^{N,\star}_s)\, ds+\mathcal{R}_1^j(t). \label{eq:bound2Rj}
\end{align}
Furthermore, notice that, by using $(\mathfrak{w}_2)$, \eqref{eq:Lipschitzv} and \eqref{eq:preB3},
\begin{multline*}
	|\bm{Z}^j_N(t)-\bm{Z}^\star_N(t)| \le C\left( \int_0^t\W_1(\bm{\mu}^{N,j}_s,\bm{\mu}^{N,\star}_s)\, ds\right. \\\left.
	+\int_0^t|\bm{Z}^j_N(s)-\bm{Z}^\star_N(s)|\, ds+\sup_{0 \le s \le t}|\bm{H}_N^j(s)-\bm{H}_N^\star(s)|\right)
\end{multline*}
that, by Gr\"onwall's inequality, \eqref{eq:bound1Rj} and \eqref{eq:bound2Rj}, leads to
\begin{align*}
	\W_1(\bm{\mu}^{N,j}_t,\bm{\mu}^{N,\star}_t) \le |\bm{Z}^j_N(t)-\bm{Z}^\star_N(t)| \le C\left( \int_0^t \W_1(\bm{\mu}^{N,j}_s,\bm{\mu}^{N,\star}_s)\, ds+\mathcal{R}_1^j(t)\right).
\end{align*}
Using again Gr\"onwall's inequality, combining it with \eqref{eq:bound1Rj} and taking the supremum and the expectation, we finally achieve
\begin{equation*}
\E\left[\sup_{0 \le t \le T}\W_1(\bm{\mu}^{N,j}_t,\bm{\mu}^{N,\star}_t)+\sup_{0 \le t \le T}|\bm{Y}_N^j(t)-\bm{Y}_N^\star(t)|\right] \le C\E\left[\mathcal{R}_1^j(T)\right].	
\end{equation*}
Once we show that $\bm{\mu}^{N,\star}$ is H\"older-continuous, the fact that $\lim_{j \to \infty}\mathcal{R}_1^j(T)=0$ follows by the same arguments as in Theorem \ref{thm:PDEODEWell}. On the other hand, by \eqref{eq:sublin} and Proposition \ref{prop:Young}, we know that we can use the dominated convergence theorem to deduce 
\begin{equation*}
	\lim_{j \to \infty}\E\left[\sup_{0 \le t \le T}\W_1(\bm{\mu}^{N,j}_t,\bm{\mu}^{N,\star}_t)+\sup_{0 \le t \le T}|\bm{Y}_N^j(t)-\bm{Y}_N^\star(t)|\right]=0.	
\end{equation*}
Finally, \eqref{eq:Nstab} follows by observing that
\begin{equation*}
	\E\left[\sup_{0 \le t \le T}|\bm{W}_N^j(t)-\bm{W}_N^\star(t)|\right] \le C\E\left[\sup_{0 \le t \le T}\W_1(\bm{\mu}^{N,j}_t,\bm{\mu}^{N,\star}_t)+\sup_{0 \le t \le T}|\bm{Y}_N^j(t)-\bm{Y}_N^\star(t)|\right]
\end{equation*}
for some constant $C>0$.

Hence, it only remains to prove the H\"older continuity of $\bm{\mu}^{N,\star}$. To do this, let $0 \le s \le t \le T$ and recall that
\begin{equation}\label{eq:bounded}
	|\bm{V}_N^\star(t)| \le |\bm{V}_N(0)|+2M_{\rm Ker}T+\sqrt{2\sigma}M_{\bm{B}}(T), \qquad \mbox{ where } \qquad 
	M_{\bm{B}}(t)=\sup_{0 \le s \le t}|B(s)|.
\end{equation}
and
\begin{equation*}
	|\bm{V}_N^\star(t)-\bm{V}_N^\star(s)| \le 2M_{\rm Ker}|t-s|+\sqrt{2\sigma}L_{\bm{B}} |t-s|^{\frac{1}{4}}, \quad \mbox{ where } \quad 
	L_{\bm{B}}:=\sup_{0 \le s < t \le T}\frac{|\bm{B}(t)-\bm{B}(s)|}{|t-s|^\gamma},
\end{equation*}
so that
\begin{equation*}
	|\bm{Z}_N^\star(t)-\bm{Z}_N^\star(s)| \le (4M_{\rm Ker}+\sqrt{2\sigma}M_{\bm{B}}(T))|t-s|+\sqrt{2\sigma}L_{\bm{B}} |t-s|^{\gamma}.
\end{equation*}
Taking the expectation and recalling that $\E[M_{\bm{B}}(T)],\E[L_{\bm{B}}]<\infty$ we get that $\bm{\mu}^{N,\star} \in \bigcup_{n \in \N}\cK_n$ and we conclude the proof.
\end{proof}
Now we are ready to prove the existence of a solution for Problem \ref{prob:MAN}. Since the proof is really similar to Theorem \ref{thm:EV}, we only underline the main differences.
\begin{thm}
	Under $(\F_0)$, $(\F_1)$, $(\F_2)$ and $(\F_3)$, Problem \ref{prob:MAN} admits at least a solution.
\end{thm}
\begin{proof}
	First, we recall that $\W_p(\R^{2d})$ for $p>1$ is a hemicompact dense subset of $\W_1(\R^{2d})$. Precisely, for all $n \in \N$ define
	\begin{equation*}
		\cK^\sharp_n:=\left\{\mu \in \W_1(\R^{2d}): \ M_p(\mu)\le n\right\}
	\end{equation*}
	and observe that $\bigcup_{n \in \N}\cK^\sharp_n=\W_p(\R^{2d})$. The fact that $\cK^\sharp_n$ are compact subsets of $\W_1(\R^{2d})$ is shown as in Lemma \ref{lem:PDEODEWell3}. Consider a minimizing sequence $\{\bm{u}^{j}\}_{j \in \N} \subset \cA_{\rm EV}$ for $\F^{N}$. Arguing as in Theorem \ref{thm:EV} and using a diagonal argument we know that there exists a function $\bm{u}^{\star}:[0,T]\times \bigcup_{n \in \N}\cK_n \to \R^m$ such that, up to a non-relabelled subsequence, for all $\bm{\mu} \in \bigcup_{n \in \N}\cK_n$
	\begin{equation}\label{eq:convergence}
		\lim_{n \to \infty}\frac{1}{|E|}\left|\int_E \left(\bm{u}^{n}(t,\bm{\mu})- \bm{u}^{\star}(t,\bm{\mu})\right)\, dt\right|=0
	\end{equation}
	and, for all $j=1,\dots,md$, $t \in [0,T]$ and $\bm{\mu},\bm{\nu} \in \bigcup_{n \in \N} \cK_n$,
	\begin{equation}\label{eq:Lipschitzustarj}
		|u^\star_j(t,\bm{\mu})-u^\star_j(t,\bm{\nu})| \le \frac{L_{\bm{u}}}{md}\W_1(\mu_t,\nu_t).
	\end{equation}
	The fact that $\bm{u}^\star \in (\cA_{\rm EV})$ follows as in Theorem \ref{thm:EV}. Furthermore, by \eqref{eq:convergence}, denoting by $(\bm{Z}^{N,j},\bm{Y}^j_N)$ and $(\bm{Z}^{N,\star},\bm{Y}^j_N)$ the solutions of \eqref{eq:prelimitN2} with controls $\bm{u}^j$ and $\bm{u}^\star$ repsectively and by $\bm{\mu}^{N,j}$ and $\bm{\mu}^{N,\star}$ the corresponding empirical measures of the followers, we have \eqref{eq:Nstab} and then, up to a subsequence, almost surely
	\begin{equation*}
		\lim_{j \to \infty}\sup_{0 \le t \le T}\W_1(\mu^{N,j}_t,\mu^{N,\star}_t)+\sup_{0 \le t \le T}|\bm{H}_N^j(t)-\bm{H}_N^\star(t)|=0.
	\end{equation*}
	Hence, by $(\F_1)$, $(\F_2)$ and the dominated convergence theorem, we have
	\begin{equation*}
		\lim_{j \to +\infty}\E\left[\int_0^T \cL(t,\bm{\mu}^{N,j},\bm{Y}_N^j)\, dt\right]=\E\left[\int_0^T \cL(t,\bm{\mu}^{N,\star},\bm{Y}_N^\star)\, dt\right].
	\end{equation*}
	On the other hand, the same arguments as in Theorem \ref{main2} and a simple application of Fatou's Lemma lead to
	\begin{equation*}
		\E\left[\int_0^T \Psi(\bm{u}^{N,\star}(t,\bm{\mu}^{N,\star}))\, dt\right] \le \liminf_{j \to \infty}\E\left[\int_0^T \Psi(\bm{u}^{N,j}(t,\bm{\mu}^{N,j}))\, dt\right].
	\end{equation*}
		The remainder of the proof follows as in Theorem \ref{main2}.
\end{proof}
However, it is really difficult to provide (even numerically) a solution of Problem \ref{prob:MAN} for a big value of $N$. For such a reason, one tries to use a mean field limit approach, justified by Theorem \ref{thm:propchaos}, to provide a suitable approximation of the optimal control. Indeed, on \eqref{eq:MKVSDEaux2}, one can consider the optimal control Problem \ref{prob:3}, whose solution exists by Theorem \ref{thm:EV} and is much more affordable, at least from the numerical point of view, since it involves a system of a PDE and $2md$ ODEs, as in \eqref{PDEODE}, in place of a system of $2(N+m)d$ SDEs as in \eqref{eq:prelimitN2}. However, it remains to show that a solution of Problem \ref{prob:3} actually approximates, in some sense, a solution of Problem \ref{prob:MAN}. We will do this for a slightly different problem. First we consider the following assumptions, that are more restrictive that the ones we considered in Section \ref{subsec:CP}.
\begin{ProblemSpecBox}{Assumptions on $\F^N$: $(\F^N)$}{$(\F^N)$}\label{ass:Fprime2}
	\begin{itemize}
		\item[$(\mathcal{F}_0)$] $\cL:[0,T]\times C([0,T];\W_p(\R^{2d})) \times C([0,T];\R^{m})$ and $\Psi:\R^{m} \to \R$ are measurable and bounded from below;
		\item[$(\mathcal{F}_1^N)$] There exists a non-negative increasing concave function $\varpi_{\cL}:\R_0^+ \to \R_0^+$ such that
		\begin{equation*}
			\left|\cL(t,\bm{\mu}^1,\bm{Y}^1)-\cL(t,\bm{\mu}^2,\bm{Y}^2)\right| \le \varpi_{\cL}\left(\sup_{0 \le s \le t}\W_1(\mu_s^1,\mu_s^2)+\sup_{0 \le s \le t}|\bm{Y}^1(s)-\bm{Y}^2(s)|\right)
		\end{equation*}
		\item[$(\F_2)$] There exists a function $M_{\cL}^\prime \in L^1(0,T)$ such that for a.a. $t \in [0,T]$, for all $\bm{\mu} \in C([0,T];\W_p(\R^{2d}))$ and for all $\bm{Y} \in C([0,T];\R^{m})$,
		\begin{equation*}
			|\cL(t,\bm{\mu},\bm{Y})| \le M_{\cL}^\prime(t).
		\end{equation*} 
		\item[$(\F_3)$] $\Psi:\R^{m} \to \R$ is convex.
	\end{itemize}
\end{ProblemSpecBox}
Furthermore, we have to further reduce the set of controls we are working on. Precisely, we define the class of controls $\cA_{\rm EVB} \subset \cA_{\rm EV}$ satisfying $(\cA_0)$, $(\cA_2)$, $(\cA_{\rm EV})$ and
\begin{itemize}
	\item[$(\cA_{B})$] For all $t \in [0,T]$ and $\bm{\mu} \in C([0,T];\W_1(\R^{2d}))$ it holds
	\begin{equation*}
		|\bm{u}(t,\bm{\mu})| \le M_{\bm{u}}.
	\end{equation*}
\end{itemize}
Assumption $(\cA_B)$ clearly implies $(\cA_1)$, but it is much more restrictive. We consider then the following optimal control problems:

\newpage

\begin{ProblemSpecBox}{Problem $3_B$}{$3_B$}\label{prob:3B}
	Find $\bm{u}^\star \in \cA_{\rm EVB}$ such that
	\begin{equation*}
		\F[\bm{u}^\star]=\min_{\bm{u} \in \cA_{\rm EVB}}\F[\bm{u}].
	\end{equation*}
\end{ProblemSpecBox}

\begin{ProblemSpecBox}{Problem $\mbox{MAB}_N$}{$\mbox{MAB}_N$}\label{prob:MABN}
	Find $\bm{u}^\star \in \cA_{\rm EVB}$ such that
	\begin{equation*}
		\F^N[\bm{u}^\star]=\min_{\bm{u} \in \cA_{\rm EVB}}\F^N[\bm{u}].
	\end{equation*}
\end{ProblemSpecBox}

Notice that in both cases the optimal solution $\bm{u}^\star$ is obtained by considering a minimizing sequence $\{\bm{u}^j\}_{j \in \N}\subset \cA_{\rm EVB}$ such that for any measurable $E \subset [0,T]$ and $\bm{\mu} \in \cK$ in case of Problem \ref{prob:3B} or $\bm{\mu} \in \bigcup_{n \in \N}\cK_n$ in case of Problem \ref{prob:MABN}
\begin{equation*}
	\lim_{j \to \infty}\frac{1}{|E|}\int_{E} \bm{u}^j(t,\bm{\mu})\, dt=\frac{1}{|E|}\int_{E} \bm{u}^\star(t,\bm{\mu})\, dt.
\end{equation*}
Theorem \ref{thm:EV} already guarantees that $\bm{u}^\star \in \cA_{\rm EV}$. Furthermore, for all $t \in E_{\sf Leb}$ and $h>0$ small enough
\begin{equation*}
	\frac{1}{h}\left|\int_{t}^{t+h} \bm{u}^j(t,\bm{\mu})\, dt\right|\le M_{\bm{u}}.
\end{equation*}
Hence, taking the limit as $j \to \infty$ and then $h \downarrow 0$, one show that $\bm{u}^\star \in \cA_{EVB}$. As a consequence, we have the following result:
\begin{thm}
	There exists at least a solution for both Problems \ref{prob:3B} and \ref{prob:MABN}.
\end{thm}
Now, we can prove the following result:
\begin{thm}\label{thm:Gammalim}
	Let $\{\bm{u}^N\}_{N \in \N}\subset \cA_{\rm EVB}$ and $\bm{u}^\star \in \cA_{\rm EVB}$ be such that for all $t \in [0,T]$ and $\bm{\mu} \in \bigcup_{n \in \N}\cK_n$
	\begin{equation*}
		\lim_{N \to \infty} \int_{0}^{t}\bm{u}^N(s,\bm{\mu})\, ds=\int_{0}^{t}\bm{u}^\star(s,\bm{\mu})\, ds.
	\end{equation*}
	Then
	\begin{equation*}
		\F[\bm{u}^\star] \le \liminf_{N \to \infty}\F^N[\bm{u}^N].
	\end{equation*}
	Furthermore for all $\bm{u} \in \cA_{\rm EVB}$ it holds
	\begin{equation*}
		\lim_{N \to \infty}\F^N[\bm{u}]=\F[\bm{u}].
	\end{equation*}
	Finally, it holds
	\begin{equation*}
		\lim_{N \to \infty}\min_{\bm{u} \in \cA_{\rm EVB}}\F^N[\bm{u}]=\min_{\bm{u} \in \cA_{\rm EVB}}\F[\bm{u}].
	\end{equation*}
\end{thm}
\begin{proof}
	Let $\{\bm{u}^N\}_{N \in \N}\subset \cA_{\rm EVB}$ and $\bm{u}^\star \in \cA_{\rm EVB}$ as in the statement. Denote by $(\bm{Z}^N,\bm{H}_N)$ the solutions of \eqref{eq:prelimitN2} with controls $\bm{u}^N$, by $\bm{\mu}^N$ the empirical measures of the followers $\bm{Z}^{N}$ and by $(\overline{\bm{\mu}}^N,\overline{\bm{Y}}^N)$ the solutions of \eqref{PDEODE} with controls $\bm{u}^N$. Arguing as in the proof of \cite[Theorem 4.6]{Ascione20236965} and using Theorem \ref{thm:propchaos}, we have
	\begin{equation}\label{eq:L1}
		\lim_{N \to \infty}\left|\E\left[\int_0^T \cL\left(t,\bm{\mu}^{N},\bm{Y}_N\right)\, dt\right]-\int_0^T \cL\left(t,\overline{\bm{\mu}}^{N},\overline{\bm{Y}}_N\right)\, dt\right|=0.
	\end{equation}
	On the other hand, by Theorem \ref{thm:PDEODEWell} and Assumptions $(\F_1^N)$ (that implies $(\F_1)$) and $(\F_2)$, we have
	\begin{equation}\label{eq:L2}
		\lim_{N \to \infty}\left|\int_0^T \cL\left(t,\overline{\bm{\mu}}^{N},\overline{\bm{Y}}_N\right)\, dt-\int_0^T \cL\left(t,\overline{\bm{\mu}}^{\star},\overline{\bm{Y}}_\star\right)\right|=0,
	\end{equation}
	where $(\overline{\bm{\mu}}^\star,\overline{\bm{Y}}_\star)$ is the solution of \eqref{PDEODE} with control $\bm{u}^\star$. Concerning $\Psi$, the same arguments as in the proof of Theorem \ref{main2}, together with Theorem \ref{thm:propchaos}, and Fatou's lemma lead to
	\begin{equation*}
		\int_0^T \Psi(\bm{u}^\star(t,\bm{\mu}^\star))\, dt \le \liminf_{N \to \infty} \E\left[\int_0^T \Psi(\bm{u}^N(t,\bm{\mu}^N))\, dt\right].
	\end{equation*}
	This, together with \eqref{eq:L1} and \eqref{eq:L2} prove that
	\begin{equation*}
		\F[\bm{u}^\star] \le \liminf_{N \to \infty}\F^N[\bm{u}^N].
	\end{equation*}
	Now consider a control $\bm{u} \in \cA_{\rm EVB}$, denote by $(\bm{Z}^N,\bm{Y}_N)$ the solutions of \eqref{eq:prelimitN2} with controls $\bm{u}$, by $\bm{\mu}^N$ the empirical measures of the followers $\bm{Z}^{N}$ and by $(\overline{\bm{\mu}},\overline{\bm{Y}})$ the solutions of \eqref{PDEODE} with controls $\bm{u}$. Since $\Psi$ is uniformly continuous on $\{z \in \R^m: \ |z| \le M_{\bm{u}}\}$, we can consider a concave modulus of continuity $\varpi_2$ and argue as in \cite[Theorem 4.6]{Ascione20236965}, by using also $(\cA_2)$ and Theorem \ref{thm:propchaos},
	\begin{equation}\label{eq:Psi1}
		\lim_{N \to \infty}\left|\E\left[\int_0^T \Psi\left(\bm{u}(t,\bm{\mu}^{N})\right)\, dt\right]-\int_0^T \Psi\left(\bm{u}(t,\overline{\bm{\mu}})\right)\, dt\right|=0.
	\end{equation}
	Combining \eqref{eq:L1} and \eqref{eq:Psi1} we achieve
	\begin{equation*}
		\lim_{N \to \infty}\F^N[\bm{u}]=\F[\bm{u}].
	\end{equation*}
	Finally, the last statement follows as in \cite[Proposition 4.7]{Ascione20236965}.
\end{proof}
\begin{rem}
	It is worth noticing that the solution $\bm{u}^\star$ of Problem \ref{prob:3B} is an \textit{approximate solution} of Problem \ref{prob:MABN} in the sense that it can be used to provide a control whose cost is \textit{nearly minimal}. However, if $\bm{u}^N$ is a solution of Problem \ref{prob:MABN}, one cannot guarantee that $\bm{u}^\star$ approximate $\bm{u}^N$, but only that $\F^N[\bm{u}^\star]$ is near $\F^N[\bm{u}^N]$.
\end{rem}

\newpage

\newpage

\bibliographystyle{abbrv}
\bibliography{biblio} 

\end{document}